%% file: Thesis.tex
\documentclass{ut-thesis}

\usepackage{graphicx}
\usepackage{rotating}
\usepackage{amsmath,amssymb,amsthm}
\usepackage{epsfig,color,dbnsymb}
\usepackage[all]{xy}

\degree{Doctor of Philosophy}
\department{Mathematics}
\gradyear{2007}
\author{Gad Naot}
\title{The Universal $sl_2$ Link Homology Theory}

\input{Defs}

\begin{document}

\begin{preliminary} 
\maketitle

\include{Abs}

\begin{dedication}
To my parents Dvora and Dan-Zvi.
\end{dedication}

\newpage

\begin{acknowledgements}

There is much more to a PhD diploma than doing your research and writing your thesis. True, one needs to perform his research well and write a thesis (which you are about to read perhaps), but during the long and winding road that brought me to this day, I felt I did and achieved much more than that. I have explored many parts of mathematics, physics, philosophy and other disciplines, I have given talks, lectures and written papers, I have traveled the world and met many people, and  most of all I have evolved intellectually and mentally while living my life to the fullest.\\
\indent I wish to thank the University of Toronto and the department of mathematics (especially Ida Bulat the graduate coordinator) for their financial support and their efficient administration. My study experience here in Toronto was perfect. It is an amazing city, and a great university to study at.\\
\indent I wish to thank my supervisor Prof. Dror Bar-Natan. If it was not for him, I would not have been here in Toronto. I wish to thank him for allowing me to do things my way and for the relaxed and easy going atmosphere which allowed me to achieve many things I have wished for. I also thank him for his part in my financial support, and of course, last but not least, for all his involvement in my research. His help and support was a major driving force.

I thank Mikhail Khovanov and Charles Frohman for many useful comments on my research.
\\

I wish to thank all the friends I have made around the world and here in Toronto that made this time fun.
\\

I wish to thank my wonderful partner, Inmar Givoni, whom I love most fiercely. She is the source of my life.

\end{acknowledgements}

\tableofcontents 

\end{preliminary}


\include{ChapIntro}
\include{ChapClassificationQ}
\include{ChapComplexReductionQ}
\include{ChapClassificationZ}
\include{ChapComplexreductionZ}
\include{ChapTQFT}

\include{ChapCompExt}

\include{ChapComments}


\bibliographystyle{alpha}
\bibliography{Thesis}

\end{document}

%% file: Defs.tex
\newcommand{\cut}[1]{}

\theoremstyle{plain}

\newtheorem{theorem}{Theorem}
\newtheorem{proposition}{Proposition}[section]

\newtheorem{lemma}[proposition]{Lemma}
\newtheorem{corollary}[proposition]{Corollary}

\theoremstyle{definition}
\newtheorem{definition}[proposition]{Definition}

\theoremstyle{remark}

\newtheorem{remark}[proposition]{Remark}

\newlength{\standardunitlength}
\catcode`\@=11
\long\def\@makecaption#1#2{%
    \vskip 10pt
    \setbox\@tempboxa\hbox{
      \small\sf{\bfcaptionfont #1. }\ignorespaces #2}%
    \ifdim \wd\@tempboxa >\captionwidth {%
        \rightskip=\@captionmargin\leftskip=\@captionmargin
        \unhbox\@tempboxa\par}%
      \else
        \hbox to\hsize{\hfil\box\@tempboxa\hfil}%
    \fi}
\font\bfcaptionfont=cmssbx10 scaled \magstephalf
\newdimen\@captionmargin\@captionmargin=2\parindent
\newdimen\captionwidth\captionwidth=\hsize
\catcode`\@=12

\newlength{\globalparindent}
\def\calF{{\mathcal F}}
\def\calO{{\mathcal O}}
\def\calC{{\mathcal C}}
\def\calA{{\mathcal A}}
\def\calH{{\mathcal H}}

\newcommand{\Cob}{{\mathcal Cob}}

\newcommand{\Cobl}{{\mathcal Cob}_{/l}}

\newcommand{\Kh}{{\text{\it Kh}}}

\newcommand{\Kom}{\operatorname{Kom}}
\newcommand{\Komh}{\operatorname{Kom}_{/h}}
\newcommand{\Mat}{\operatorname{Mat}}
\newcommand{\Mor}{\operatorname{Mor}}

\newcommand{\eps}[2]{{\hspace{-3pt}\begin{array}{c}%
  \raisebox{-2.5pt}{\includegraphics[width=#1]{figs/#2.eps}}%
\end{array}\hspace{-3pt}}}
\newcommand{\epsg}[2]{{\hspace{-3pt}\begin{array}{c}%
  \raisebox{0pt}{\includegraphics[width=#1]{figs/#2.eps}}%
\end{array}\hspace{-3pt}}}

\def\qed{{ \hfill $\eps{8mm}{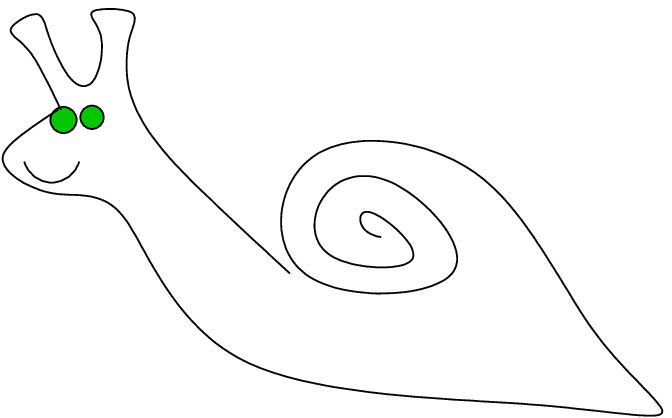}$ }}

%% file: Abs.tex
\begin{abstract} 

We explore the complex associated to a link in the geometric formalism of Khovanov's ($n=2$) link homology theory, determine its exact underlying algebraic structure and find its precise universality properties for link homology functors. We present new methods of extracting all known link homology theories directly from this universal complex, and determine its relative strength as a link invariant by specifying the amount of information held within the complex.
\\

We achieve these goals by finding a complex isomorphism which reduces the complex into one in a simpler category. We introduce few tools and methods, including surface classification modulo the 4TU/S/T relations and genus generating operators, and use them to explore the relation between the geometric complex and its underlying algebraic structure. We identify the universal topological quantum field theory (TQFT) that can be used to create link homology and find that it is ``smaller'' than what was previously reported by Khovanov. We find new homology theories that hold a controlled amount of information relative to the known ones.
\\

The universal complex is computable efficiently using our reduction theorem. This allows us to explore the phenomenological aspects of link homology theory through the eyes of the universal complex in order to explain and unify various phenomena (such as torsion and thickness). The universal theory also enables us to state results regarding specific link homology theories derived from it. The methods developed in this thesis can be combined with other known techniques (such as link homology spectral sequences) or used in the various extensions of Khovanov link homology (such as $sl(3)$ link homology).

\end{abstract}

%% file: ChapIntro.tex
\chapter{Introduction(s)}\label{chap:Intro}

Writing an introduction chapter to a PhD thesis is not easy. Doing so in mathematics is even harder. I have decided to write three introduction sections in this chapter. The first will address the general audience (non mathematicians) and mathematicians who know nothing about knot theory and wish to get a bird's eye overview (with no definitions!). I do so since I believe a mathematical text, especially a thesis, needs to be accessible to such an audience as well (at least in the ``big picture'') and thus written (in the beginning) using the simplest terms possible. The second introduction section will be the main one. It will address mathematicians who might know something about knot theory, probably heard about Khovanov homology, but know little or nothing about the construction of this homology theory (especially what is called the geometric construction). The third and last introduction section would address ``experts'' in the field and would prepare the grounds for reading this thesis. Experts are mathematicians who are familiar with the geometric construction of Khovanov homology or readers that survived the second introduction.

\section{Bird's eye overview for general audience}

Knot theory. The theory that studies knots. Yes, the things down your legs on your shoes. It is hard to imagine that knots have not been around since the age of early civilizations. Sophisticated uses of knots (such as encoding system) date from centuries ago in south America, and mystical uses of knots date even further back, thousands of years ago, in Egypt (not to mention using knots to actually tie something). Nonetheless, mankind only started looking \emph{mathematically} on such objects at the end of the 19th century. This is probably due to the lack of the proper mathematical frames, but it might be also the case that it was due to lack of interest in doing so.
\\

The first mathematical interest in knot theory came via explorations in physics. Gauss, for example, raised questions about linking of currents as part of his electro-magnetism theory. Lord Kelvin and others raised the idea that different atoms are different kind of knots created in the \emph{ether} (what used to be thought of the vacuum surrounding us), giving them required stability. Gold atom = $\eps{10mm}{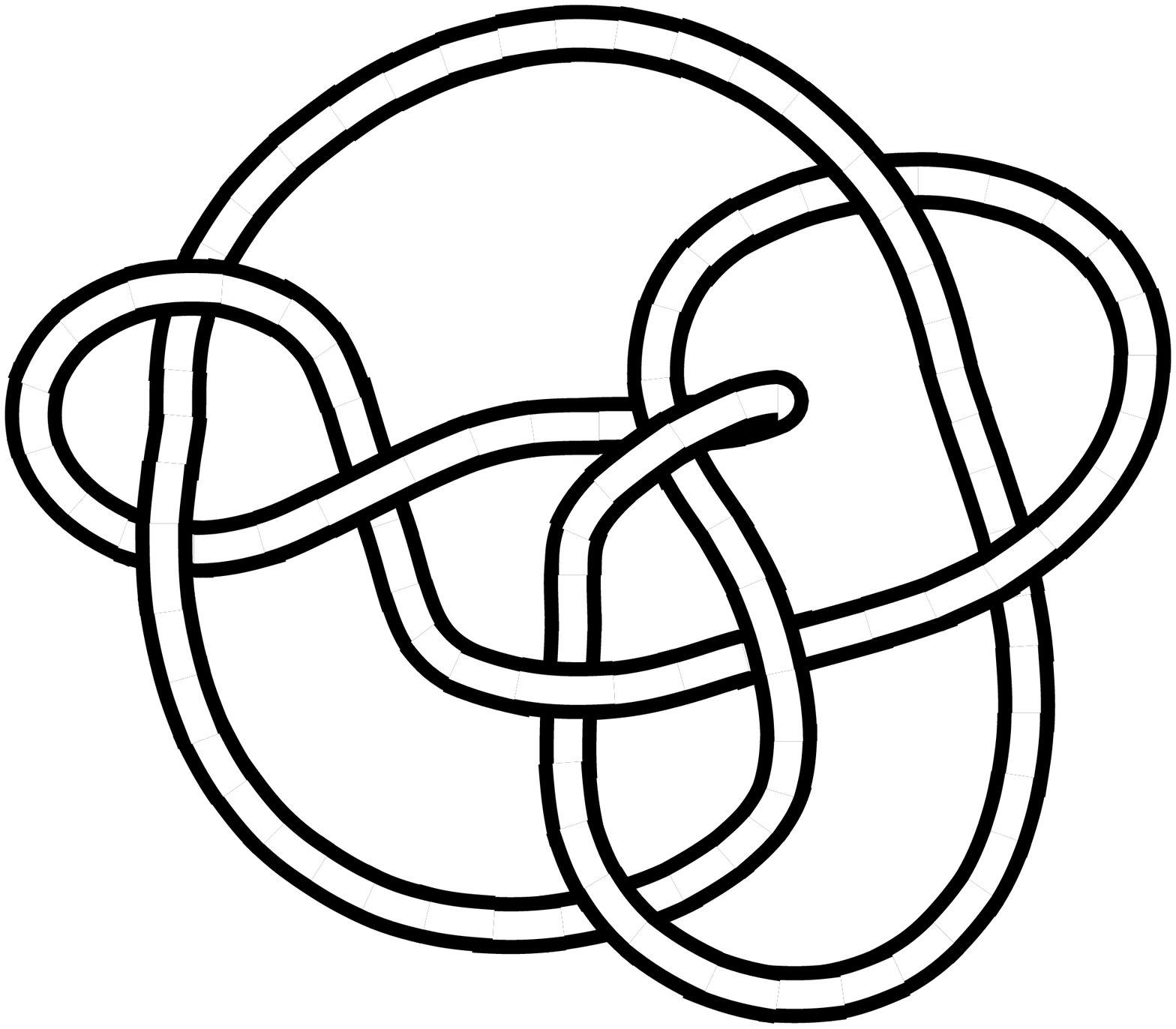}$? Interactions between atoms or molecules are changes in the knotting structure. Thus was born the first (and still dominant) question in knot theory - the question of knot classification. How do we classify all different knots?
Is $\eps{10mm}{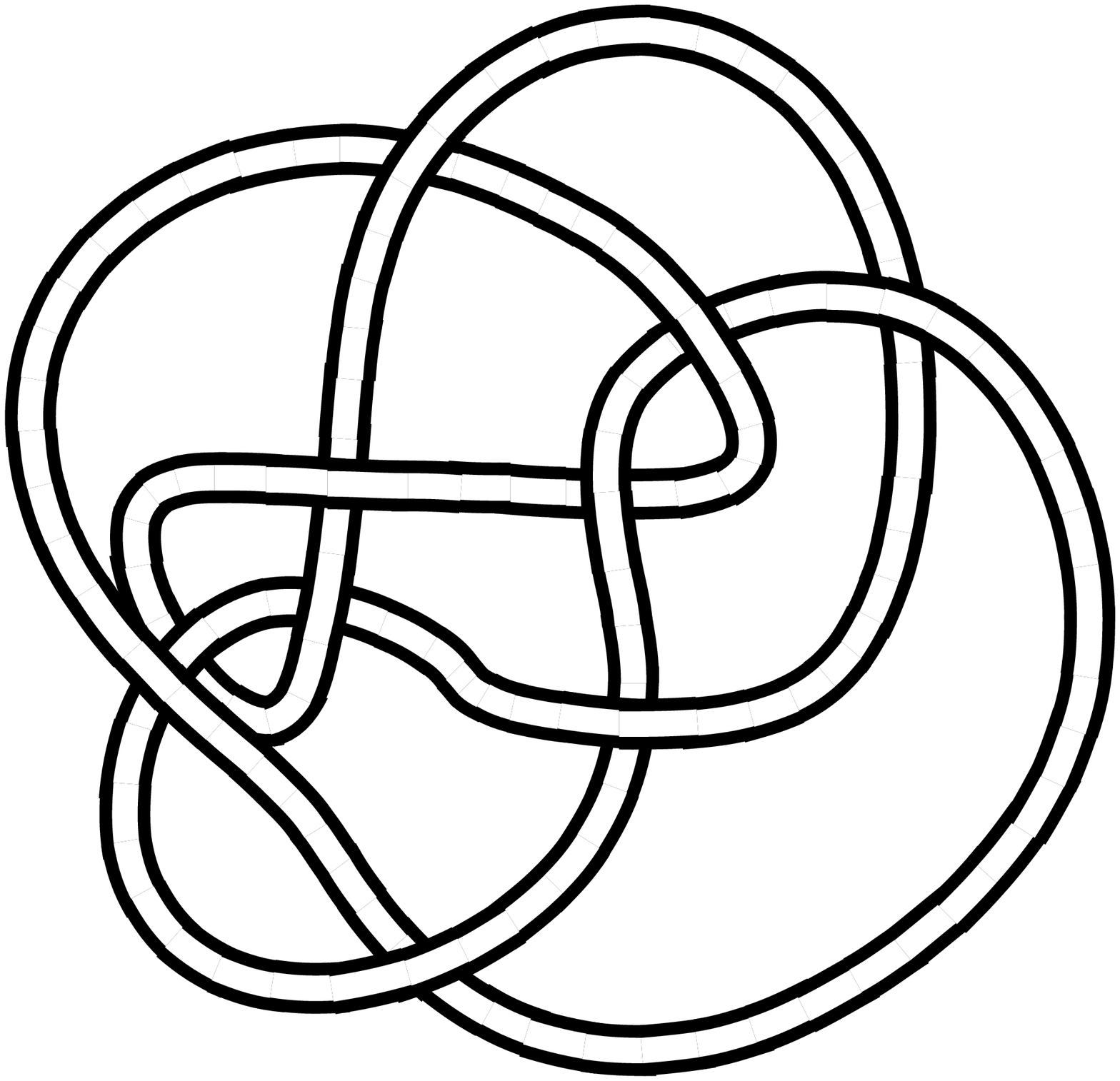}$ really different from $\eps{10mm}{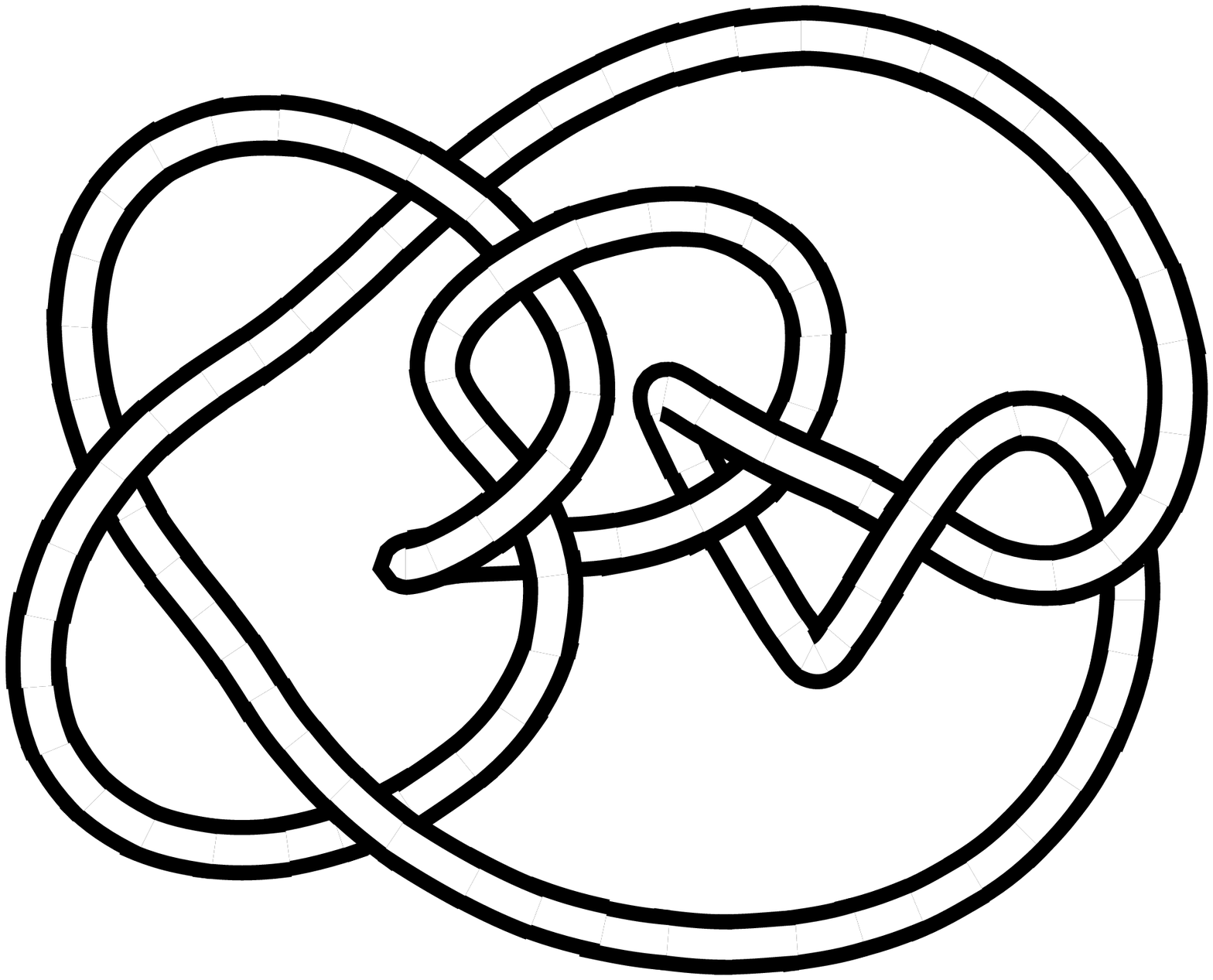}$ ?
First knot tabulations appeared around the same time (end of 19th century) giving birth to what is known today in mathematics as \emph{knot theory}. Here are all the simple knots (in some sense) \footnote{See the knot atlas at www.katals.math.toronto.edu for everything you ever wanted to know about knots classification.}: \[\eps{10mm}{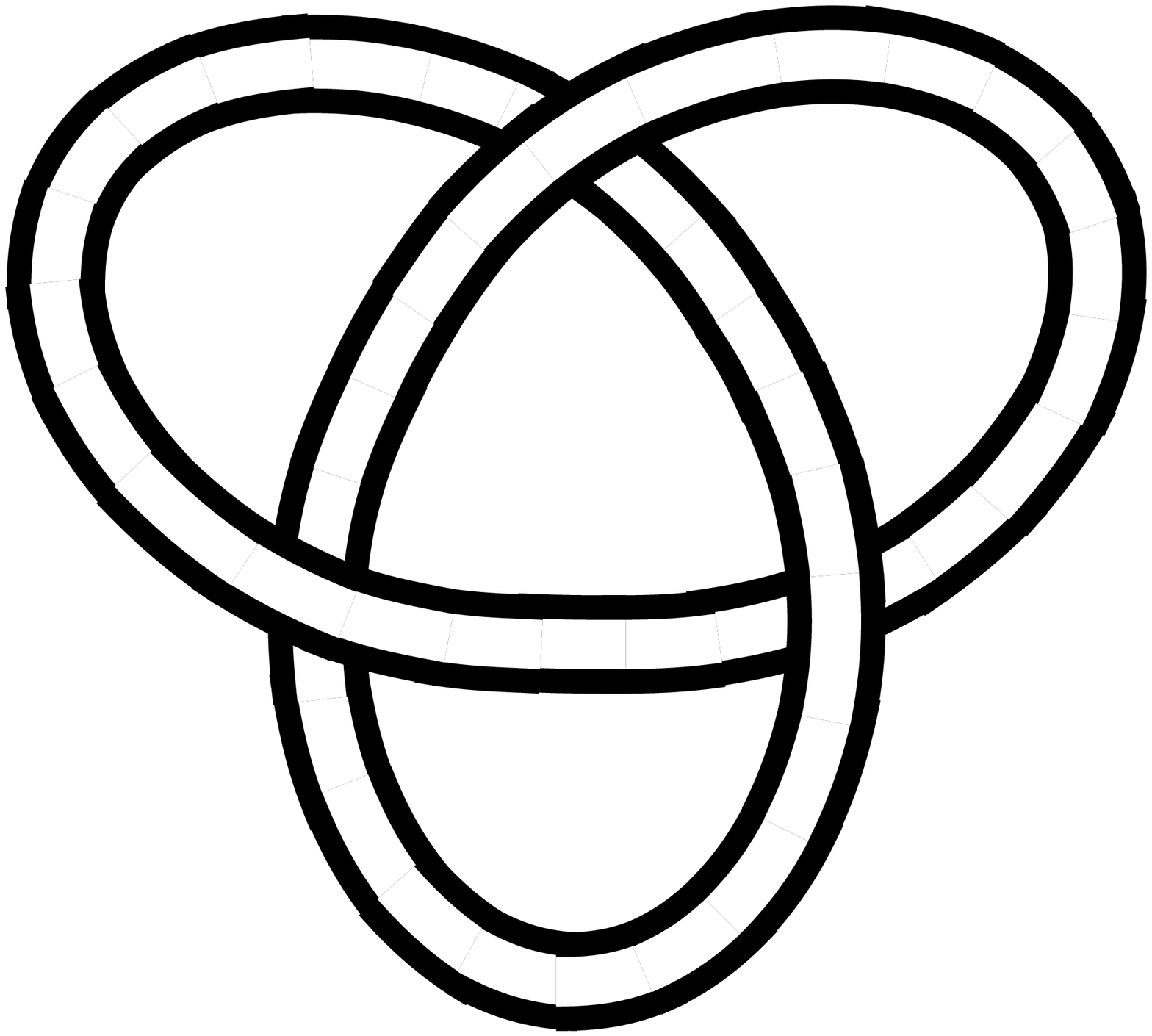}\eps{10mm}{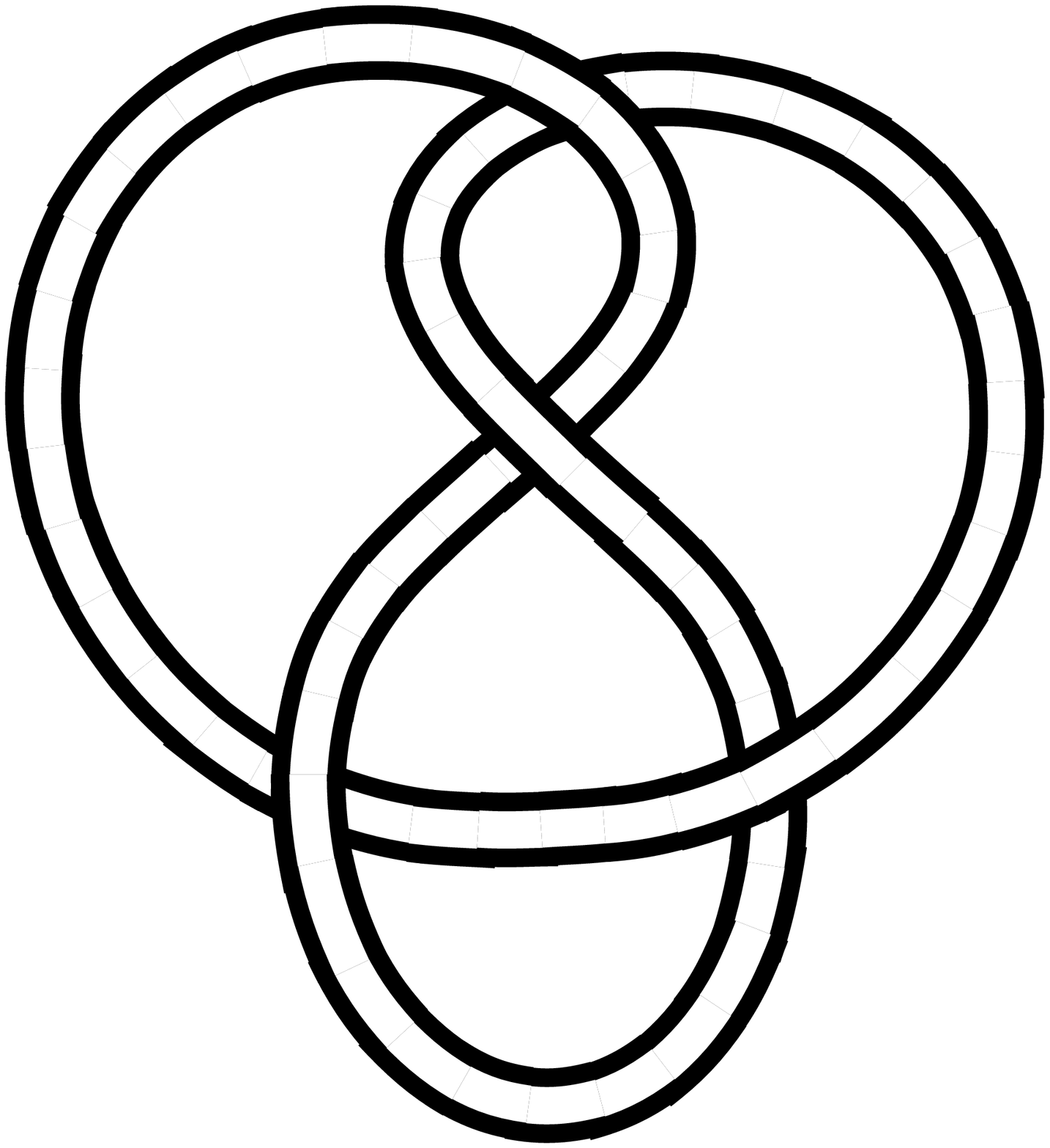}\eps{10mm}{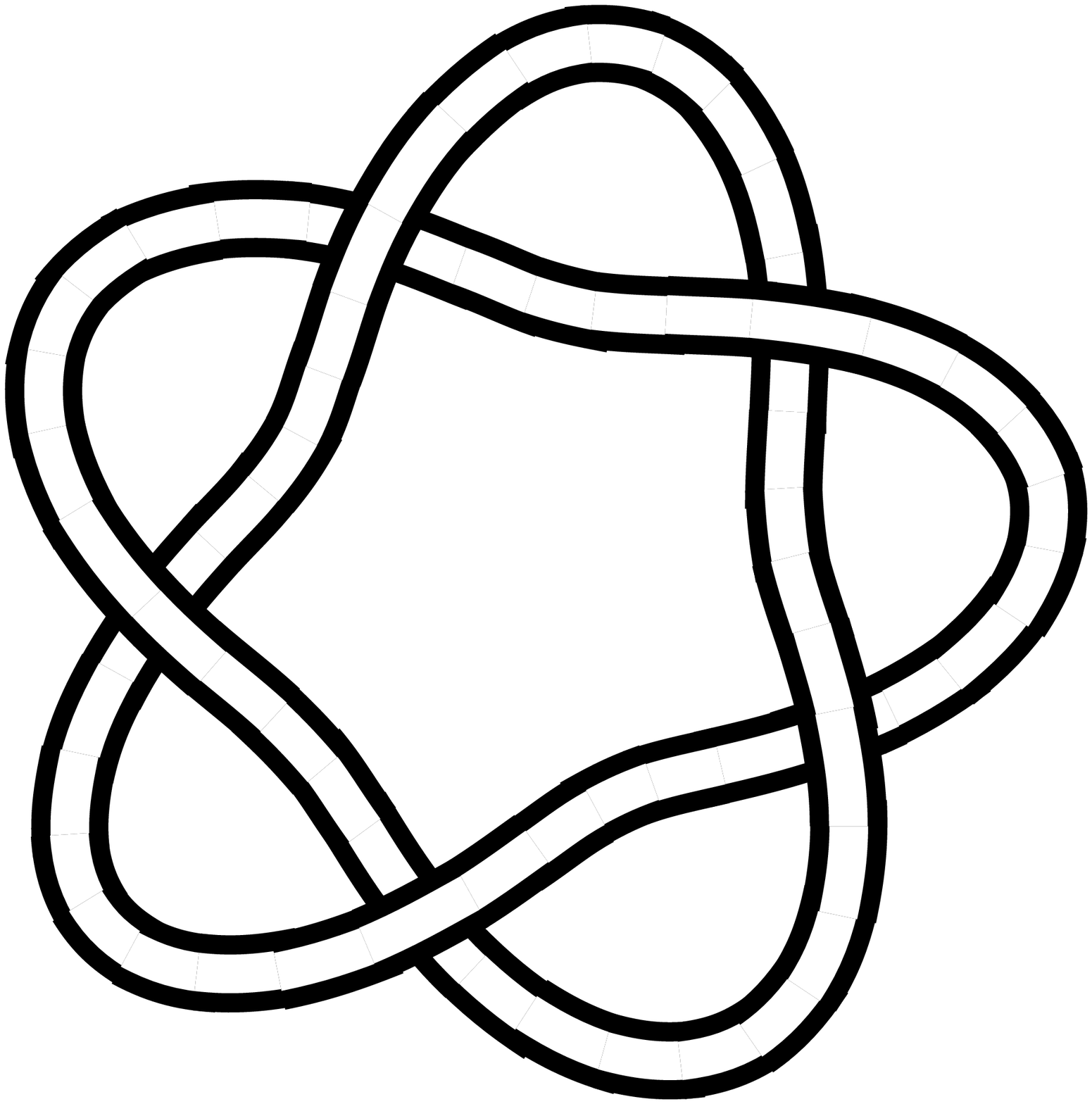}\eps{10mm}{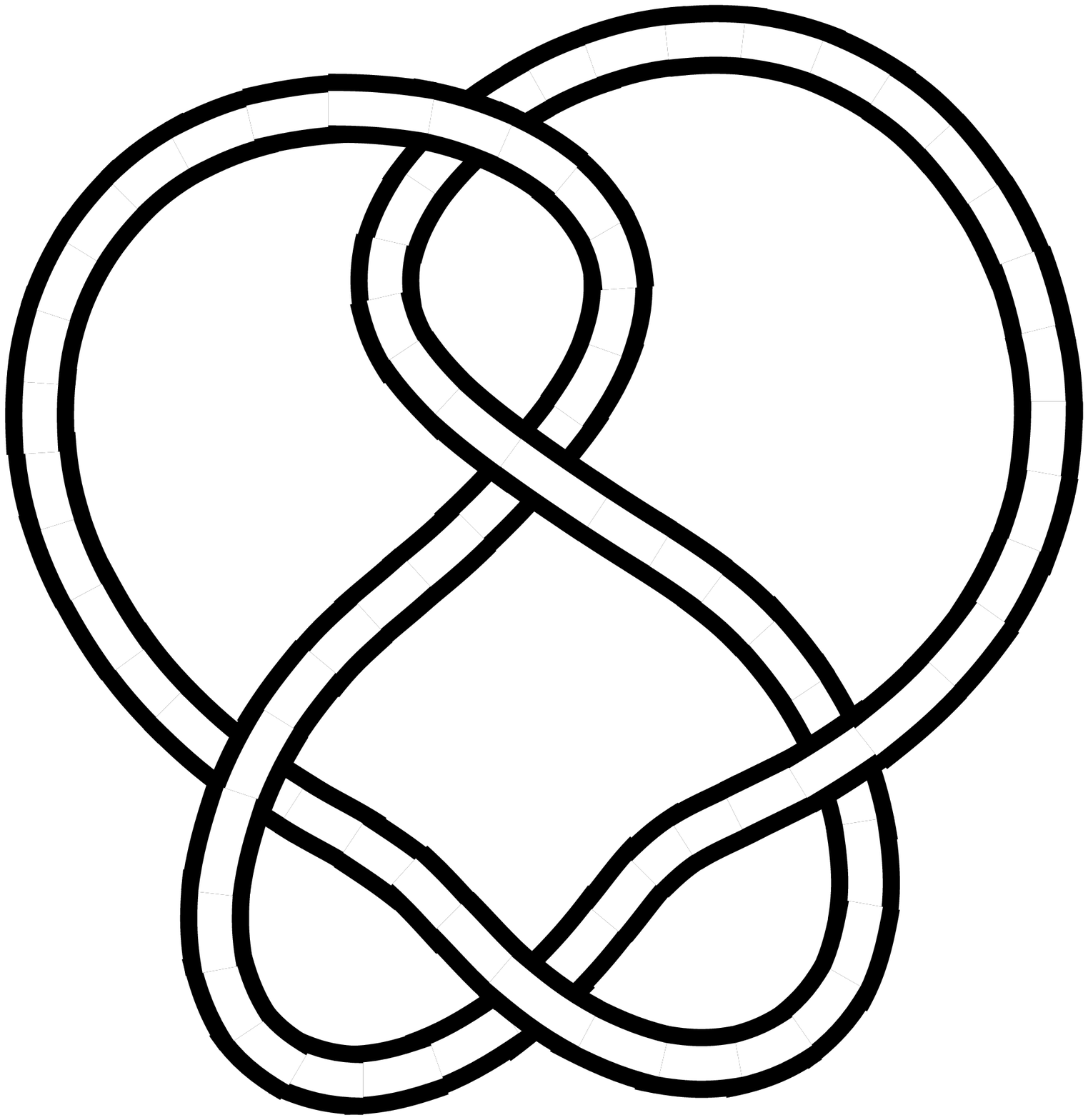}\eps{10mm}{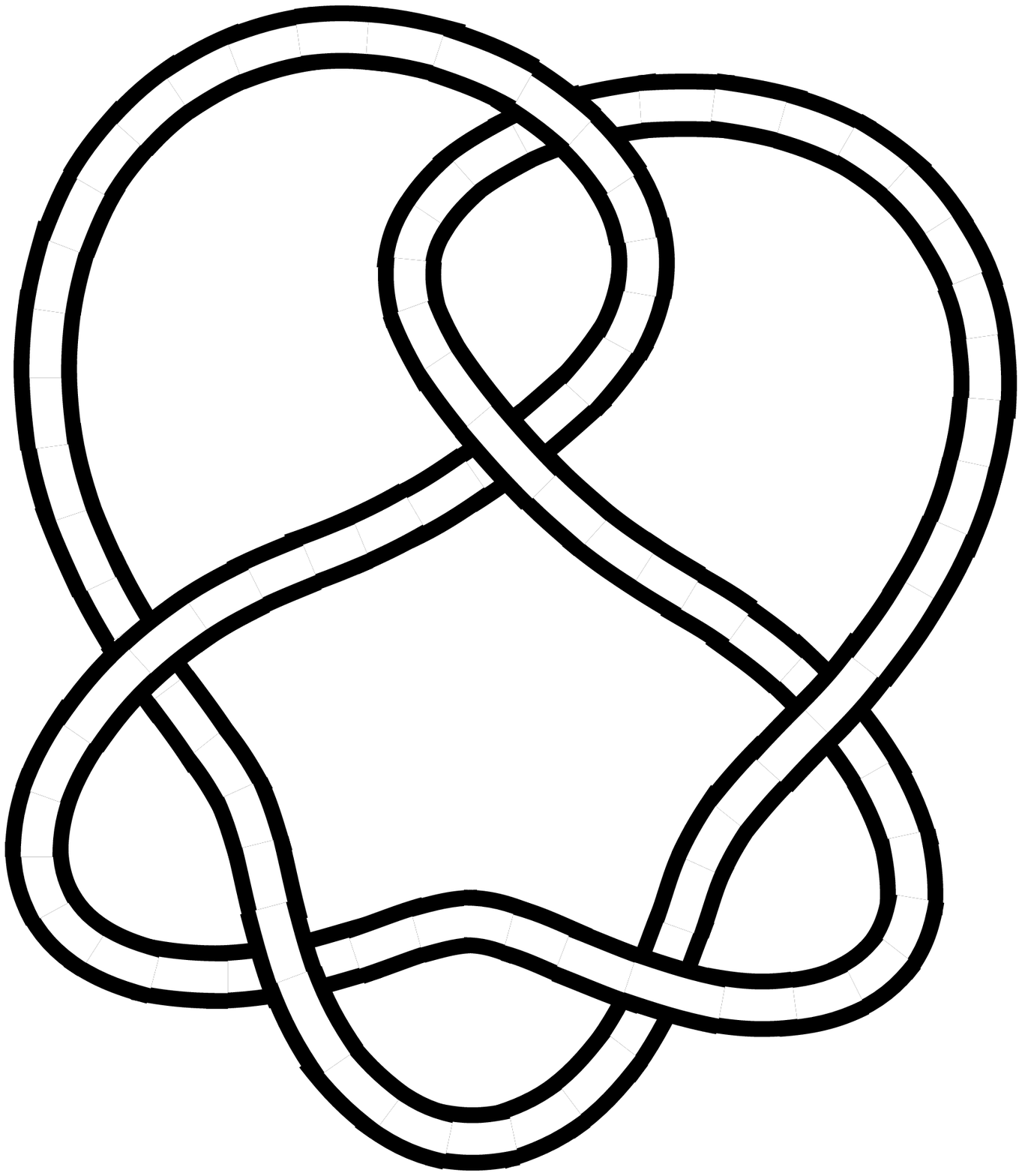} \eps{10mm}{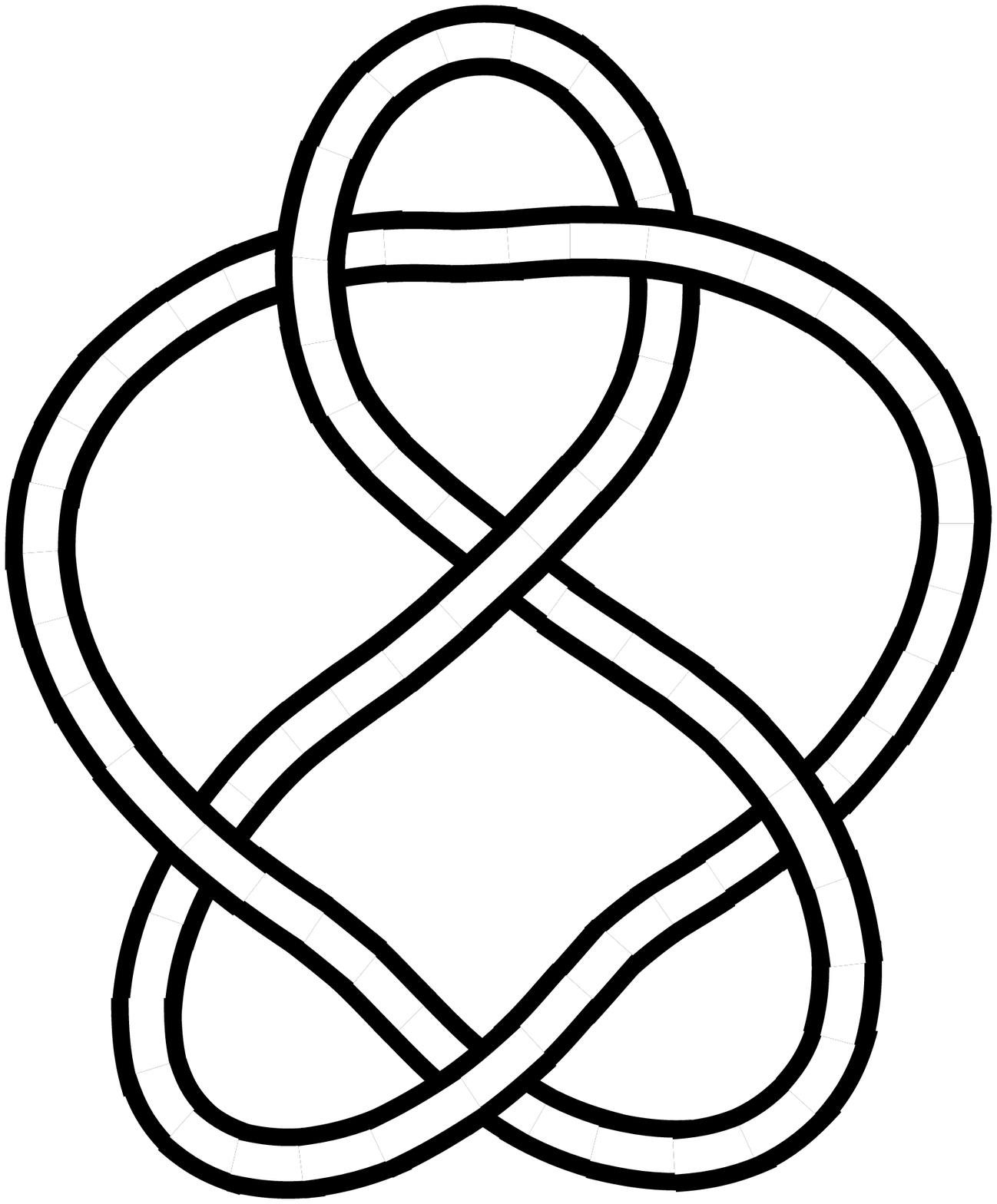}\eps{10mm}{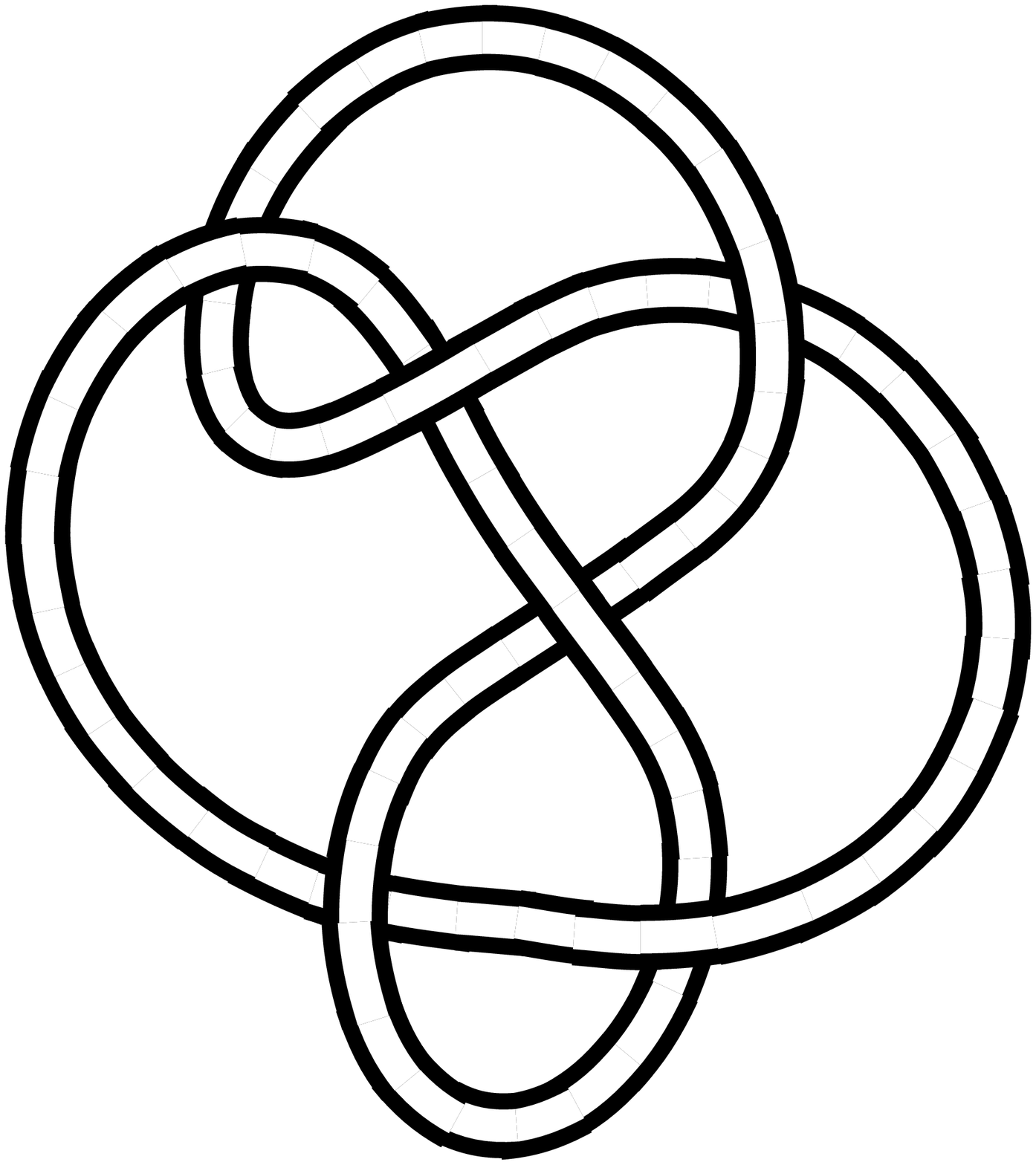}\eps{10mm}{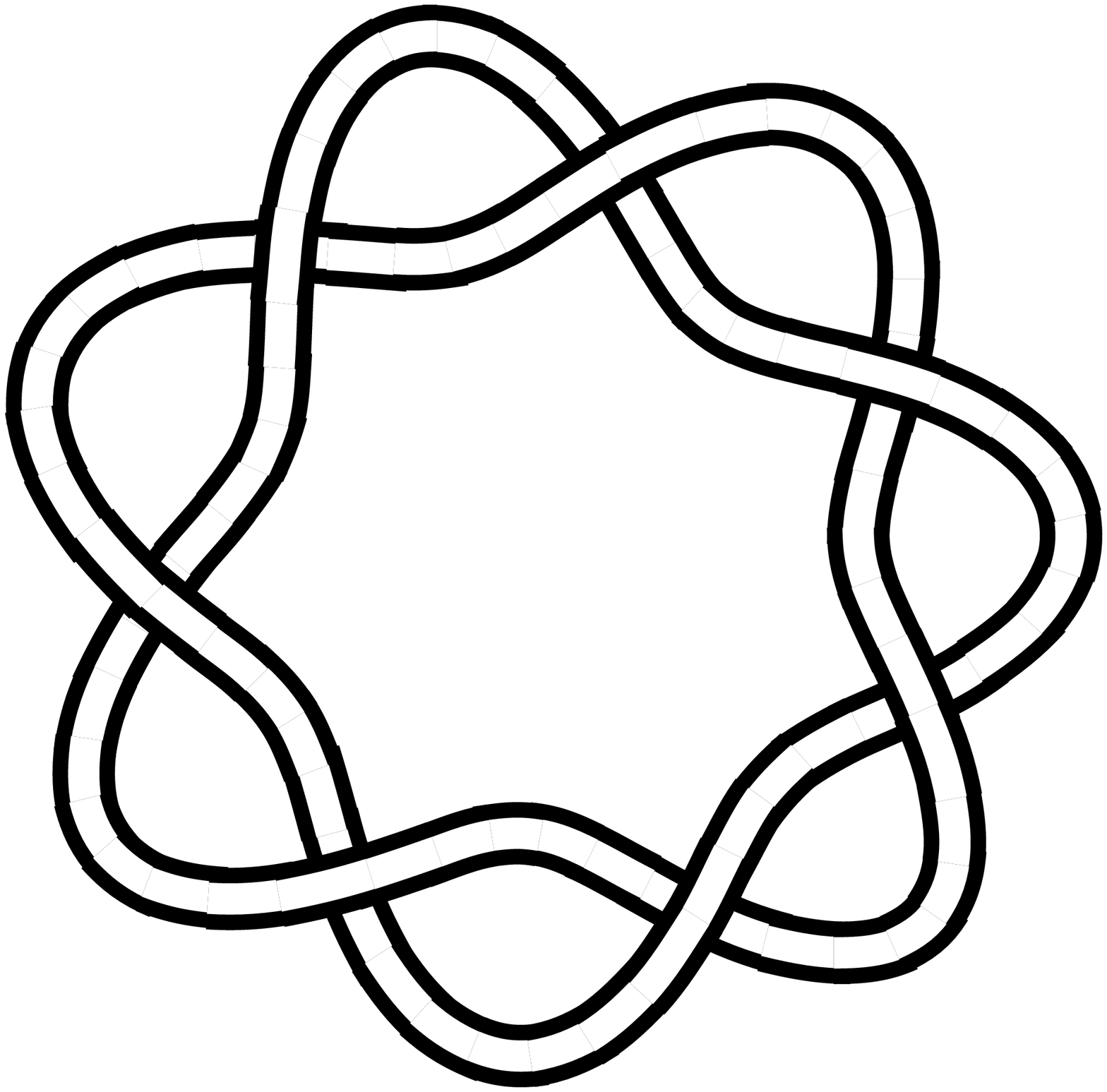}\eps{10mm}{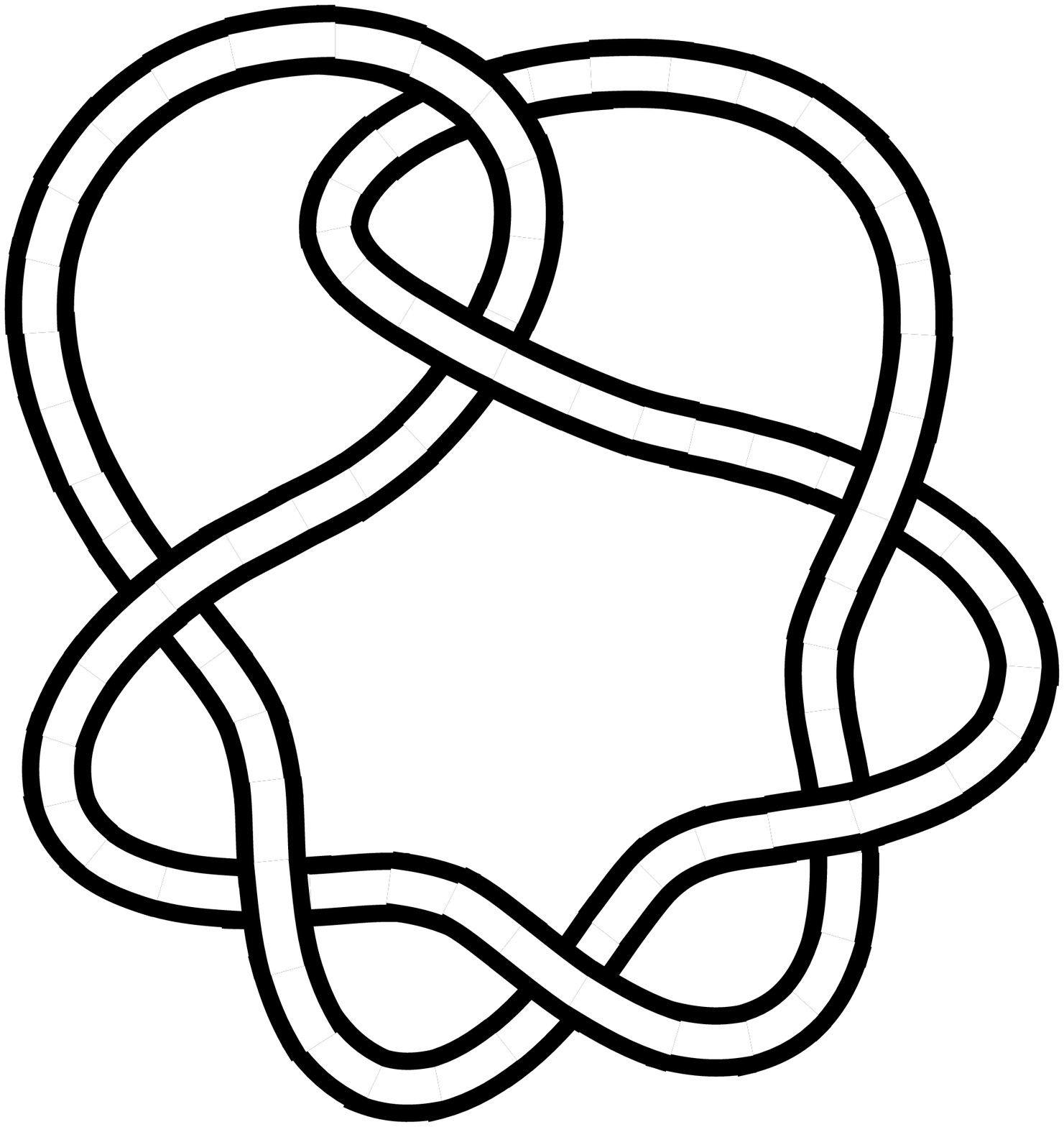}\eps{10mm}{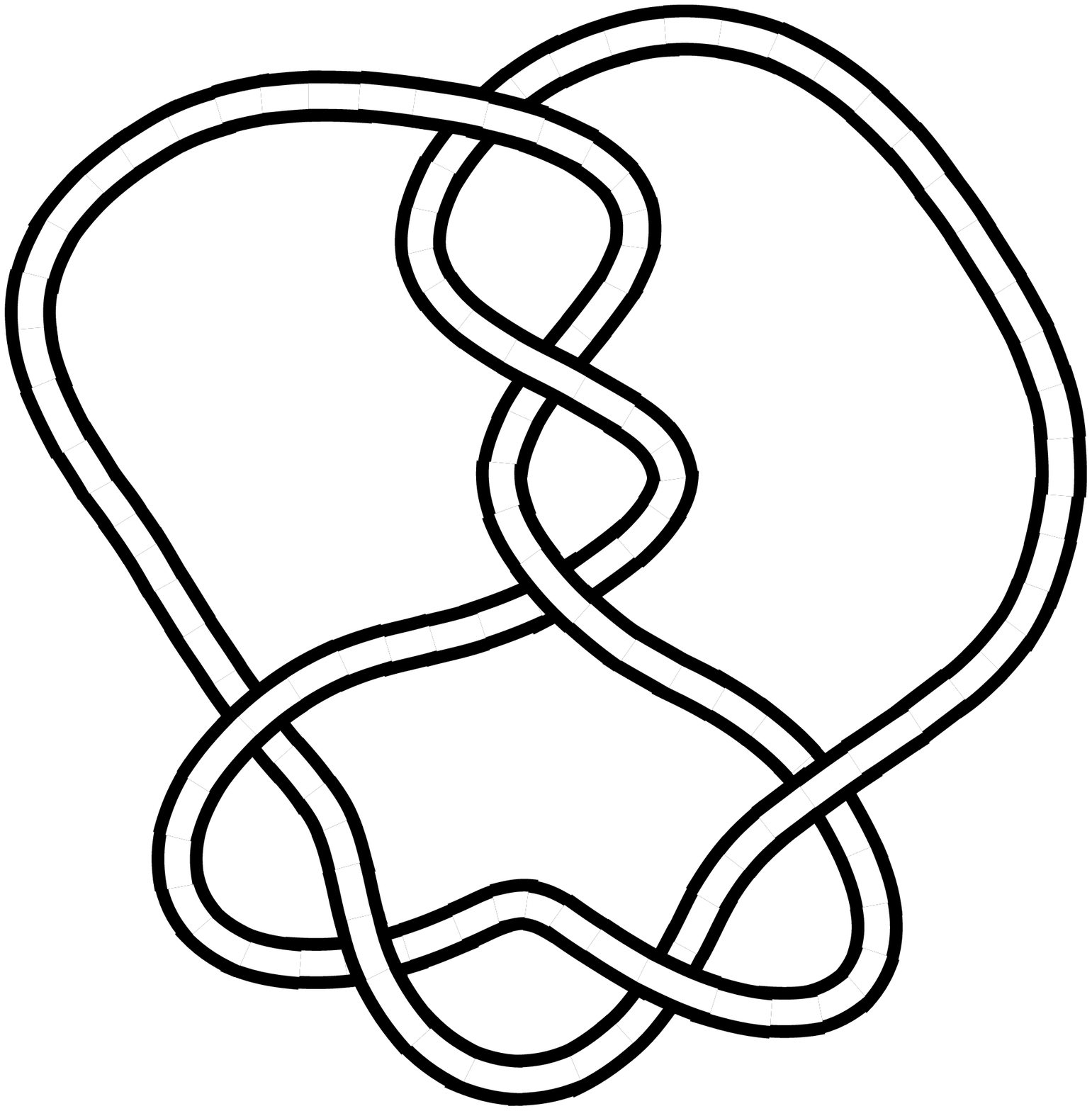} \eps{10mm}{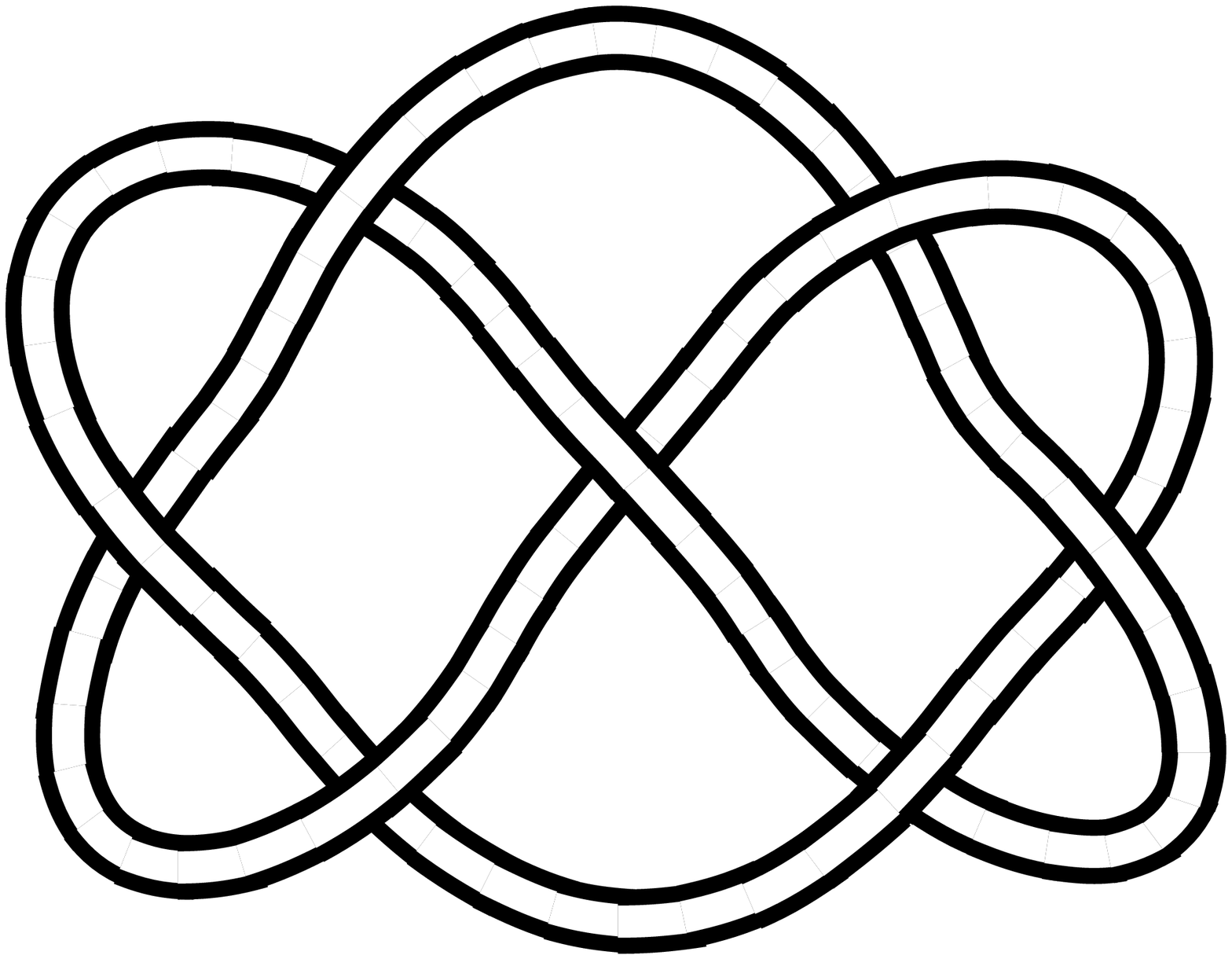}\eps{10mm}{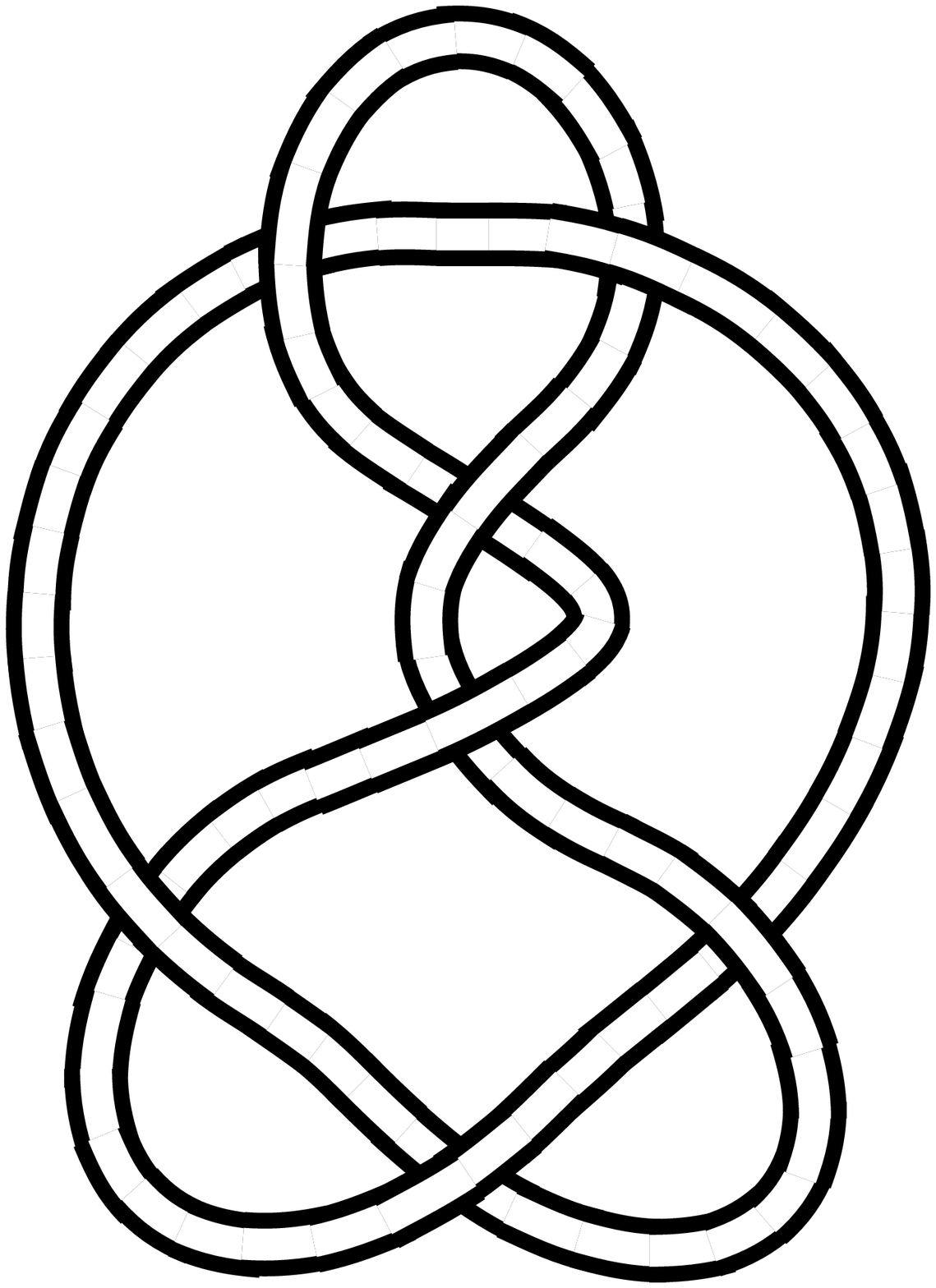}\eps{10mm}{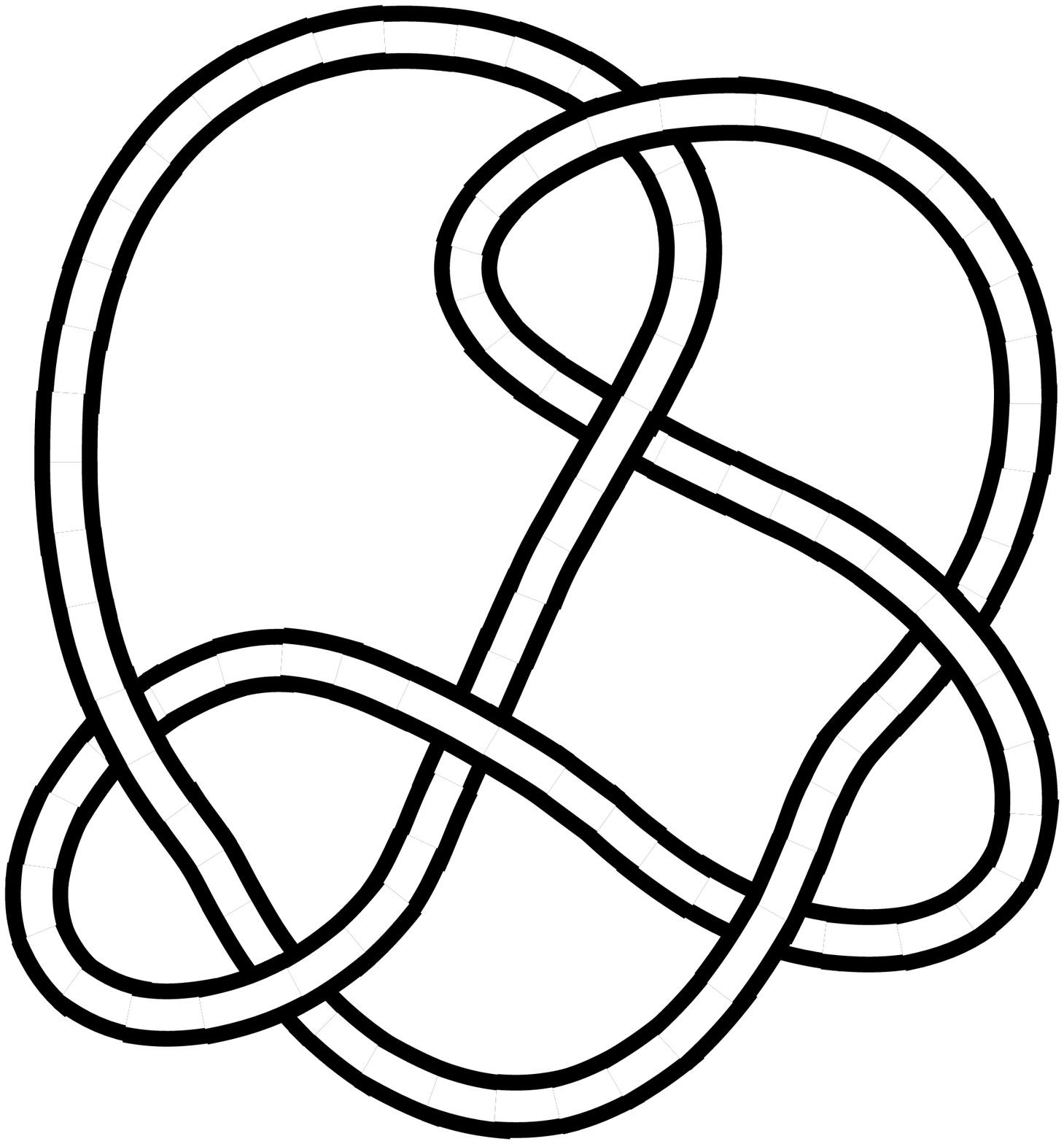}\eps{10mm}{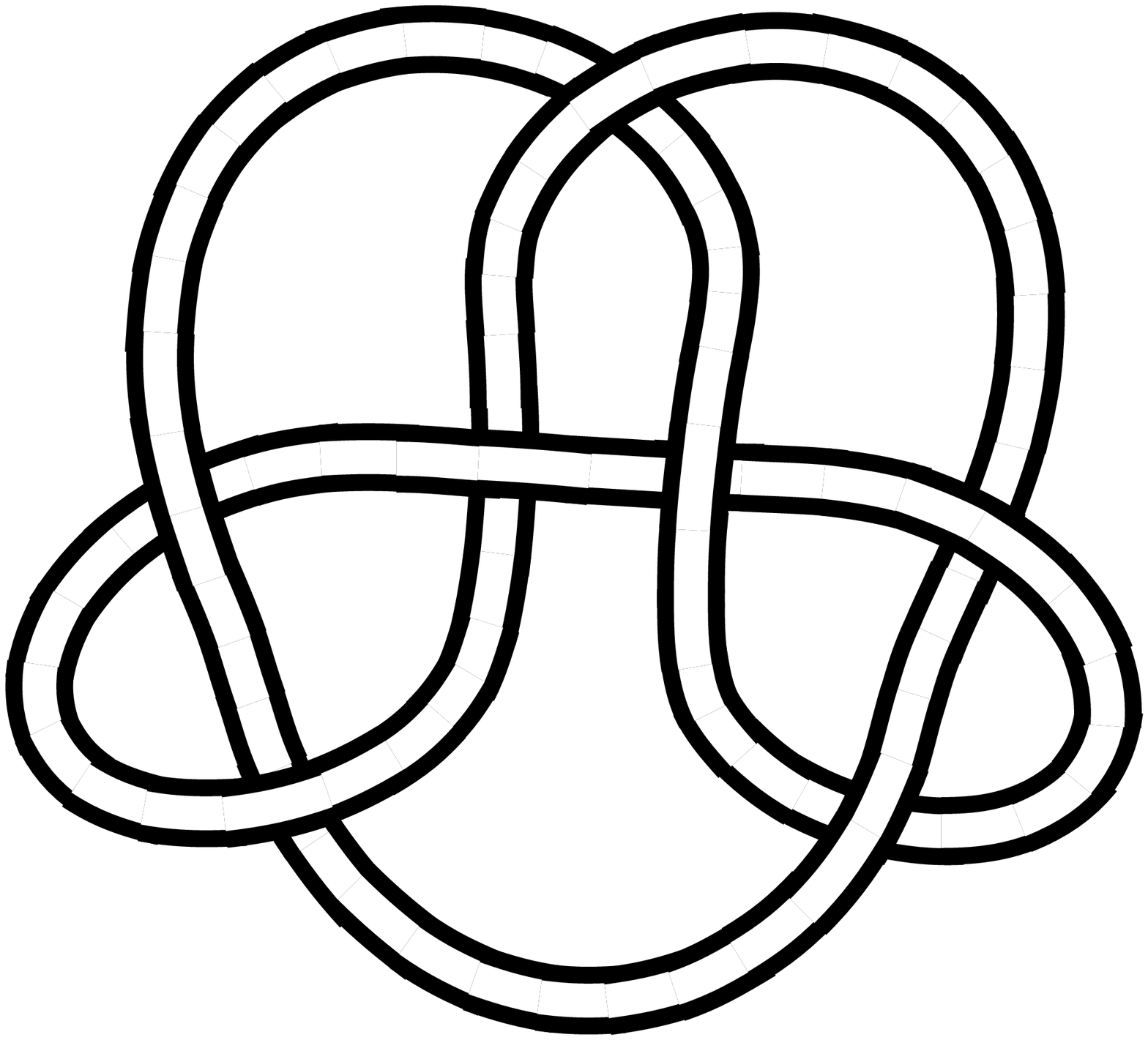}\]
\\

At the beginning of the 20th century knot theory immediately found its place along with its big sisters in the low dimensional topology family. It was realized fast that the theory of one dimensional objects ($\eps{10mm}{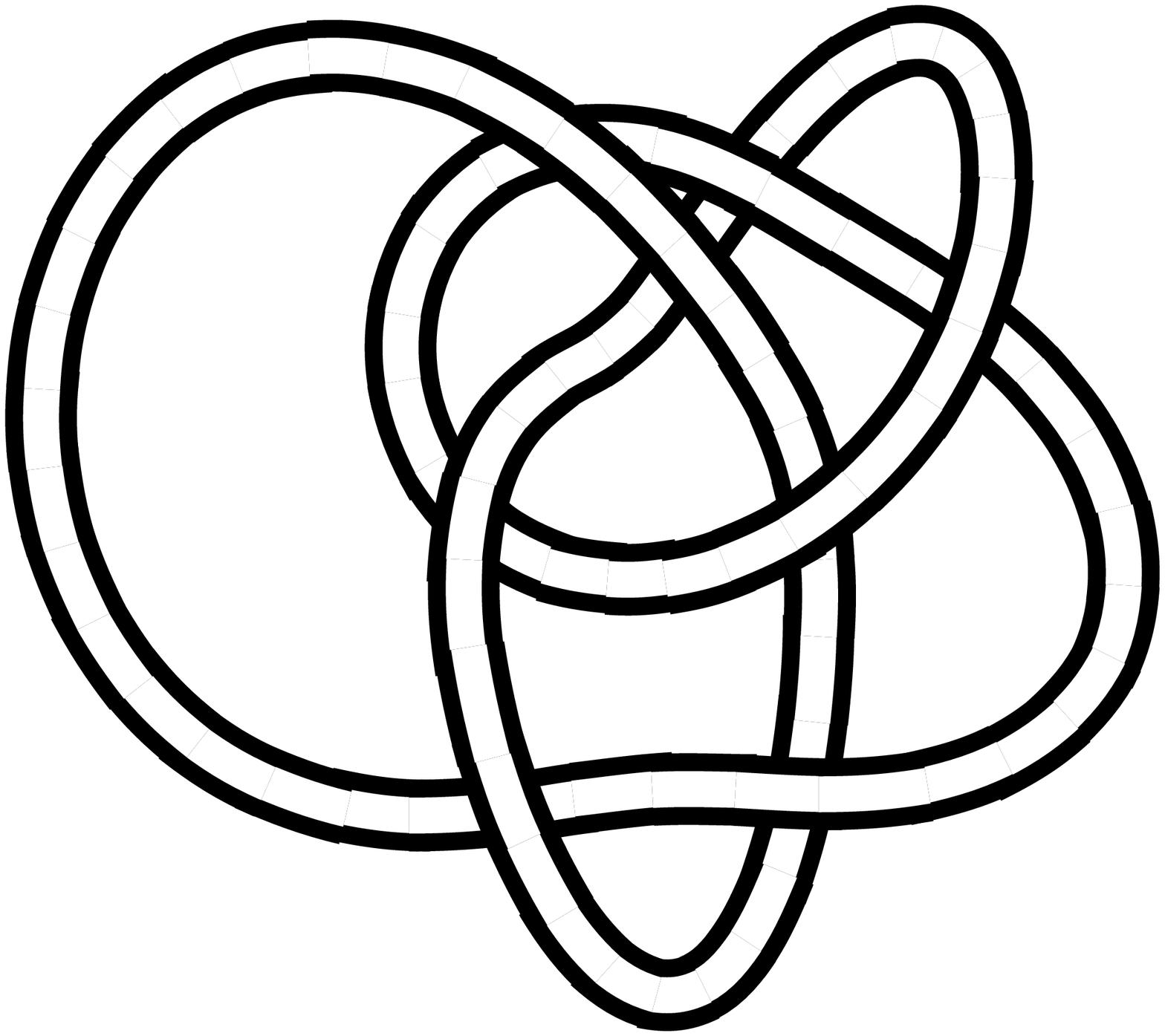}$) cannot be separated from the three dimensional objects that the knot is embedded in (mysteriously also realized by Buddhists on a completely different level of thinking \footnote{See the cover of Dale Rolfsen's book "Knots and Links" for one of the eight glorious emblems of Tibetan Buddhism - $\eps{5mm}{7_4}$}). Constructions of certain surfaces associated to knots and its embedding spaces further tightened the connection between the low dimensional topology family. These relations produced the main advances, in the early-mid 20th century, to what is now known as \emph{classical} knot theory. During that period of time many \emph{knot invariants} appeared, coming mainly from various topological constructions, using the developing theory of algebraic-topology. These algebraic knot invariants are objects that are associated to knots and stay the same (i.e invariant) as long as the knot does not change its type (i.e no cutting!). The main use for such invariants, at that time, was in the old question of classification, but as usual in math, they evolved to a topic of their own. \\

Although the ether theory was rejected due to the rise of special relativity theory, physics still plays a major role in the development of knot theory (and low dimensional topology). It is interesting to note how the physical notions change through the decades, but an elegant mathematical theory, like knot theory, is just too hard to resist. Knotted ether became knotted vortex of force field, or knotted vortex of turbulence or even knotted space-time wormholes! The big 80's string theory revolution brought with it a new and exciting era to knot theory. As tools such as conformal field theory and topological quantum field theory evolved, new insight on knot theory was brought to light from these directions. Thus \emph{quantum knot invariants theory} was born. Using ``physics tools'' like path integration and re-normalization, using analysis and combinatorics tools like perturbation theory, and using some heavy algebraic tools like Lie algebras and quantum groups, many new invariants were born and a new theory evolved around them. Although some of these algebraic invariants lack the ``topological interpretations'' (i.e. something you can actually see) all proved insightful for knot theory, and the relations between knots and three or four dimensional objects. The end of the 20th century was the age of finite type invariants and polynomial invariants, Chern-Simons theory and the Kontsevich integral, quantum algebras and configuration spaces, associators and R-matrices \footnote{See the preface for Tomotada Ohtsuki's book "Quantum Invariants" and references therein for a historical overview (and the rest of the book if you are interested in the math).}- ideas flowing in and out of quantum field theory and statistical physics. Knot theory has also enlarged its deep connections with many other closely related mathematical fields, along with some less obvious mathematical neighbors, like number theory, for example.
\\

The first spark of the quantum revolution was probably the Jones polynomial. By now there are a few definitions for this polynomial, some simple and completely combinatorial, some purely algebraic and some use heavy ``physics tools''. No matter which definition you choose, the result is a certain polynomial attached to every knot or link. This polynomial does not have a pure topological interpretation although it is a topological invariant of the knot. The Jones polynomial spawned many generalizations and extensions with similar properties, all based on the various algebraic/combinatorial/physical constructions used in its definitions. This \emph{quantum} invariant still symbolizes the simplest version of knot polynomial invariants, and serves as a basis for any theory that uses the above tools. One of its contributions to the field was demonstrating the ability of using simple combinatorial constructions on \emph{knot diagrams} (two dimensional projections of knots) to define invariants. These tools are usually encoded in what is called \emph{skein relations} -- algebraic relations between various changes of the knot diagram. Based on these combinatorial relations it is easier to construct more knot invariants.
\\

Things slowed down towards the turn of the century, but the beginning of the 21st century promised new horizons for knot theory, and knot invariants theory. A few major developments contributed to the new knot theory ``revolution''.  The first is the introduction of the \emph{categorification} notion in algebra. We will discuss it in mathematical terms in the second introduction section, but for now, the idea is as follows: take an algebraic object and view it as a certain part of a higher category of objects that encode much more information and structure. For example, take a polynomial and view it as an Euler characteristic of some homology theory in some category (I admit, a pretty vague description!). During the first years of the 00's Khovanov took this notion and categorified the Jones polynomial. This theory is known as \emph{Khovanov link homology theory}. Soon after, using a second major development in physics from the 90's (the introduction of matrix factorization tools for D-branes), categorification of other knot polynomials was introduced. A third development took place in three-four dimensional topology with the introduction of Hegaard-Floer homology which contributed the knot Floer homology (which is a categorification of the Alexander knot polynomial). These developments opened the door to other categorification schemes, one of which is the \emph{geometric formalism} due to Bar-Natan. It also opened the door to large amount of mathematical work experimenting computationally and theoretically these theories, trying to uncover its properties. \textbf{This thesis is about the geometric formalism of Khovanov's link homology theory}.

\vspace{5mm}
There are many questions surrounding these new theories, and work done on the subject in the last 5 years is still in progress. There are questions regarding the deep relations between these theories and some aspects of topological string theory (physics again!). There are questions regarding the connection between the various categorifications, and ways to unite them all under one formalism. There are structural questions and computational questions, theoretical questions and applied questions, big questions and small questions - a new field was born. This thesis is focusing on Khovanov's homology theory, and using the geometric formalism I try to answer some of these questions. In loose terms, the questions that this thesis focuses on are questions of information held within the theory (the exact underlying structure of the theory), questions of \emph{universality} (the largest amount of information extractable) and questions of information extraction (how to compute it?). If you are some sort of a mathematician (or a physicist), please proceed to the second introduction section (experts jump to the third introduction section). If you are not, I hope you enjoyed the good story and maybe learned some knot history!
\\

\section{Main Introduction}

We now turn to the main part of the introduction chapter. I assume the reader has some strong mathematical background, perhaps knows knot theory, maybe even heard about Khovanov homology, but is definitely not familiar with the geometric construction of it. I will introduce the main definitions and constructions which form the starting point of this thesis. If you already know the original construction by Khovanov (from \cite{kho2} or \cite{ba5} for example), or other common combinatorial constructions \cite{viro1}, I still recommend reading this part of the introduction chapter, as this thesis uses the geometric formalism only.
\\

The geometric formalism for constructing Khovanov's link homology is due to Bar-Natan and described in full in \cite{ba1}. I follow \emph{closely} the definitions and constructions described in that paper (borrowing many nice figures with permission). Once the reader is familiar with the geometric formalism he/she may consider themselves as experts in the field and continue to the expert introduction.
\\

\subsection{The cube of resolutions}\label{susection:cubefigure}
We start with a \emph{knot}. An embedding of $S^1$ into a 3 dimensional space, say $\mathbb{R}^3$, considered up to isotopies. A non intersecting union of few knots is called a link. In this chapter we will use the word link to denote links or knots. Given a link, one projects it on a two dimensional plane (in a generic way, i.e. no singularities except transversal self crossings) to get \emph{a link diagram}. The diagram also keeps track, at each crossing, which strand is over and which is under (i.e. we remember the information of the 3 dimensional embedding). We enumerate the crossings. Here is an example of a Trefoil knot diagram :
\[ \eps{50mm}{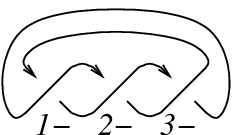} \]

After giving the link an orientation we distinguish between positive and negative crossings according to the convention:

\[ \eps{40mm}{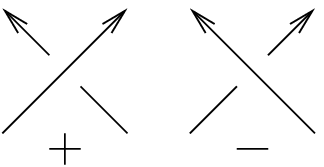}\]

A link may have many diagrams, of course, but due to a theorem by Reidemeister, diagrams belong to of a single link type if and only if they are connected by a series of diagram moves, called the Reidemeister moves:

\[ \eps{150mm}{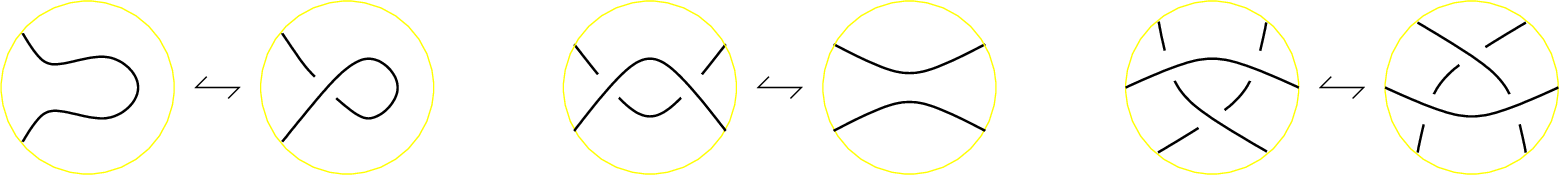} \]

We now describe how to construct a certain chain complex given a link diagram. This chain complex is known as \emph{the cube of resolutions}. Given a link diagram, denote by $n$ the total number of crossings. $n_+$ and $n_-$ will denote the number of positive and negative crossings respectively. The cube of resolutions is indeed an $n$ dimensional cube with $2^n$ vertices. Let us describe the vertices and edges of that cube.
\\

\textbf{Vertices.} Each vertex of the cube carries a \emph{smoothing} of the link diagram. A smoothing is a planar diagram obtained by replacing every crossing in the diagram with either a ``$0$-smoothing'' or with a ``$1$-smoothing'' according to the following rules : a crossing involves two strands. The $0$-smoothing is when you enter on the lower strand (level $0$) and turn right at the crossing. The $1$-smoothing is when you enter on the upper strand (level $1$) and turn right at the crossing. The following diagram demonstrates this notion using arrows (not to be confused with the diagram overall orientation) \footnote{One can also put her hand on the upper strand and rotate it counterclockwise to pick the regions that are connected in the $0$-smoothing.}.

\[ \eps{150mm}{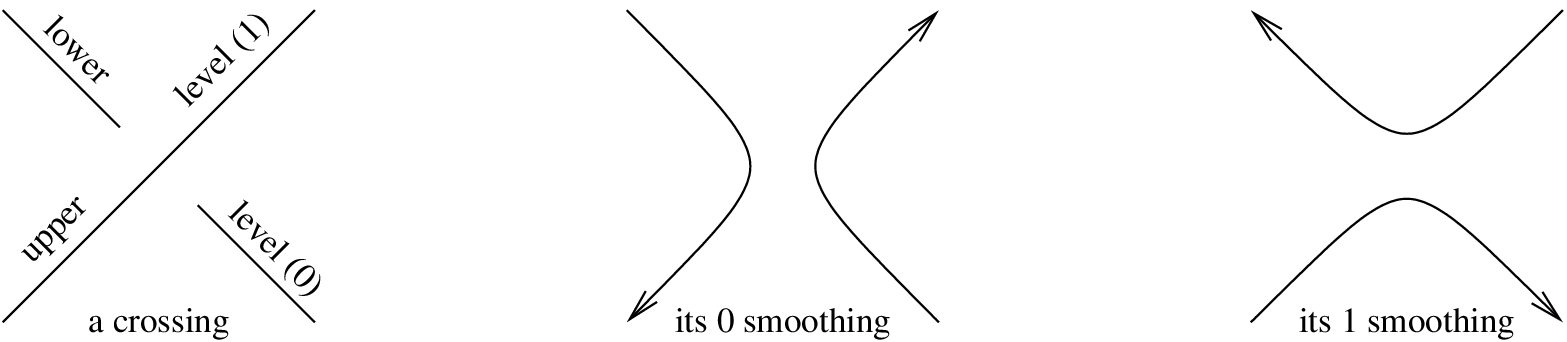} \]

\noindent Each smoothing (that is each vertex) can be coded via a sequence of $n$ numbers taking $0$ or $1$ values, according to the type of smoothing each crossing was resolved to. An example of a vertex for our Trefoil would look like:

\[ \eps{40mm}{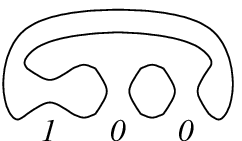} \]

\noindent The total cube is skewered along its main diagonal. More precisely, each vertex of the cube has a ``height'', the sum of its coordinates, a number between $0$
and $n$. The cube is displayed in such a way so that vertices of height
$k$ project down to the point $k-n_-$ on a line marked below the cube. With such notation, each edge of the cube is marked in the natural manner by n-letter strings of $0$'s and $1$'s with exactly one $\star$ (the $\star$ denotes the coordinate which
changes from $0$ to $1$ along a given edge). The trefoil cube will look like:

\[\eps{100mm}{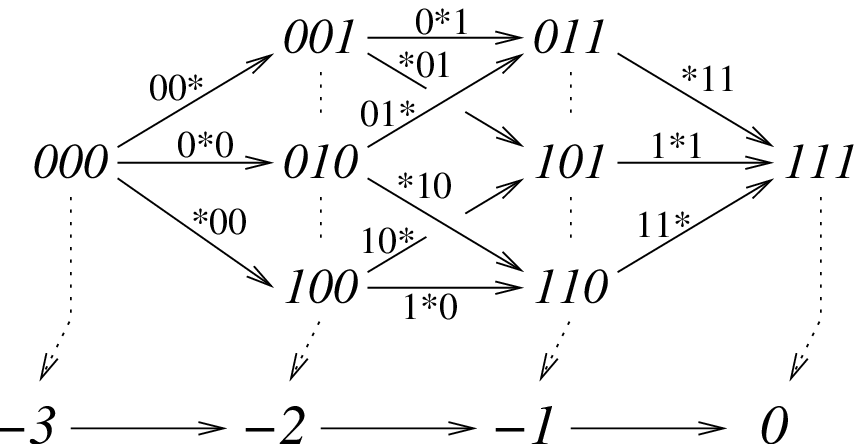}\]

\vspace{3mm}
\textbf{Edges.} Each edge of the cube carries a \emph{cobordism} between the smoothing on the tail of that edge and the smoothing on its head. A cobordism is an oriented two dimensional surfaces embedded in $\mathbb{R}^2\times[0,1]$ whose boundary lies entirely in $\mathbb{R}^2\times\{0,1\}$ and whose ``top'' boundary is the ``tail'' smoothing and whose ``bottom'' boundary is the ``head'' smoothing. Specifically, to get the cobordism for an edge $(\xi_i)\in\{0,1,\star\}^n$ for which $\xi_j=\star$ we remove a disk
neighborhood of the crossing $j$ from the smoothing $\xi(0):=\xi|_{\star\to 0}$ of our link diagram, cross with $[0,1]$, and fill the empty cylindrical slot around the missing crossing with a saddle cobordism: $\eps{10mm}{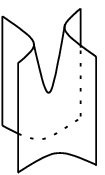}$. Here is an example of a cobordism on the edge $\star00$ in the Trefoil example :

\[ \eps{70mm}{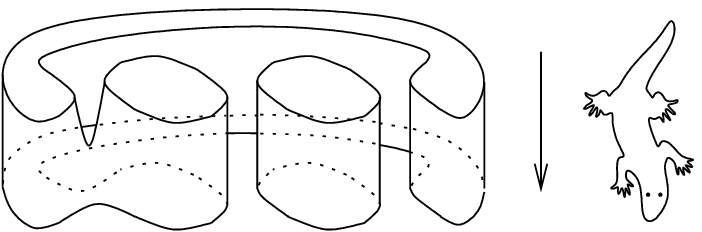}\]

\vspace{2mm}

\noindent We use the following symbolic notations for cobordism: the diagram-piece
$\HSaddleSymbol$ stands for the saddle cobordism with top $\smoothing$ and
bottom $\hsmoothing$.

\textbf{Signs.} Some of the edges (cobordisms) carry signs with them.
If an edge $\xi$ is labeled by a sequence $(\xi_i)$ in the alphabet
$\{0,1,\star\}$ and if $\xi_j=\star$, then the sign on the edge $\xi$
is $(-1)^\xi:=(-1)^{\sum_{i<j}\xi_i}$. The basis to the exterior algebra in $n$ generators can be easily used to determine these signs, as the following picture demonstrates for the Trefoil example :

\[ \eps{80mm}{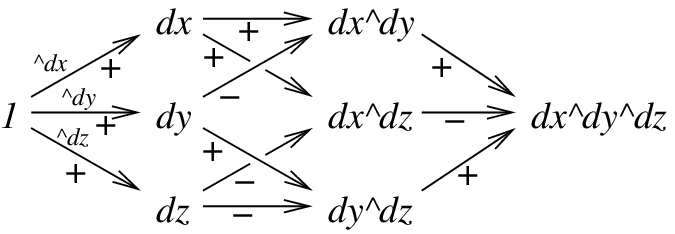}\]

\vspace{2mm}

Put all together, the cube of resolutions for the Trefoil knot will look like:

\[
\input figs/Main.pstex_t
\]


\subsection{Categorical frame for the cube of resolutions and the chain complex associated to a link}

We now define the category in which the cube of resolution is an object of. This would turn the cube of resolution into a chain complex, and prepare the ground for defining a homology theory associated to each link. We build this category step by step.

\begin{definition}
$\Cob$ is the category whose objects are smoothings (i.e.,
simple curves in the plane) and whose morphisms are cobordisms between such
smoothings. The cobordisms regarded up to boundary-preserving isotopies.
The composition of morphisms is given by placing one cobordism atop the other. The smoothing on the right (in the usual composition notation) goes on top. For example: $\eps{50mm}{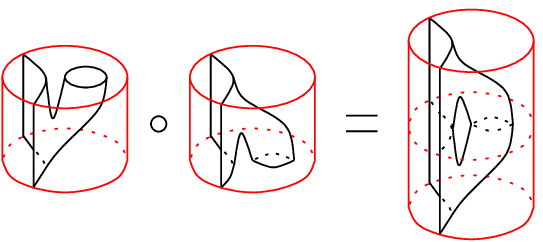}$.
\end{definition}

An \emph{additive category} is a category in which the sets of morphisms
(between any two given objects) are Abelian groups and the composition
maps are bilinear. Let $\calC$ be some arbitrary
category. If it is not additive to start with we make it additive by extending every set of morphisms $\Mor(\calO,\calO')$ to also allow formal $\mathbb{Z}$-linear combinations of morphisms and by extending the composition maps in the
natural bilinear manner.
\\

\begin{definition} Given an additive category $\calC$ as above,
the additive category $\Mat(\calC)$ is defined as follows:
\begin{itemize}
\item The objects of $\Mat(\calC)$ are formal direct sums (possibly empty)
  $\oplus_{i=1}^n\calO_i$ of objects $\calO_i$ of $\calC$.
\item If $\calO=\oplus_{i=1}^m\calO_i$ and $\calO'=\oplus_{j=1}^n\calO'_j$,
  then a morphism $F:\calO'\to\calO$ in $\Mat(\calC)$ will be an $m\times
  n$ matrix $F=(F_{ij})$ of morphisms $F_{ij}:\calO'_j\to\calO_i$ in
  $\calC$.
\item Morphisms in $\Mat(\calC)$ are added using matrix addition.
\item Compositions of morphisms in $\Mat(\calC)$ are defined by a rule
  modeled on matrix multiplication, but with compositions in $\calC$
  replacing the multiplication of scalars,
  \[
    \left((F_{ij})\circ(G_{jk})\right)_{ik} := \sum_jF_{ij}\circ G_{jk}.
  \]
\end{itemize}
\end{definition}

It is often convenient to represent objects of $\Mat(\calC)$ by column
vectors and morphisms by bundles of arrows pointing from
one column to another, viewed as ``matrices'' of morphisms.

\[\input figs/Mat.pstex_t\]

Denote the cube of resolutions associated to a link diagram $K$ by $\|K\|$. The cube of resolution for a link diagram can be interpreted as a chain of morphisms $\| K\|=\left(\xymatrix{\|
K\|^{-n_-}\ar[r]&\| K\|^{-n_-+1}\ar[r]&\dots\ar[r]&\|
K\|^{n_+}}\right)$.

\begin{definition} Given an additive category $\calC$, let
$\Kom(\calC)$ be the category of complexes over $\calC$, whose objects
are chains of finite length $\xymatrix{ \dots \ar[r] & \Omega^{r-1}
\ar[r]^{d^{r-1}} & \Omega^r \ar[r]^{d^r} & \Omega^{r+1} \ar[r] & \dots
}$ for which the composition $d^r\circ d^{r-1}$ is $0$ for all $r$, and
whose morphisms $F:(\Omega_a^r,\,d_a)\to(\Omega_b^r,\,d_b)$ are
commutative diagrams:
\\
\[\xymatrix{
  \dots \ar[r] &
  \Omega_a^{r-1} \ar[r]^{d_a^{r-1}} \ar[d]_{F^{r-1}} &
  \Omega_a^r \ar[r]^{d_a^r} \ar[d]_{F^r} &
  \Omega_a^{r+1} \ar[r] \ar[d]_{F^{r+1}} &
  \dots  \\
  \dots \ar[r] &
  \Omega_b^{r-1} \ar[r]^{d_b^{r-1}} &
  \Omega_b^r \ar[r]^{d_b^r} &
  \Omega_b^{r+1} \ar[r] &
  \dots
}\]
\\
in which all arrows are morphisms in $\calC$. Like in ordinary homological algebra, the
composition $F\circ G$ in $\Kom(\calC)$ is defined via $(F\circ G)^r:=F^r\circ G^r$.
\end{definition}

\begin{proposition} \cite{ba1} For any link diagram $K$ the cube of resolution (i.e. the chain complex $\| K\|$) is a complex in $\Kom(\Mat(\Cob))$. I.e., $d^r\circ d^{r-1}$ is always $0$ for these chains.
\end{proposition}

\begin{definition}
Let $\calC$ be a category. Just like in ordinary homological algebra, we
say that two morphisms $F,G:(\Omega_a^r)\to(\Omega_b^r)$ in $\Kom(\calC)$
are \emph{homotopic} (and we write $F\sim G$) if there exists ``backwards
diagonal'' morphisms $h^r:\Omega_a^r\to \Omega_b^{r-1}$ so that
$F^r-G^r=h^{r+1}d^r+d^{r-1}h^r$ for all $r$.

\[\xymatrix@C=2cm{
  \Omega_a^{r-1} \ar[r]^{d_a^{r-1}}
    \ar@<-2pt>[d]_{F^{r-1}} \ar@<2pt>[d]^{G^{r-1}} &
  \Omega_a^r \ar[r]^{d_a^r} \ar[ld]_{h^r}
    \ar@<-2pt>[d]_{F^r} \ar@<2pt>[d]^{G^r} &
  \Omega_a^{r+1} \ar[ld]_{h^{r+1}}
    \ar@<-2pt>[d]_{F^{r+1}} \ar@<2pt>[d]^{G^{r+1}} \\
  \Omega_b^{r-1} \ar[r]^{d_b^{r-1}} &
  \Omega_b^r \ar[r]^{d_b^r} &
  \Omega_b^{r+1}
}\]
\end{definition}

\begin{definition} $\Komh(\calC)$ is $\Kom(\calC)$ modulo homotopies. That
is, $\Komh(\calC)$ has the same objects as $\Kom(\calC)$ (formal
complexes), but homotopic morphisms in $\Kom(\calC)$ are declared to be the
same in $\Komh(\calC)$.
\end{definition}

\begin{definition}
Let $\Cobl$ denote the category $\Cob$ with the morphisms mod out by the following local relations :

\[ The~4TU~relation:
\begin{array}{c}
\includegraphics[height=1cm]{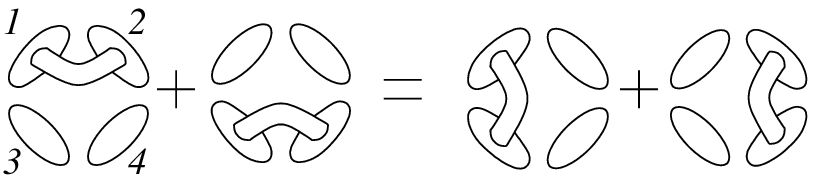}
\end{array}
\]
\[
The~S~relation: \eps{7mm}{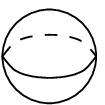}=0
\]
\[
The~T~relation: \eps{7mm}{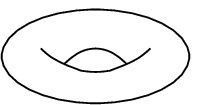}=2.
\]

The $S$ relation says that whenever a cobordism contains a sphere it is replaced by zero. The $T$ relation says that whenever a cobordism contains a torus, the torus
may be dropped and replaced by $2$. To understand the $4TU$ relation, start from some given cobordism and assume its intersection with a certain ball is the union of four disks $D_1$ through $D_4$ (these disks may well be on different connected
components of $C$). Let $C_{ij}$ denote the result of removing $D_i$
and $D_j$ from $C$ and replacing them by a tube that has the same
boundary. The relation asserts that $C_{12}+C_{34}=C_{13}+C_{24}$.
\end{definition}

Note that under these relations it does not matter whether we work with embedded surfaces or abstract surfaces in our cobordism category (see \cite{ba1} and chapter \ref{chap:comments} of this thesis).

\begin{theorem} \cite{ba1}
The isomorphism class of the cube of resolutions (i.e. the complex $\| K\|$) regarded in $\Komh(\Mat(\Cobl))$ is an invariant of the link. That
is, it does not depend on the ordering of the layers of a cube as column
vectors and on the ordering of the crossings and it is invariant under
the three Reidemeister moves.
\end{theorem}

\textbf{Summary:} What do we have so far? Given a link diagram one can construct the cube of resolutions. This was done in the original construction of Khovanov, as it is done in any other constructions of Khovanov homology. The novelty of the geometric construction, due to Bar-Natan, is to view this cube of resolutions as a chain complex in an appropriate category. The category is chosen as to make the chain complex homotopy type \emph{a link invariant}. This is done \emph{before} applying any homology theory to get the categorification of the Jones polynomial. We now continue and see how one uses the geometric construction to create homology theories, which give rise to the categorification of the Jones polynomial, as in Khovanov's original construction.

\subsection{Homology theories and the Jones polynomial}

Let us now give a definition of the Jones polynomial, which this theory categorify. We use the combinatorial definition of \emph{a skein relation} or \emph{a bracket}.

\begin{definition}
The Jones polynomial of a link $J(K)$ is the polynomial one gets by the following procedure. First, we apply the following (\emph{bracket}) relation repeatedly on every crossing:
\[\langle\backoverslash \rangle = \langle\hsmoothing\rangle -q\langle\smoothing\rangle\]
Then, we apply the (normalization) relation :
\[\langle \bigcirc \rangle = q^{-1}+q\]
The Jones polynomial (in this specific normalization) is then defined to be
\[J(K)=(-1)^{n_-}q^{n_+-2n_-}\langle K \rangle\]
\end{definition}

By categorification of the Jones polynomial we mean finding a graded homology theory $\mathcal{H}^{i,j}$ such that the graded Euler characteristic of it is equal to $J(K)$ :
\[J(K)=\sum(-1)^iq^jdim\mathcal{H}^{ij}\]

So far we have a chain complex associated to the a link in $\Komh(\Mat(\Cobl))$. We wish to get a homology theory out of it which categorify the Jones polynomial. For this purpose we will need two things. First, we will need some grading. Second, we will need a functor taking us from the additive category $\Komh(\Mat(\Cobl))$  to an Abelian category, where kernels make sense (thus we will get a homology theory). We start with the grading, again following closely the constructions in \cite{ba1}.

\begin{definition} A \emph{graded category} is an additive category $\calC$
with the following two additional properties:
\begin{enumerate}
\item For any two objects $\calO_{1,2}$ in $\calC$, the morphisms
  $\Mor(\calO_1,\calO_2)$ form a graded Abelian group, the composition
  maps respect the gradings and all identity maps are of degree $0$.
\item There is a $\mathbb{Z}$-action $(m,\calO)\mapsto\calO\{m\}$, called
  ``grading shift by $m$'',  on the objects of $\calC$. As plain
  Abelian groups, morphisms are unchanged by this action,
  $\Mor(\calO_1\{m_1\},\calO_2\{m_2\}) = \Mor(\calO_1,\calO_2)$. But
  gradings do change under the action; so if
  $f\in\Mor(\calO_1,\calO_2)$ and $\deg f=d$, then as an element of
  $\Mor(\calO_1\{m_1\},\calO_2\{m_2\})$ the degree of $f$ is $d+m_2-m_1$.
\end{enumerate}
\end{definition}

\begin{definition}
Let $C\in\Mor(\Cob)$ be a cobordism in a cylinder, with $|B|$
vertical boundary components on the side of the cylinder. Define $\deg
C:=\chi(C)-\frac12|B|$, where $\chi(C)$ is the Euler characteristic of
$C$.
\end{definition}

Using the above definitions one can show that $\Cob$ is a graded
category, and so is $\Cobl$. Hence so is the target category $\Komh(\Mat(\Cobl))$.

\begin{definition} Let $K$ be a link diagram with $n_+$ positive
crossings and $n_-$ negative crossings. Let $\Kh(K)$ be the complex whose
chain spaces are $\Kh^r(K):=\| K\|^r\{r+n_+-n_-\}$ and whose
differentials are the same as those of $\| K\|$:
\[
  \def\st{\displaystyle}
  \begin{array}{rccccccc}
    \st\| K\|: & \quad &
      \st\| K\|^{-n^-} & \st\longrightarrow &
      \st\cdots & \st\longrightarrow &
      \st\| K\|^{n_+} \\
    &&&&&& \\
    \st\Kh(K): & &
    \st\| K\|^{-n^-}\{n_+-2n_-\} & \st\longrightarrow &
      \st\cdots & \st\longrightarrow &
      \st\| K\|^{n_+}\{2n_+-n_-\}
  \end{array}
\]
\end{definition}

\begin{theorem} \cite{ba1}
\begin{enumerate}
\item All differentials in $\Kh(K)$ are of degree $0$.
\item $\Kh(K)$ is an invariant of the link $L$ up to degree-$0$ homotopy
  equivalences. I.e., if $K_1$ and $K_2$ are tangle diagrams which
  differ by some Reidemeister moves, then there is a homotopy equivalence
  $F:\Kh(L_1)\to\Kh(L_2)$ with $\deg F=0$.
\end{enumerate}
\end{theorem}

We now have a graded version of the geometric chain complex in our hands. Moreover, the grading was chosen in such a way as to fit Khovanov's original construction in \cite{kho2}, which categorified the Jones polynomial. We do not have a homology theory though. There are various types of functors one can apply to the geometric complex.
Let $\calA$ be some arbitrary Abelian category (category of modules for example). Any functor $\calF:\Cobl\to\calA$ extends to a functor $\calF:\Komh(\Mat(\Cobl))\to\Kom(\calA)$. Thus for any link diagram $L$,
$\calF\Kh(L)$ is an ``ordinary'' chain complex and an up-to-homotopy invariant of
the link $L$. Thus the isomorphism class of the homology $H(\calF\Kh(L))$ is an invariant of $L$. If in addition $\calA$ is graded and the functor $\calF$ is degree-respecting then the homology $H(\calF\Kh(L))$ is a graded invariant of $L$.
\\

One choice of such functors is a topological quantum field theory (a TQFT). These functors on $\Cob$ are valued in a category of graded modules over a ring and map disjoint unions to tensor products (``tensorial'' functors). TQFT functors are classified by the Frobenius algebra $V$ which satisfies $\calF(\bigcirc)=V$ (see \cite{ab} for all the appropriate definitions). Khovanov's original choice of Frobenius algebra was the graded $\mathbb{Q}$-module freely generated by two elements $\{v_\pm\}$ with $\deg v_\pm=\pm1$ defined on the morphism $\calF(\epsg{5mm}{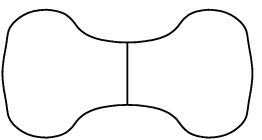})=\Delta:V\to V\otimes V$ and
$\calF(\epsg{5mm}{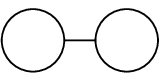})=m:V\otimes V\to V$ given by

\[
  \Delta: \begin{cases}
    v_+ \mapsto v_+\otimes v_- + v_-\otimes v_+ &\\
    v_- \mapsto v_-\otimes v_-  &
  \end{cases}
  m: \begin{cases}
    v_+\otimes v_-\mapsto v_- &
    v_+\otimes v_+\mapsto v_+ \\
    v_-\otimes v_+\mapsto v_- &
    v_-\otimes v_-\mapsto 0
  \end{cases}
\]
\\

After identifying $v_-$ with X and $v_+$ with 1 This Frobenius algebra is $V=\mathbb{Q}[X]/X^2$. This functor is well defined on $\Cobl$ (satisfies the S/T/$4TU$ relations) and thus extends to $\Komh(\Mat(\Cobl))$, creating a homology theory on the geometrical complex. By comparing constructions, from \cite{kho2} one gets that the graded Euler characteristic of $\mathcal{H}\calF Kh(L)$ is the Jones polynomial $J(L)$. Thus the geometric formalism succeeds in reproducing Khovanov's link homology theory and generalizes it to some extent.
\\

Another type of functors are the ``tautological'' functors.
\begin{definition} Let $\calO$ be an object of $\Cobl$. The tautological functor $\calF_\calO:\Cobl\to\mathbb{Z}-mod$ is defined on objects by $\calF_\calO(\calO'):=\Mor(\calO,\calO')$ and on morphisms by composition on the left. \end{definition}

It was shown in \cite{ba1} that
\[
  \calF_(\calO'):=\mathbb{Z}[\frac{1}{2}]\otimes_\mathbb{Z}\Mor(\emptyset,\calO')/((g>1)=0)
\]
where $(g>1)=0$ means mod out all surfaces with genus higher than 1, is equivalent to Khovanov's original homology theory (the one mentioned above).
\\

It is the place to say that the geometric complex has some other advantages such as being a local theory, that is, defined over tangles and behaves well under tangle compositions (property that the original theory lacked). It is also the place to mention that now that you read this part of the introduction chapter, you can consider yourself an expert in the field and continue on to the next part of the introduction. The next section finally says something about this thesis !

\section{Expert introduction}

The geometric formalism of Khovanov link homology (described fully in the previous section) gives not only a visual description of the ``standard'' Khovanov link homology construction, but also creates a unifying underlying framework for any Khovanov type link homology theory. It emerged together with the development of the algebraic language used in the categorification process but was not explored as much.
\\

The basic idea of the geometric formalism is as follows (see section \ref{susection:cubefigure} for full details, if needed). Given a link diagram $D$ one builds the ``cube of resolutions'' from it (a cube built of all possible 0 and 1 smoothings of the crossings). The edges of the cube are then given certain surfaces (cobordisms) attached to them (with the appropriate signs). The entire cube is ``summed'' into a complex (in the appropriate geometric category) while taking care of degree issues. The figure for the Trefoil knot example, produced in section \ref{susection:cubefigure}, should serve the reader as a reminder.
\\


The category in which one gets a link invariant is
$\Komh(\Mat(\Cobl))$, the category of complexes, up to homotopy,
built from columns and matrices of objects and morphisms
(respectively) taken from $\Cobl$. $\Cobl$ is the category of
2-dimensional (orientable) cobordisms between 1 dimensional objects
(circles), where we allow formal sums of cobordisms over some ground
ring, modulo the following local relations:
\[ The~4TU~relation:
\begin{array}{c}
\includegraphics[height=1cm]{figs/4Tu.eps}
\end{array}
\]
\[
The~S~relation: \eps{7mm}{S}=0
\]
\[
The~T~relation: \eps{7mm}{T}=2.
\]

Though studied in ~\cite{ba1}, the full scope of the geometric theory
was not explored, and only various reduced cases (with extra
geometrical relations put and ground ring adjusted) were used in
connection with TQFTs and homology calculations. A full understanding of the interplay between TQFTs used to create different link homology theories and the underlying geometric complex was not achieved although the research on the TQFT side is
considerably advanced ~\cite{kho1}.
\\

The objectives of this thesis are to explore the full geometric
theory (working over $\mathbb{Z}$ with no extra relations imposed)
in order to determine the algebraic structure governing the geometric
theory, the universality properties of the complex and the strength of the geometric complex relative to the various TQFTs applied to it. We also wish to know how to extract all the various TQFTs in a simple unified way out of the geometric complex, and how to make the full geometric complex computable efficiently.
\\

We start by classifying surfaces with boundary modulo the 4TU/S/T relations, and thus getting hold on the structure of the underlying category of the geometric complex. For this purpose we will introduce the \emph{genus generating operators} and use them to extend the ground ring. We prove a useful lemma regarding the free move of 2-handles between components of surfaces in $\Cobl$ which combines with the genus generating operators to produce the classification. This introduces the topological/geometric motivation for the rest of the thesis. Chapter \ref{chap:ClassificationQ} presents the classification when 2 is invertible in the ground ring and chapter \ref{chap:ClassificationZ} presents the general case over $\mathbb{Z}$.
\\

We continue to construct a reduction of the complex associated to a link.
We find an isomorphism of complexes that reduces the complex into one in a simpler
category. This category has only one object and the entire complex is composed of columns of that single object. The complex maps are matrices with monomial entries in one variable (a genus generating operator). Thus the complex is equivalent to one built from free modules over a polynomial ring in one variable. Chapter \ref{chap:ComplexReductionQ} presents these results when 2 is invertible and chapter \ref{chap:ComplexReductionZ} presents these results for the general case over $\mathbb{Z}$.
\\

The reduction theorem presents us with a ``pre-TQFT'' structure of purely
topological/geometric nature. It turns out that the underlying structure of the full geometric complex associated to a link (over $\mathbb{Z}$) is the same as the one given by the \emph{co-reduced} link homology theory using the following TQFT (over
$\mathbb{Z}[H]$):
\[
  \Delta_1: \begin{cases}
    v_+ \mapsto v_+\otimes v_- + v_-\otimes v_+  - H v_+\otimes v_+ &\\
    v_- \mapsto v_-\otimes v_- &
  \end{cases}
  \]
  \[
  m_1: \begin{cases}
    v_+\otimes v_-\mapsto v_- &
    v_+\otimes v_+\mapsto v_+ \\
    v_-\otimes v_+\mapsto v_- &
    v_-\otimes v_-\mapsto Hv_-
  \end{cases}
\]
\\

In chapter \ref{chap:TQFT} we explore the most general TQFTs that can be applied to the geometric complex in order to get a link homology theory. This will result in a theorem which states that the above co-reduced TQFT structure is the universal TQFT as far as
information in link homology is concerned. We also get (chapter \ref{chap:TQFT}) a new unified way of extracting all TQFTs directly from the geometric complex. This process is named \emph{promotion} and can be used to get new unfamiliar homology theories as well. These new homology theories contain limited (controlled) amount of information, coming only from surfaces up to a certain genus (in some sense resembles a perturbation expansion in the genus). We get an extrapolation between the standard Khovanov link homology theory and our universal TQFT. This chapter  (and the whole thesis) simplifies, completes and takes into a new direction some of the results of ~\cite{kho1}. It also generalizes some of the results of ~\cite{ba1} and completes it in some respects.
\\

Chapter \ref{chap:CompExt} discusses how to compute efficiently the universal theory. It also presents computational examples of the universal theory for many knots. We state various results on specific TQFTs by looking at the different promotions and parameter specifications of the universal theory. We relate the universal theory to other tools used in the field, mainly spectral sequences. The efficient computations allow us also to use the universal theory in order to simplify, explain and check various phenomenological statements regarding link homology. Chapter \ref{chap:CompExt} also shortly discusses various possible extensions of the techniques introduced in this thesis to various generalizations of the Khovanov link homology theory ($sl(3)$ link homology, open-closed and unoriented TQFTs). Finally, chapter \ref{chap:comments} discusses some topics related to the universal theory that are best left to the end.
\\

During my studies I have created two major works. The first was published in \cite{naot2}, and will not be discussed in this thesis. The second work was published in \cite{naot1}, and I base my thesis on this work. Chapters \ref{chap:ClassificationQ}, \ref{chap:ComplexReductionQ}, \ref{chap:ClassificationZ}, \ref{chap:ComplexReductionZ} and \ref{chap:TQFT} reproduce chapters from \cite{naot1}.
\\ 

%% file: figs/Main.pstex_t
\begin{picture}(0,0)%
\includegraphics{figs/Main.pstex}%
\end{picture}%
%
%
\setlength{\unitlength}{3079sp}%
\begingroup\makeatletter\ifx\SetFigFont\undefined%
\gdef\SetFigFont#1#2#3#4#5{%
  \reset@font\fontsize{#1}{#2pt}%
  \fontfamily{#3}\fontseries{#4}\fontshape{#5}%
  \selectfont}%
\fi\endgroup%
\begin{picture}(9654,4612)(210,-4268)
\put(9076,-4261){\makebox(0,0)[b]{\smash{\SetFigFont{12}{14.4}{\rmdefault}{\mddefault}{\itdefault}{\color[rgb]{0,0,0}$0$}%
}}}
\put(6376,-4261){\makebox(0,0)[b]{\smash{\SetFigFont{12}{14.4}{\rmdefault}{\mddefault}{\itdefault}{\color[rgb]{0,0,0}$-1$}%
}}}
\put(3676,-4261){\makebox(0,0)[b]{\smash{\SetFigFont{12}{14.4}{\rmdefault}{\mddefault}{\itdefault}{\color[rgb]{0,0,0}$-2$}%
}}}
\put(976,-4261){\makebox(0,0)[b]{\smash{\SetFigFont{12}{14.4}{\rmdefault}{\mddefault}{\itdefault}{\color[rgb]{0,0,0}$-3$}%
}}}
\put(976,-586){\makebox(0,0)[b]{\smash{\SetFigFont{8}{9.6}{\rmdefault}{\mddefault}{\updefault}{\color[rgb]{0,0,0}$(n_+,n_-)=(0,3)$}%
}}}
\end{picture}

%% file: figs/Mat.pstex_t
\begin{picture}(0,0)%
\includegraphics{figs/Mat.pstex}%
\end{picture}%
\setlength{\unitlength}{3947sp}%
\begingroup\makeatletter\ifx\SetFigFont\undefined%
\gdef\SetFigFont#1#2#3#4#5{%
  \reset@font\fontsize{#1}{#2pt}%
  \fontfamily{#3}\fontseries{#4}\fontshape{#5}%
  \selectfont}%
\fi\endgroup%
\begin{picture}(5268,1762)(967,-1636)
\put(1951,-376){\makebox(0,0)[lb]{\smash{\SetFigFont{10}{12.0}{\rmdefault}{\mddefault}{\updefault}{\color[rgb]{0,0,0}$G_{21}$}%
}}}
\put(1876,-571){\makebox(0,0)[lb]{\smash{\SetFigFont{10}{12.0}{\rmdefault}{\mddefault}{\updefault}{\color[rgb]{0,0,0}$G_{31}$}%
}}}
\put(1689,-181){\makebox(0,0)[lb]{\smash{\SetFigFont{10}{12.0}{\rmdefault}{\mddefault}{\updefault}{\color[rgb]{0,0,0}$G_{11}$}%
}}}
\put(4351,-1336){\makebox(0,0)[lb]{\smash{\SetFigFont{10}{12.0}{\rmdefault}{\mddefault}{\updefault}{\color[rgb]{0,0,0}$F_{23}$}%
}}}
\put(3601,-736){\makebox(0,0)[b]{\smash{\SetFigFont{12}{14.4}{\rmdefault}{\mddefault}{\updefault}{\color[rgb]{0,0,0}$\calO'_2$}%
}}}
\put(3601,-1336){\makebox(0,0)[b]{\smash{\SetFigFont{12}{14.4}{\rmdefault}{\mddefault}{\updefault}{\color[rgb]{0,0,0}$\calO'_2$}%
}}}
\put(3601,-136){\makebox(0,0)[b]{\smash{\SetFigFont{12}{14.4}{\rmdefault}{\mddefault}{\updefault}{\color[rgb]{0,0,0}$\calO'_1$}%
}}}
\put(1201,-436){\makebox(0,0)[b]{\smash{\SetFigFont{12}{14.4}{\rmdefault}{\mddefault}{\updefault}{\color[rgb]{0,0,0}$\calO''_1$}%
}}}
\put(1201,-1036){\makebox(0,0)[b]{\smash{\SetFigFont{12}{14.4}{\rmdefault}{\mddefault}{\updefault}{\color[rgb]{0,0,0}$\calO''_2$}%
}}}
\put(6001,-436){\makebox(0,0)[b]{\smash{\SetFigFont{12}{14.4}{\rmdefault}{\mddefault}{\updefault}{\color[rgb]{0,0,0}$\calO_1$}%
}}}
\put(6001,-1036){\makebox(0,0)[b]{\smash{\SetFigFont{12}{14.4}{\rmdefault}{\mddefault}{\updefault}{\color[rgb]{0,0,0}$\calO_2$}%
}}}
\put(2401,-1636){\makebox(0,0)[b]{\smash{\SetFigFont{12}{14.4}{\rmdefault}{\mddefault}{\updefault}{\color[rgb]{0,0,0}$G$}%
}}}
\put(4801,-1636){\makebox(0,0)[b]{\smash{\SetFigFont{12}{14.4}{\rmdefault}{\mddefault}{\updefault}{\color[rgb]{0,0,0}$F$}%
}}}
\put(4351,-286){\makebox(0,0)[lb]{\smash{\SetFigFont{10}{12.0}{\rmdefault}{\mddefault}{\updefault}{\color[rgb]{0,0,0}$F_{21}$}%
}}}
\put(4351,-736){\makebox(0,0)[lb]{\smash{\SetFigFont{10}{12.0}{\rmdefault}{\mddefault}{\updefault}{\color[rgb]{0,0,0}$F_{22}$}%
}}}
\end{picture}

%% file: ChapClassificationQ.tex
\chapter{Surfaces modulo the 4TU/S/T relations with 2 invertible}\label{chap:ClassificationQ}

The geometric complex, as explained in the introduction, is an invariant of links and tangles which takes values in the category $\Komh(\Mat(\Cobl))$. We wish to study this
category and hopefully reduce the geometric complex into a simpler complex in
a simpler category. This chapter is devoted to the
study of the underlying category $\Cobl$ (2-dimensional orientable
cobordisms between unions of circles) in the case when 2 is
invertible in the ground ring we work over ($\mathbb{Q}$ or
$\mathbb{Z}[\frac{1}{2}]$, for example). More specifically, we will classify surfaces modulo 4TU, S and T relations, which will give us the morphism groups of $\Cobl$.
\\

It is known ~\cite{ba1} that when 2 is invertible in the ground ring
we work over, the 4TU relation is equivalent to \emph{the neck
cutting relation}:
\[
NC ~relation:
~2\begin{array}{c}\includegraphics[height=7mm]{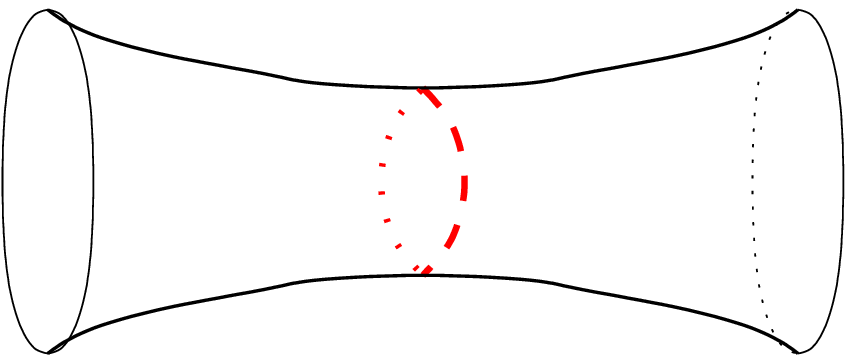}\end{array}
  =\begin{array}{c}\includegraphics[height=7mm]{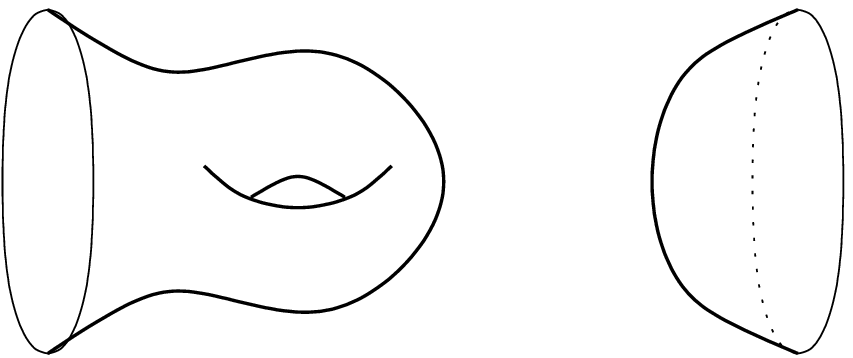}\end{array}
  +\begin{array}{c}\includegraphics[height=7mm]{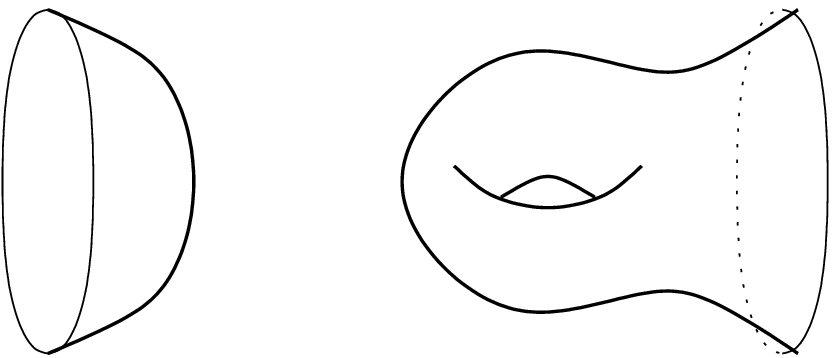}\end{array}
\]
We will use this relation throughout the chapter.
\\

\section{Notations}
We will try to keep discussion on a ``picture level'' as much as possible, yet we still need some formal notations. Let $\Sigma_g(\alpha_1,\alpha_2,\cdots)$
denote a surface with genus g and boundary circles
$\alpha_1,\alpha_2,\cdots$. A disconnected union of such surfaces
will be denoted by
$\Sigma_{g_1}(\alpha_1,\cdots)\Sigma_{g_2}(\beta_1,\cdots)$. If the
genus or the boundary circles are not relevant for the argument at
hand, they will be omitted. Whenever we have a piece of surface
which looks like $\epsg{10mm}{CNN}$ we will call it \emph{a neck}.
If cutting a neck separates the component into 2 disconnected
components then it will be a \emph{separating neck}, if not then it
is a \emph{non-separating neck} which means it is a part of \emph{a
handle} on the surface. A handle on the surface always looks locally
like $\eps{5mm}{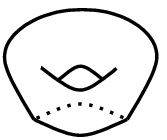}$, and by \emph{2-handle} on the surface we mean
a piece of surface which looks like $\eps{5mm}{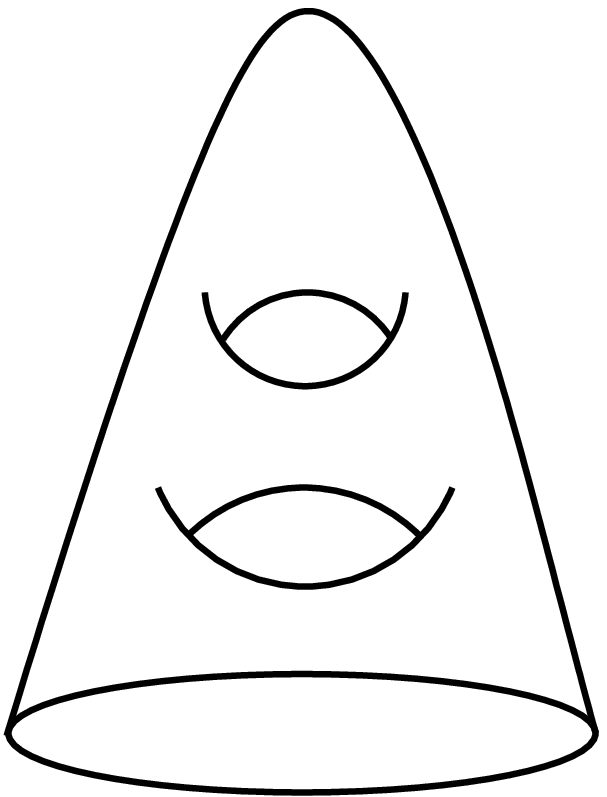}$.

\section{The 2-handle lemma}
We start with proving a lemma that will become useful in classifying
surfaces modulo the 4TU/S/T relations.

\begin{lemma}
In $\Cobl$ 2-handles move freely between components of a surface.
I.e. modulo the 4TU relation, a surface with a 2-handle on one of
its connected components is equal to the same surface with the
2-handle removed and glued on a different component (see the picture
below the proof for an example).
\end{lemma}

\begin{proof}
The proof is an application of the neck cutting relation (NC), which
follows from the 4TU relation (over any ground ring). We look at a
piece of surface with a handle and two necks on it which
looks like $\eps{20mm}{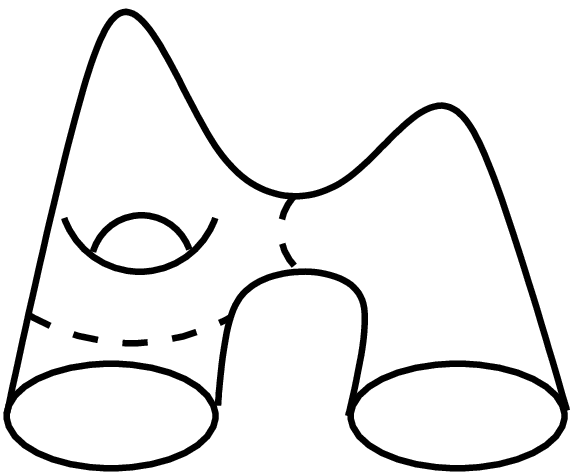}$. The rest of the surface continues outside the bottom circles and is not drawn. Applying the NC relation to the vertical dashed neck and to the horizontal dashed neck gives the following two equalities:
\[\eps{7mm}{2handle}\eps{7mm}{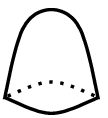} +
\eps{7mm}{vm}\eps{7mm}{vm} = 2\cdot\eps{15mm}{2lemma} =
\eps{7mm}{vm}\eps{7mm}{vm} + \eps{7mm}{vp}\eps{7mm}{2handle}
\]
We get $\eps{7mm}{2handle}\eps{7mm}{vp} =
\eps{7mm}{vp}\eps{7mm}{2handle}$. Since these are the local parts of
\emph{any} surface the lemma is proven.
\qed
\end{proof}

\vspace{4mm} Notice that this lemma applies to any ground ring we
work over (in $\Cobl$), and does not require 2 to be invertible.

\vspace{3mm} \emph{Example:} The following equality holds in
$\Cobl$:
$\eps{9mm}{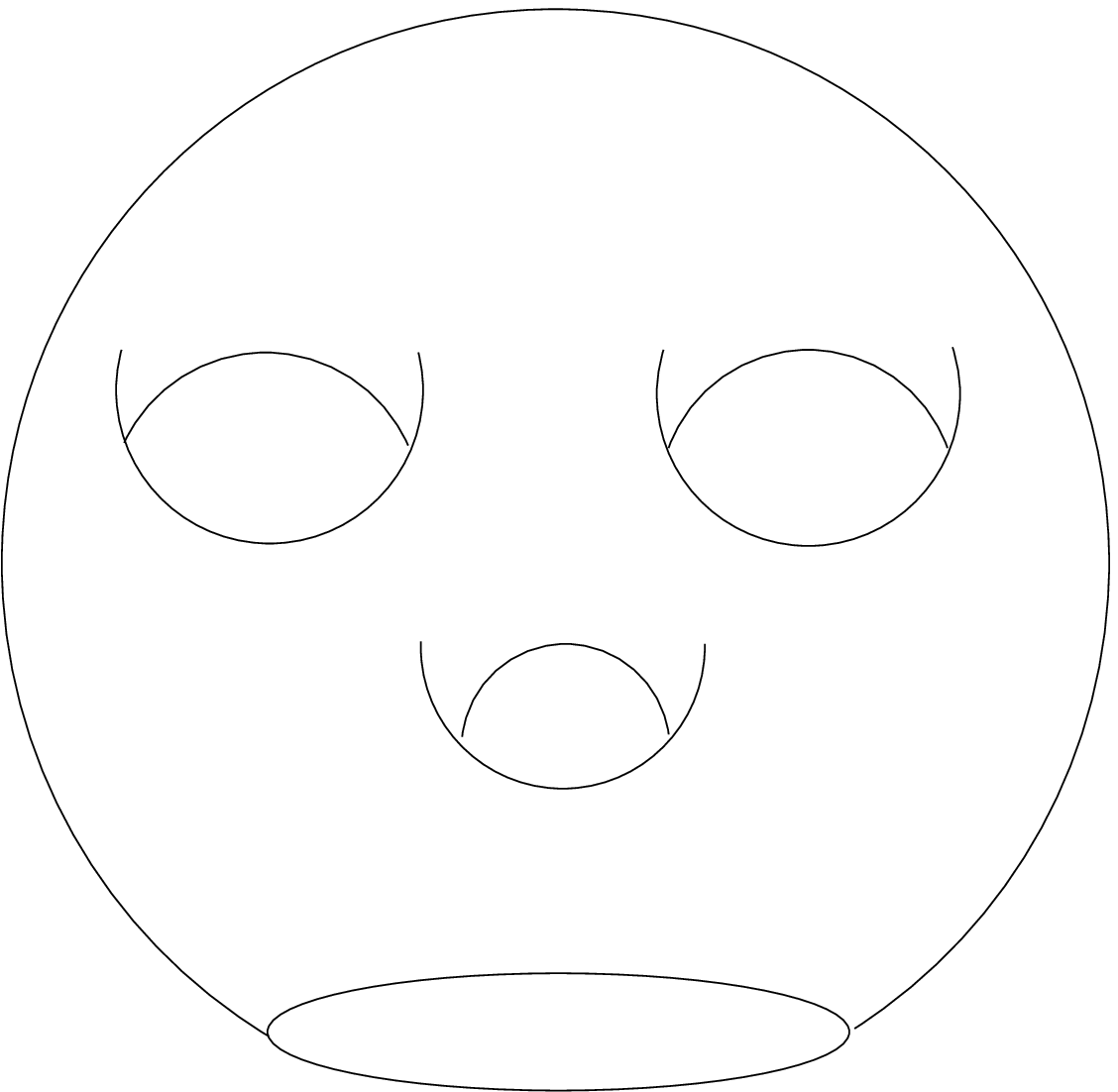}\eps{7mm}{2handle}$=$\eps{7mm}{vm}\eps{10mm}{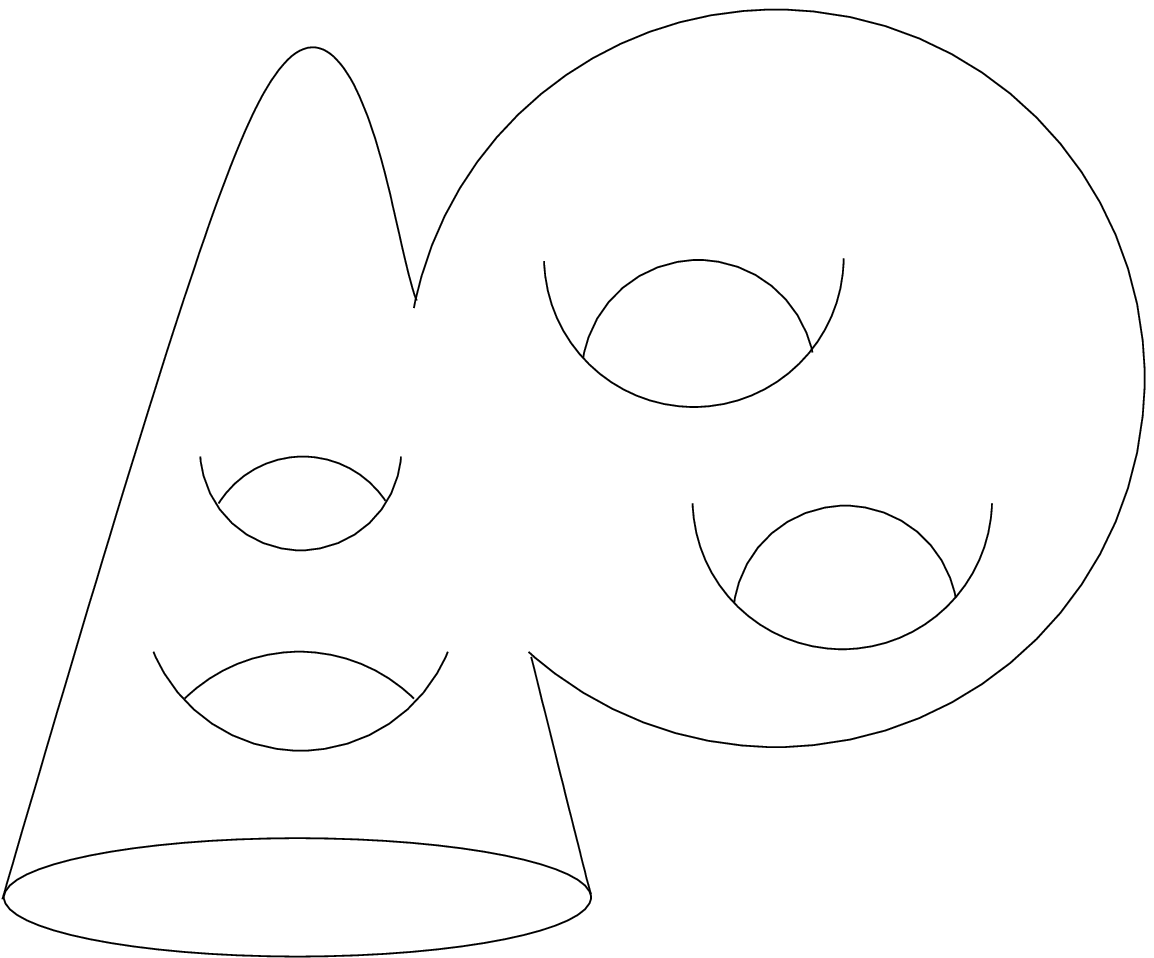}$=
$\eps{10mm}{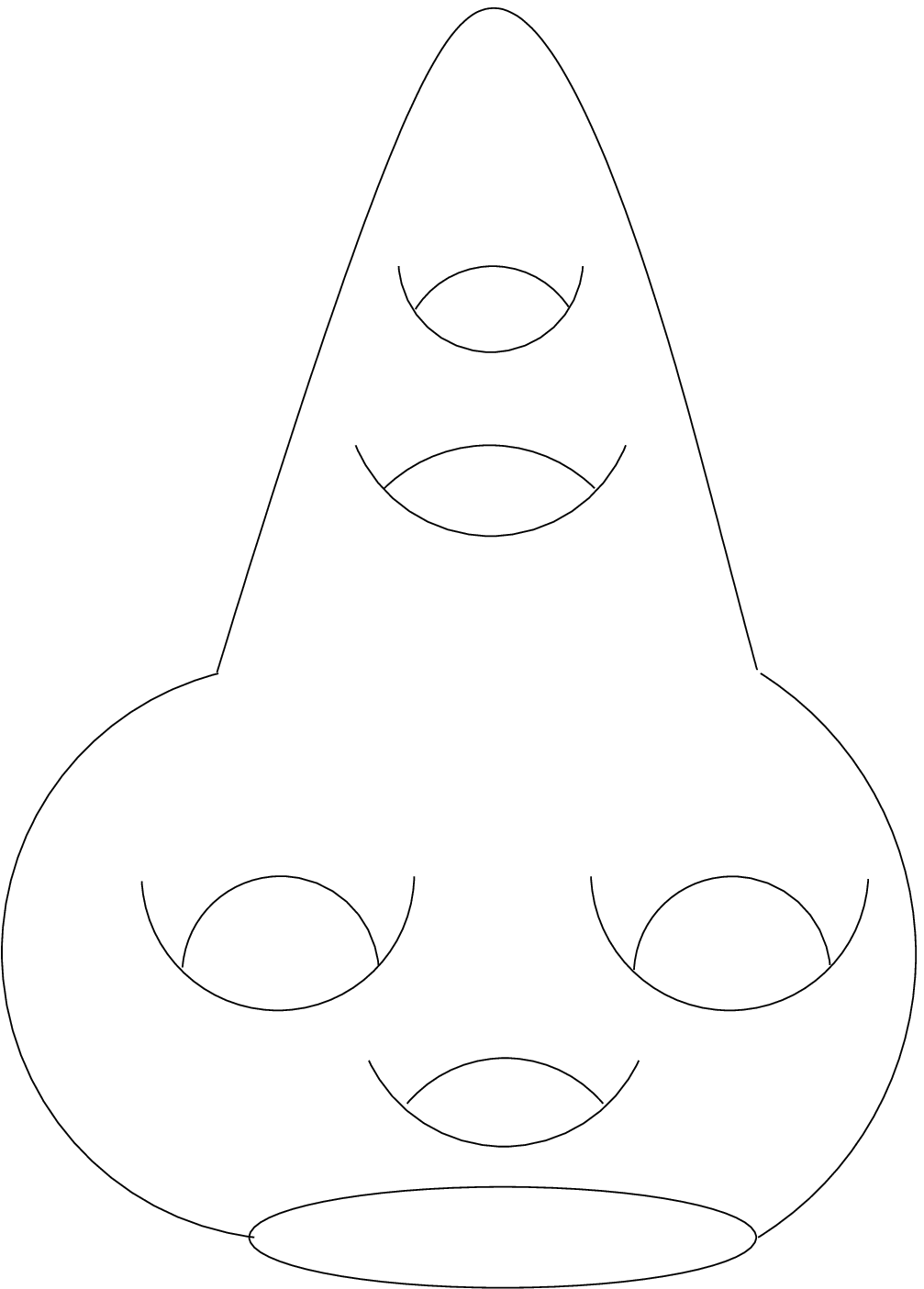}\eps{7mm}{vp}$.

\vspace{3mm}

The 2-handle lemma allows us to reduce the classification problem to
a classification of surfaces with \emph{at most} one handle in each
connected component. This is done by extending the ground ring $R$
to $R[T]$ with $T$ a ``global'' operator (i.e. acts \emph{anywhere}
on the surface) defined as follows :

\begin{definition}
\emph{The 2-handle operator}, denoted by $T$, is the operator that
glues a 2-handle somewhere on a surface (anywhere).
\end{definition}

\emph{Examples}: $T\cdot\Sigma_{g_1}(\alpha)\Sigma_{g_2}(\beta) =
\Sigma_{g_1+2}(\alpha)\Sigma_{g_2}(\beta) =
\Sigma_{g_1}(\alpha)\Sigma_{g_2+2}(\beta)$ or $T\cdot
\eps{6mm}{2handle}=\eps{10mm}{4handle}$.
\\

The 2-handle operator is given degree $-4$ according to the degrees conventions of \cite{ba1}, thus keeping all the elements of the theory graded.
\\

\begin{remark}
The notation for the 2-handle operator $T$ should not be confused with the T relation. Whenever an equation of surfaces appear (like in the example above and below), or a ground ring is extended, $T$ always mean the 2-handle operator and got nothing to do with the T relation.
\end{remark}

Observe that since the torus equals 2 in $\Cobl$ (the T relation), multiplying a surface $\Sigma$ with $2T$ (twice the 2-handle operator) is equal to taking the union of $\Sigma$ with the genus 3 surface: \[2T \cdot \Sigma = T \cdot \Sigma \cup \epsg{5mm}{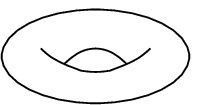} = \Sigma \cup \epsg{9mm}{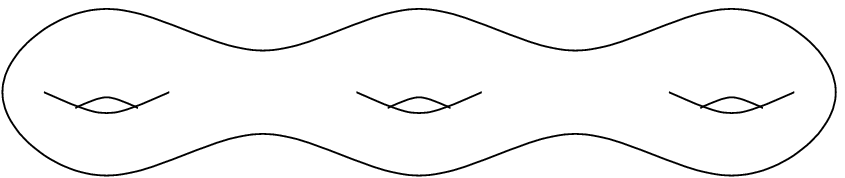}\]
\\
Note that the 2-handle operator does not operate on the empty
cobordism. Since 2 is invertible the operation of the 2-handle
operator can be naturally presented as the union with
$\epsg{9mm}{g3}/2$, thus allowing meaningful operation on the empty
cobordism. We will adopt this presentation of the 2-handle operator $T$, and call it
\emph{the genus-3 presentation.}

\section{Classification of surfaces modulo 4TU/S/T over $\mathbb{Z}[\frac{1}{2}]$}

\begin{proposition}
Over the extended ring $\mathbb{Z}[\frac{1}{2},T]$, where $T$ is the
2-handle operator in the genus-3 presentation, the morphism groups
of $\Cobl$ (surfaces modulo S, T and 4TU relations) are generated
freely by unions of surfaces with exactly one boundary component and
genus 0 or 1 ($\eps{5mm}{vp},\eps{5mm}{vm}$), together with the
empty cobordism.
\end{proposition}

In other words, each morphism set from $n$ circles to $m$ circles is a
free module of rank $2^{n+m}$ over $\mathbb{Z}[\frac{1}{2},T]$,
where $T$ is the 2-handle operator in the genus-3 presentation. The rest of this chapter consists of the proof to this proposition.
\\

\begin{proof}
Given a surface $\Sigma$ with (possibly) few connected components
and with (possibly) few boundary circles we can apply the 2-handle
lemma and reduce it to a surface that has at most genus 1 (i.e one
handle) in each connected component. This is done by extending the
ground ring we work over to $\mathbb{Z}[\frac{1}{2},T]$, where $T$
is the 2-handle operator.
\\

When 2 is invertible the 4TU relation is equivalent to the NC
relation, and by dividing the relation equation by 2 we see that we
can cut any neck and replace it by the right hand side of the
equation :
$\begin{array}{c}\includegraphics[height=5mm]{figs/CNN.eps}\end{array}
  =\frac{1}{2}\begin{array}{c}\includegraphics[height=5mm]{figs/CNL.eps}\end{array}
  +\frac{1}{2}\begin{array}{c}\includegraphics[height=5mm]{figs/CNR.eps}\end{array}$.
We will mod out this relation and reduce any surface into a combination of free generating surfaces. We will need to show that our reduction uses only the NC relation and is well defined (that is reducing the NC relation to zero).
\\

Let us first look at the classification of surfaces with no boundary
at all. Using the NC relation it is not hard to show that $2\Sigma_{2g}=0$ and
with 2 invertible $\Sigma_{2g}=0$. Thus we can consistently extend
the ground ring with the operator $T$ and get: $\Sigma_g =
T^{\frac{g}{2}}\cdot\epsg{5mm}{S}=0$ for $g$ even, and $\Sigma_g =
T^{\frac{g-1}{2}}\cdot\epsg{5mm}{g1}=2T^{\frac{g-1}{2}}$ for $g$
odd. Thus every closed surface is reduced to a ground ring element.
Allowing the empty cobordism (``empty surface''), every closed
surface will reduce to it over $\mathbb{Z}[\frac{1}{2},T]$ when we
use the genus-3 presentation of $T$. The closed surfaces morphism
set is thus isomorphic to $\mathbb{Z}[\frac{1}{2},T]$ and through
the genus-3 presentation acts on other morphism sets in the
category, creating a module structure.
\\

Now, when cutting a separating neck in a component, one reduces a
component into two components at the expense of adding handles. Thus
using only the NC relation, over $\mathbb{Z}[\frac{1}{2},T]$, every
surface reduces to a sum of surfaces whose components contain
\emph{exactly} one boundary circle and at most one handle.
\\

The reduction can be put in a formula form, considered as a map
$\Psi$ taking a surface $\Sigma_g(\alpha_1,\ldots,\alpha_n)$ into a
combination of generators (a sum of unions of surfaces with one
boundary circle and genus 0 or 1). This map is an isomorphism of
modules over the closed surfaces ring
($\mathbb{Z}[\frac{1}{2},T]$-modules):

\vspace{3mm} \emph{The Surface Reduction Formula Over
$\mathbb{Z}[\frac{1}{2},T]$:}
\begin{equation}
\Psi : \Sigma_g(\alpha_1,\ldots,\alpha_n) \mapsto \frac{1}{2^{n-1}}
\cdot \sum_{i_1,\ldots,i_n=0,1}{odd(g+z) \cdot
T^{\lfloor\frac{g+z}{2}\rfloor}
\cdot\Sigma_{i_1}(\alpha_1)\cdots\Sigma_{i_n}(\alpha_n)}
\end{equation}
$z$ is the number of 0's among $\{i_1,\ldots,i_n\}$. $odd(g+z)=1$ if $g+z$ is odd, and $0$ otherwise. If $n=0$ it is understood that the entire sum is dropped and replaced by $odd(g) \cdot T^{\lfloor\frac{g}{2}\rfloor}$ times the empty surface. The
formula extends naturally to unions of surfaces.
\\

A close look at the definition will show that indeed this map uses
only the NC relation to reduce the surface, in other words, the
difference between the reduced surface and the original surface is
only a combination of NC relations. (i.e. $\Psi(\Sigma) - \Sigma =
\sum{NC}$). To see this, one need to show that there is at least one
way of using NC relations to get the formula. This can be done, for
instance, by taking the boundary circles one by one, and each time
cut the unique neck separating \emph{only} this circle from the
surface.
\\

Last, we turn to the issue of possible relations between the
generators. These can be created whenever we apply $\Psi$ to two
surfaces which are related by NC relation and get two different
results. Thus the generators are free from relations if the
reduction formula is well defined, i.e. respects the NC relation.
This is true, as the following argument proves $\Psi(NC)=0$. Look at
the NC relation, $\Sigma_{g_1+1}(\alpha)\Sigma_{g_2}(\beta) +
\Sigma_{g_1}(\alpha)\Sigma_{g_2+1}(\beta) -
2\Sigma_{g_1+g_2}(\alpha,\beta)$, where
$\alpha=\{\alpha_1,\ldots,\alpha_n\}$ and
$\beta=\{\beta_1,\ldots,\beta_m\}$ denote families of boundary
circles. Apply the formula everywhere:
$\frac{1}{2^{n+m-2}}\cdot[odd(g_1+z_1+1)odd(g_2+z_2)
T^{\lfloor\frac{g_1+z_1+1}{2}\rfloor}
T^{\lfloor\frac{g_2+z_2}{2}\rfloor} +odd(g_1+z_1)odd(g_2+z_2+1)
T^{\lfloor\frac{g_1+z_1}{2}\rfloor}T^{\lfloor\frac{g_2+z_2+1}{2}\rfloor}
-
odd(g_1+g_2+z_1+z_2)T^{\lfloor\frac{g_1+z_1+g_2+z_2}{2}\rfloor}]\cdot
\sum{\Sigma_{i_1}(\alpha_1)\cdots\Sigma_{i_n}(\alpha_n)\Sigma_{j_1}(\beta_1)\cdots\Sigma_{j_m}(\beta_m)}$
where $z_1$ is the number of 0's among $\{i_1,\ldots,i_n\}$, $z_2$
is the number of 0's among $\{j_1,\ldots,j_m\}$ and the sum is over
these sets of indices as in formula (1). Going over all 4 options of
$(g_1+z_1,g_2+z_2)$ being (odd/even,odd/even) one can see the result
is always 0.
\\

Formula (1) gives a map from surfaces to combinations of generators.
This map was shown to be well defined ($\Psi(NC)=0$), and uses only
the NC relations ($\Psi(\Sigma) - \Sigma = \sum{NC}$), thus
completing the proof.
\qed
\end{proof} 

%% file: ChapComplexReductionQ.tex
\chapter{The structure of the geometric complex over $\mathbb{Q}$}\label{chap:ComplexReductionQ}

Armed with the classification of the morphisms of $\Cobl$ we are ready to simplify the complex associated to a link in the case where 2 is
invertible. We introduce an isomorphism of objects in $\Cobl$, known
as \emph{delooping}, that extends to an isomorphism in
$\Komh\Mat(\Cobl)$. This will reduce the complex associated to a
link into a complex taking values in a simpler category. One of the
consequences is uncovering the underlying algebraic structure of the
geometric complex. We will use the boxed surface notation in which
boxed surface has the geometric meaning of half a handle:
\[\eps{7mm}{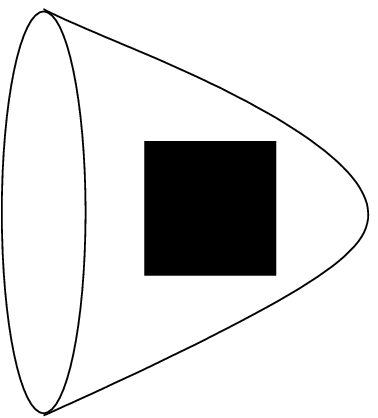}=\frac{1}{2}\eps{7mm}{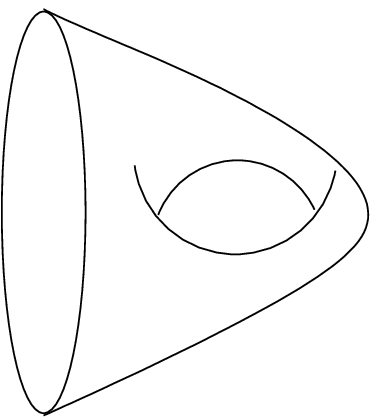}\]

\section{The reduction theorem over $\mathbb{Q}$}

\begin{theorem}
The complex associated to a link, over $\mathbb{Q}$ (or any ground
ring $R$ with 2 invertible), is equivalent to a complex built from
the category of free $\mathbb{Q}[T]$-modules ($R[T]$-modules
respectively).
\end{theorem}

This theorem is a corollary of the following proposition, which we
will prove first :

\begin{proposition}
In $\Mat(\Cobl)$, when 2 is invertible, the circle object is
isomorphic to a column of two empty set objects. This isomorphism
extends (taking degrees into account) to an isomorphism of complexes
in $\Komh\Mat(\Cobl)$, where the circle objects are replaced by
empty sets together with the appropriate induced complex maps.
\end{proposition}

\begin{definition}
The isomorphism that allows us to eliminate circles in exchange for
a column of empty sets is called \emph{delooping} and is given by:
\begin{center}
$\eps{30mm}{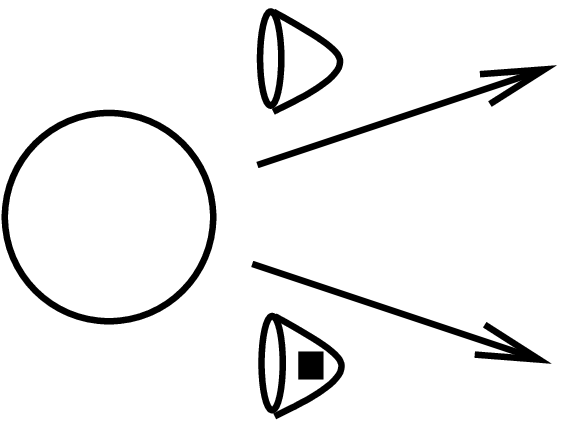}$ $ \left[ \begin{array}{c}
\emptyset \{-1\}\\
\\
\\
\emptyset\{+1\}
\end{array} \right] $
$\eps{30mm}{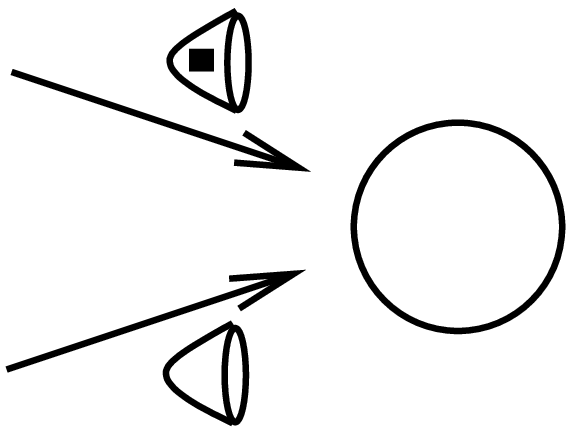}$
\end{center}
\end{definition}

\vspace{3mm}

\begin{proof}(Prop.)\\
We need to check that the compositions of the morphisms above are
the identity (thus making the circle isomorphic to the column of two
empty sets). One direction (from empty sets to empty sets) just
follows from the S and T relations (a boxed sphere equals 1), and
the other direction (from circle to circle) just follows from the NC
relation
$\begin{array}{c}\includegraphics[height=5mm]{figs/CNN.eps}\end{array}
  =\frac{1}{2}\begin{array}{c}\includegraphics[height=5mm]{figs/CNL.eps}\end{array}
  +\frac{1}{2}\begin{array}{c}\includegraphics[height=5mm]{figs/CNR.eps}\end{array}$.\\

This proves that the circle object is isomorphic to a column of two
empty set objects in $\Mat(\Cobl)$. The extension to the category $\Komh\Mat(\Cobl)$ is now straight forward by replacing every appearance of a circle by a column of two empty sets (with the degrees shifted as shown in the picture, maintaining the grading of the theory). The induced maps are the composed maps of the isomorphism above with the original maps of the complex.

\qed
\end{proof}

\vspace{3mm}

\begin{proof}(Thm.)\\
The theorem follows from the proposition. The only objects that are
left in the complex are the empty sets (in various degree shifts).
The induced morphisms in the reduced complex are just the set of
closed surfaces, isomorphic to $\mathbb{Q}[T]$ in the genus-3
presentation, thus making the complex isomorphic to one in the
category of free $\mathbb{Q}[T]$-modules. Thus the complex will look like columns of empty sets (the single object) and maps which are matrices with entries in $\mathbb{Q}[T]$.
\qed
\end{proof}

\vspace{3mm} \emph{Remark:} The morphisms
appearing in the complex associated to a link are actually matrices with $\mathbb{Z}[\frac{1}{2},T]$ entries. Furthermore, we can work with boxed surfaces and re-normalize the 2-handle operator by dividing it by 4, to get that the maps of the complex associated to a link are always $\mathbb{Z}[T]$ matrices. Degree considerations will tell us immediately that the entries are only monomials in $T$.

\vspace{3mm} This theorem extends the simplification reported in
~\cite{ba2,ba4}. There, it was done in order to compute the
``standard'' Khovanov homology, which in the context of the
geometric formalism means imposing one more relation which states
that any surface having genus higher than 1 is set equal to zero
(see also section 9 in ~\cite{ba1}). In our context this means
setting the action of the operator $T$ to zero. When we set $T=0$
the complex reduces to columns of empty sets with maps being integer
matrices.

\section{The algebraic structure underlying the complex
over $\mathbb{Q}$} Now that we simplified the complex significantly (as far as possible objects) it is interesting to see what can we learn about the underlying
algebraic structure of the complex and about the complex maps in terms of the empty sets.
\\

As a first consequence we can see that the circle (the basic object in the complex associated to a link) carries a structure composed of two copies of another basic object (the empty set). It has to be understood that this decomposition is a direct consequence of the 4TU/S/T relations and it is an intrinsic structure of the geometric
complex that reflects these relations. This already suggests that any algebraic structure which respects 4TU/S/T and put on the circle (for instance a Frobenius algebra, to give a TQFT functor) will have to factor into direct sum of two identical copies of a smaller structure shifted by two degrees (two copies of the ground ring of the algebra, for the same example). This shows, perhaps in a slightly more fundamental way, the result given in terms of Frobenius extensions in ~\cite{kho1}.
\\

Let us denote $\emptyset\{-1\}$ by $v_-$ (comparing to ~\cite{ba1}) or X
(comparing to ~\cite{kho1}) and $\emptyset\{+1\}$ by $v_+$ or 1
(comparing accordingly). We also define a re-normalized 2-handle
operator $t$ to be equal the 2-handle operator $T$ divided by 4 (thus $t=\epsg{7mm}{g3}/8$ in the genus-3 presentation). Note that in the
category we are working right now this is just a notation, the
actual objects are the empty sets and we did not add any extra
algebraic structure. We use the tensor product symbol to denote
unions of empty sets and thus keeping track of degrees (remember
that all elements of the theory are graded).
\\

Take the pair of pants map $\eps{7mm}{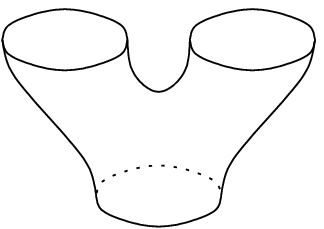}$ between
$\bigcirc\bigcirc$ and $\bigcirc$, and look at the two complexes
$\bigcirc\bigcirc\xrightarrow{pants}\bigcirc$ and
$\bigcirc\xrightarrow{pants}\bigcirc\bigcirc$. Reduce these
complexes using the reduction theorem (i.e. replace the circles with empty
sets columns and replace the complex maps with the induced maps). We
denote by $\Delta_2$ the following composition:

\begin{center}
$ \Delta_2: \left[ \begin{array}{c}
v_-\\
v_+
\end{array} \right]
\xrightarrow{isomorphism} \bigcirc
\xrightarrow{\eps{7mm}{InvertedPOP}} \bigcirc\bigcirc
\xrightarrow{isomorphism} \left[
\begin{array}{c}
v_-\otimes v_-\\
v_-\otimes v_+\\
v_+ \otimes v_-\\
v_+ \otimes v_+
\end{array} \right]
$
\end{center}

We denote by $m_2$ the composition in the reverse direction. By
composing surfaces and using the re-normalized 2-handle operator we get the maps:

\[
  \Delta_2: \begin{cases}
    v_+ \mapsto \left[ \begin{array}{c} v_+\otimes v_- \\ v_-\otimes v_+ \end{array} \right] & \\
    v_- \mapsto \left[ \begin{array}{c} v_-\otimes v_- \\ t v_+\otimes v_+ \end{array} \right] &
  \end{cases}
  \qquad
  m_2: \begin{cases}
    v_+\otimes v_-\mapsto v_- &
    v_+\otimes v_+\mapsto v_+ \\
    v_-\otimes v_+\mapsto v_- &
    v_-\otimes v_-\mapsto tv_+.
  \end{cases}
\]

\vspace{3mm} The maps technically have no linear structure, these are just
$2\times4$ matrices of morphisms in the geometrical category $\Mat(\Cobl)$.
Nonetheless this is exactly the pre-algebraic structure of the generalized Lee TQFT (see section 9 of \cite{ba1}). See ~\cite{lee1} for the non generalized case (where $t=1$) The above TQFT is denoted $\calF_3$ in ~\cite{kho1}. It is important to note that so
far in our category we did not apply any TQFT or any other functor to get this structure, these maps are not multiplication or co-multiplication in an algebra. It is all done intrinsically within our category, and as we will see later it imposes obvious restrictions on the most general type of functors one can apply to the geometric
complex. One can say that the geometric structure over $\mathbb{Q}$ has the underlying structure of this specific \emph{``pre-TQFT''}.

%% file: ChapClassificationZ.tex
\chapter{Surfaces modulo the 4TU/S/T relations over $\mathbb{Z}$}\label{chap:ClassificationZ}

When 2 is not invertible the neck cutting relation is not equivalent to
the 4TU relation and we need to make sure we use the ``full
version'' of the 4TU relation in reducing surfaces into generators.
Still there is a simpler equivalent version of the 4TU relation
which involves only 3 sites on the surface:
\begin{eqnarray*}
  3S1:\qquad& \eps{7cm}{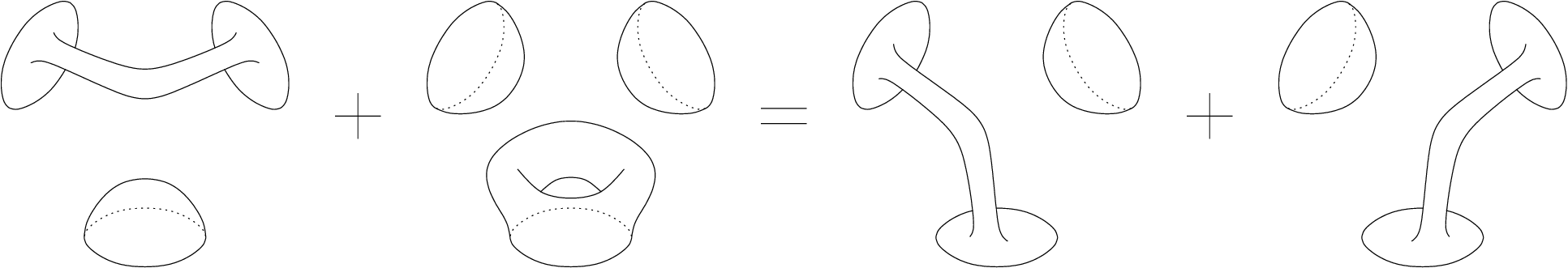} \\
  3S2:\qquad& \displaystyle
  \sum_{\parbox{2cm}{\begin{center}\scriptsize\rm
    $0^\circ$, $120^\circ$, $240^\circ$ \newline
    rotations
  \end{center}}}
  \left(\eps{1.2cm}{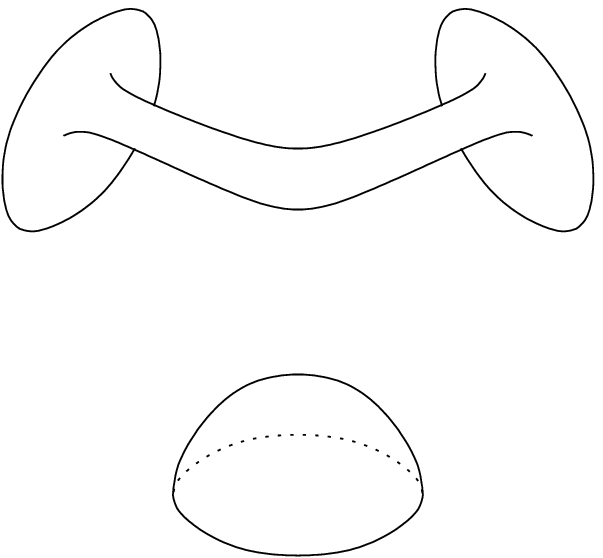}-\eps{1.2cm}{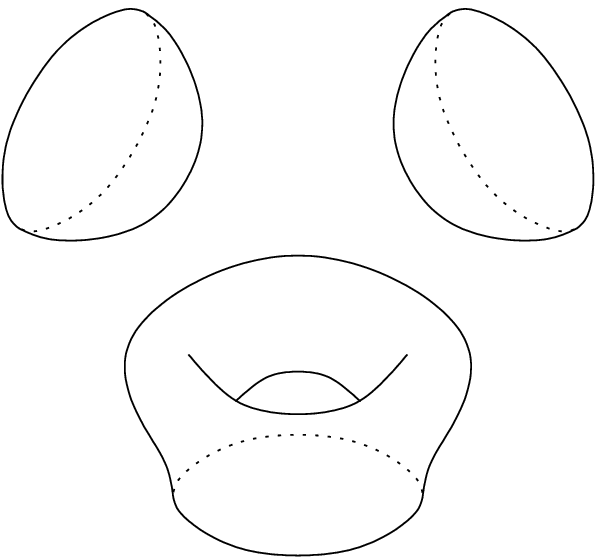}\right) = 0
\end{eqnarray*}

The version that will be most convenient for us is the $3S1$
relation. We will use the 2-handle lemma (which does not depend on
the invertibility of 2), extend the ground ring again, and identify
a complete set of generators.

\begin{definition}
Assume a surface has at least one boundary circle. Choose one
boundary circle. The component that contains the specially chosen
boundary circle is called \emph{the special component}. All other
components will be called \emph{non-special}.
\end{definition}

\begin{definition}
The \emph{special 1-handle operator}, denoted $H$, is the operator
that adds a handle to the special component of a surface.
\end{definition}

Note that $T=H^2$, though $H$ acts ``locally'' (acts \emph{only} on
the special component) and $T$ acts ``globally'' (anywhere on the
surface). We extend our ground ring to $\mathbb{Z}[H]$. The entire theory is still graded with $H$ given degree $-2$.
\\

\begin{remark}
From now on we will assume that all surfaces have at least one
boundary component, thus there is always a special component, and
the action of $H$ is well defined (by choosing a special component).
\end{remark}

\begin{proposition}
Over the extended ring $\mathbb{Z}[H]$ (after a choice of special
boundary circle) the morphisms of $\Cobl$ with one of the
source/target objects non empty (i.e. surfaces modulo 4TU/S/T
relations with at least one boundary circle) are generated freely by
surfaces which are composed of a genus zero special component with
any number of boundary circles on it and zero genus non-special
components with exactly one boundary circle on them. $H$ is the
special 1-handle operator.
\end{proposition}

An example for a generator would be $\eps{20mm}{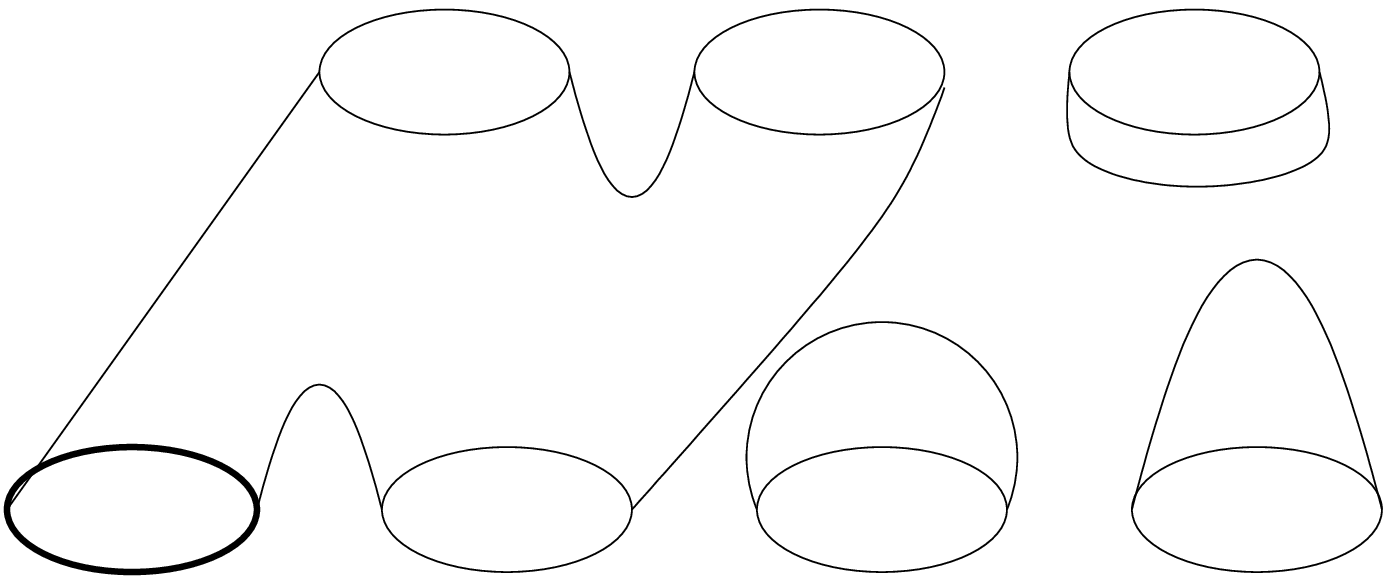}$ (the special
circle is at the bottom left).
\\

Another example would be the ``Shrek surface'' :
$\eps{20mm}{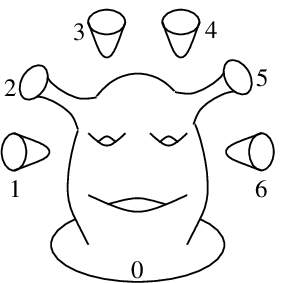}$. This surface has a special boundary circle
marked with the number zero (Shrek's neck), 6 other boundary circles
(number 2 and 5 belong to the special component -- the head) and 3
handles on the special component. Over the extended ring, this
surface is generated by the ``Shrek shadow'' $\eps{20mm}{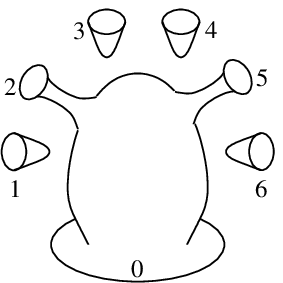}$
and equals to $H^3\eps{20mm}{shrekh}$.
\\

\begin{remark} \label{remark:Hlinear}
In order for the composition of surfaces (morphisms) to be
$H$-linear, and thus respect the $\mathbb{Z}[H]$-module structure,
one needs to restrict to a category where all the morphisms preserve
the special component, i.e. we always have the special boundary
circles of the two composed morphisms in the same component. As we
will see in the next chapter, this happens naturally in the context of
link homology.
\end{remark}

The rest of this chapter is devoted to the somewhat long proof of the proposition
above.
\\

\begin{proof}

We start with a couple of toy models. If all the surface components
involved in the $3S2$ (or $3S1$) relation have one boundary circle in total, then the $3S2$ (or $3S1$) relation is equivalent to NC relation, and further more up to the 2-handle lemma it is trivially satisfied. This makes the classification of surfaces with only one boundary component simple. Extend the ground ring to $\mathbb{Z}[T]$ and the generators will be $\eps{0.5cm}{vm}$ and $\eps{0.5cm}{vp}$. We can now go one step further and use the ``local'' 1-handle operator $H$. Over $\mathbb{Z}[H]$ there is only
one generator $\eps{0.5cm}{vp}$.
\\

If the surface components involved have only two boundary circles in total, then the $3S2$ (or $3S1$) relation is again equivalent to the NC relation. We can classify now all surfaces with 2 boundary circles modulo the 4TU relation over $\mathbb{Z}$. Let us extend the ground ring to $\mathbb{Z}[T]$, thus reducing to components with
genus 1 at most. Then, we choose one of the boundary circles (call
it $\alpha$), and use the NC relation to move a handle from the
component containing the other boundary circle (call it $\beta$) to
the component containing $\alpha$ (i.e.
$\Sigma_1(\beta)\Sigma_g(\alpha) = 2\Sigma_g(\beta,\alpha) -
\Sigma_0(\beta)\Sigma_{g+1}(\alpha)$). This leaves us with the
following generators over $\mathbb{Z}[T]$: $\eps{1cm}{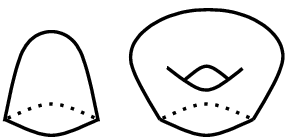}$,
$\eps{5mm}{vp}\eps{5mm}{vp}$, $\eps{10mm}{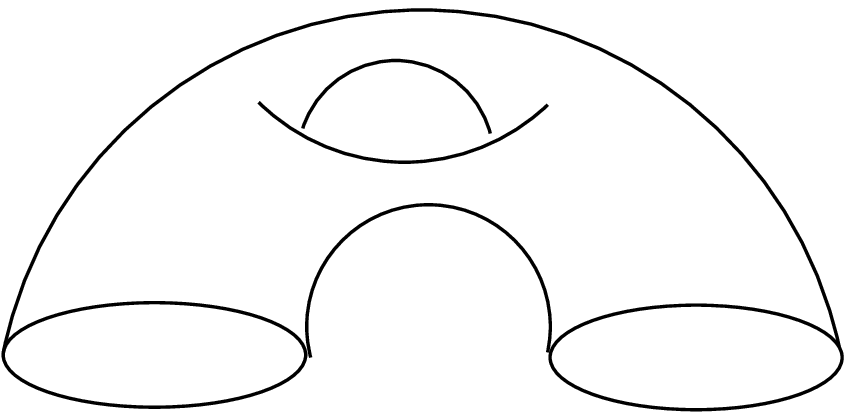}$ and
$\eps{1cm}{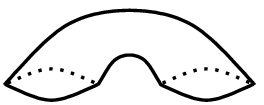}$ ($\alpha$ is drawn on the right). This asymmetry
in the way the generating set looks like is caused by the symmetry
braking in the way the NC relation is applied (we chose the special
circle $\alpha$, on the right). On the other hand it allows us again
to replace the ``global'' 2-handle operator $T$, with the ``local''
1-handle operator $H$, which operates only on the component
containing $\alpha$. Thus, over $\mathbb{Z}[H]$ we only have the
following generators: $\eps{5mm}{vp}\eps{5mm}{vp}$ and
$\eps{1cm}{MM12P3}$. This set is symmetric again, and the asymmetry
hides within the definition of $H$.
\\

The place where the difference between the NC and $3S2$ (or $3S1$)
relations comes into play is when three different components with
boundary on each are involved. We would use the $3S1$ relation as a ``neck cutting'' relation for the neck in the upper part of $\eps{10mm}{3S2}$ (by replacing it with the other 3 elements in the $3S1$ equation). The way the relation is applied is not
symmetric (not even visually), while the upper 2 sites are interchangeable the lower site plays a special role. This hints that this scheme of classification would be easier when choosing a special component of the surface. This is done by choosing a special
boundary circle as in definition 4.1.

\textbf{Cutting non separating necks over $\mathbb{Z}$.} Assume
first that we only use $\eps{10mm}{3S2}$ in the $3S1$ relation to
cut a non-separating neck (i.e. the upper two sites remain in the same connected component after the cut), and we take the lower site to be on the special component. Now, whenever we have a handle in a non-special component (dashed line on the left component in the picture below) we can move it using the $3S1$ relation to the
special component (the component on the right) at the price of
adding surfaces in which the non-special components become special: $\eps{20mm}{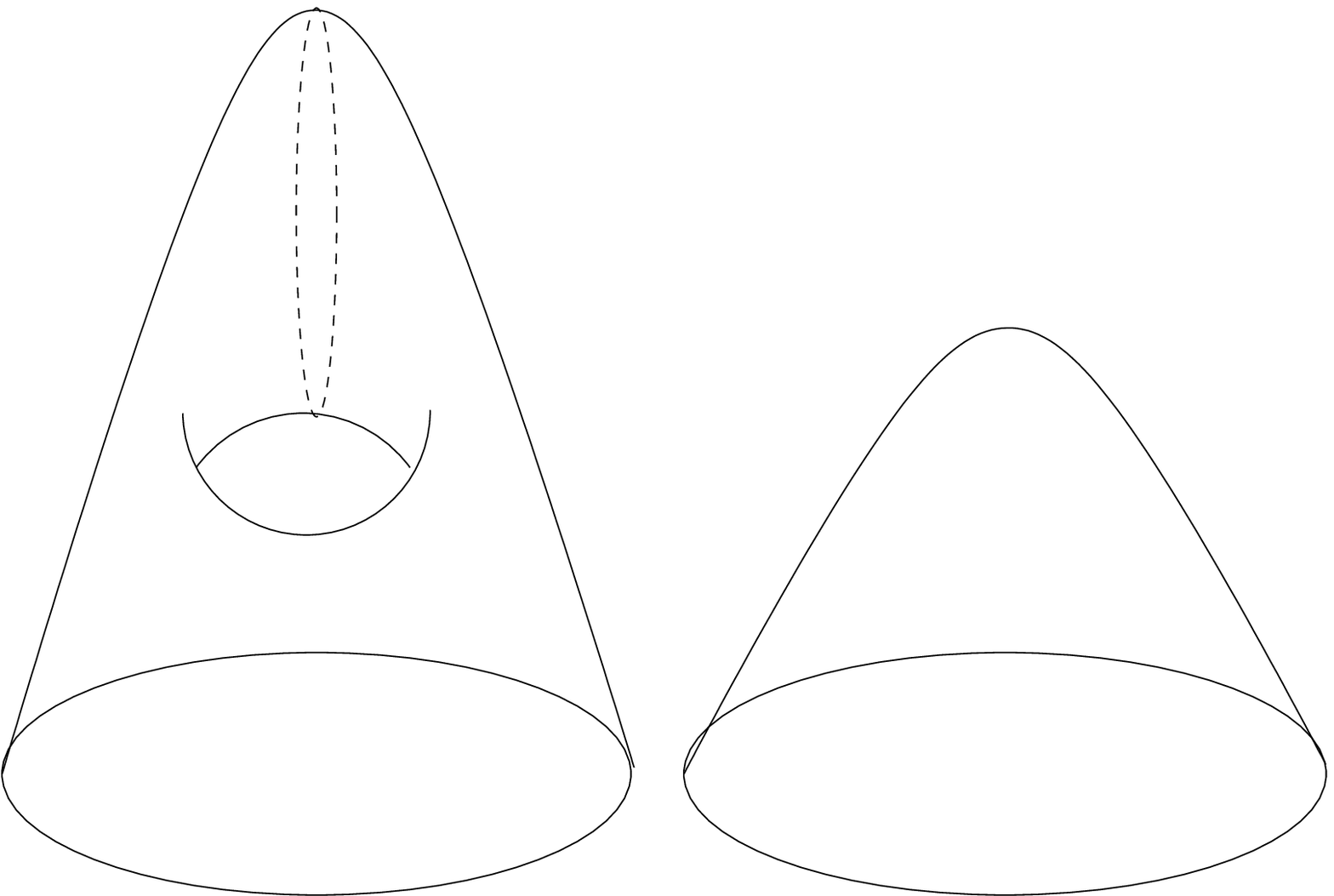} = \eps{20mm}{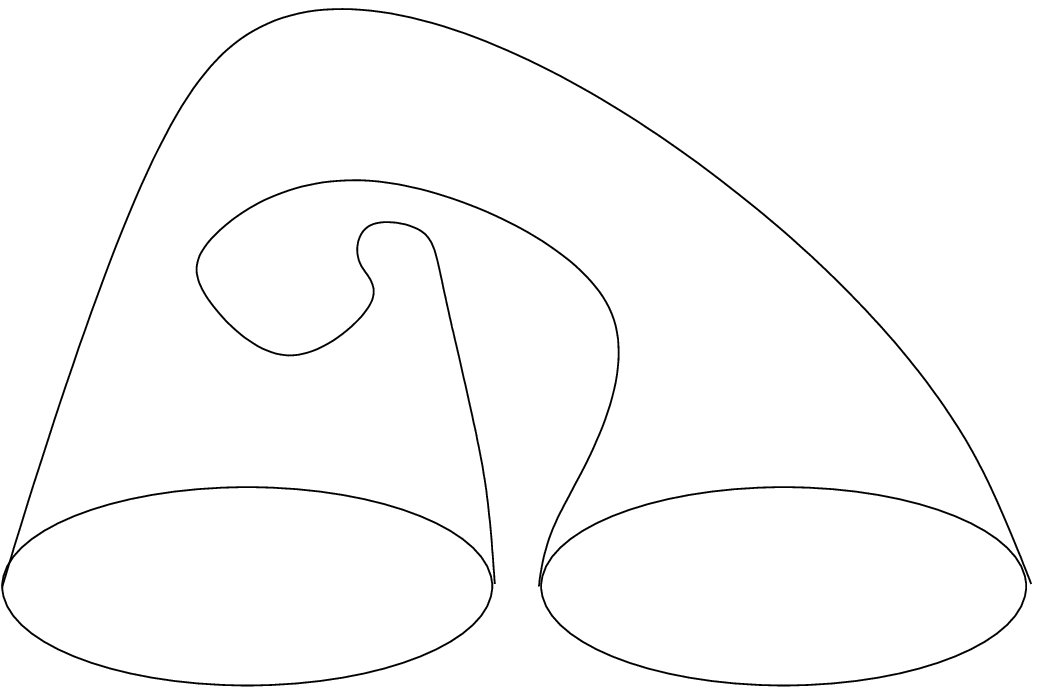} + \eps{20mm}{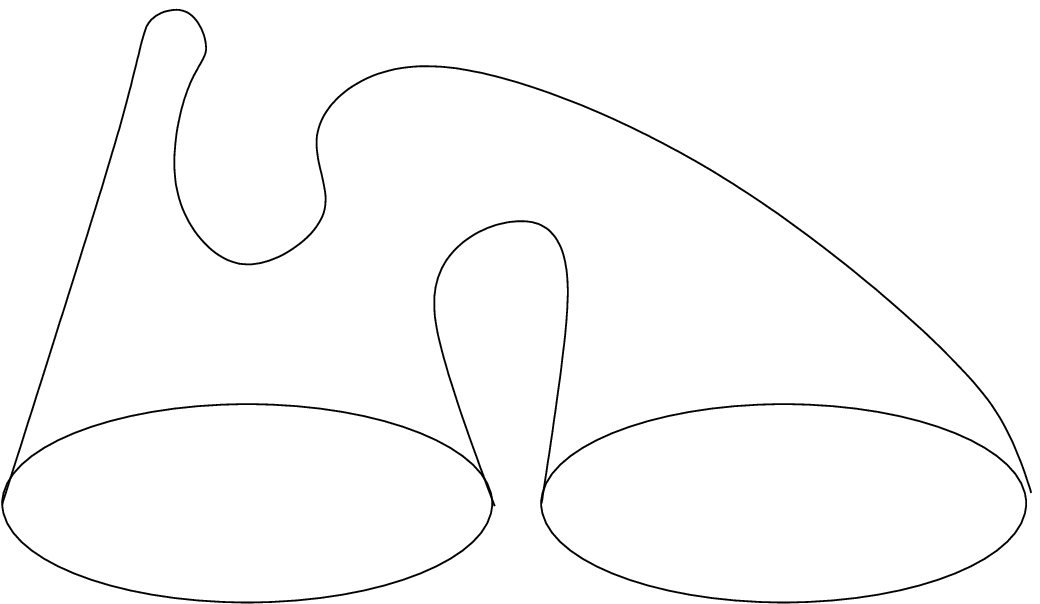} -
\eps{20mm}{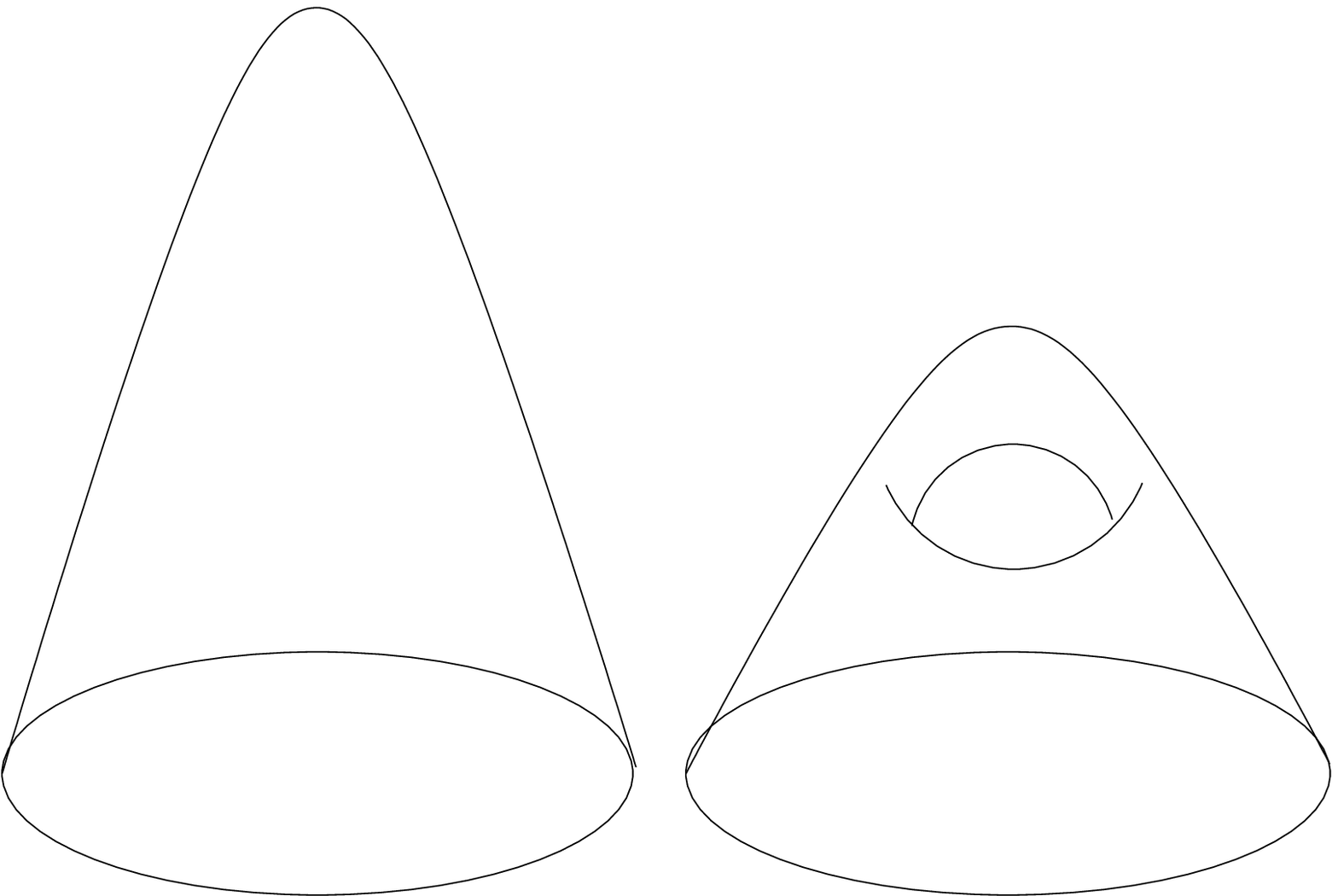}$. After this reduction we are left with surfaces involving handles only on the special component. The rest of the components are of genus zero (with any amount of boundary circles on them). For example $\eps{20mm}{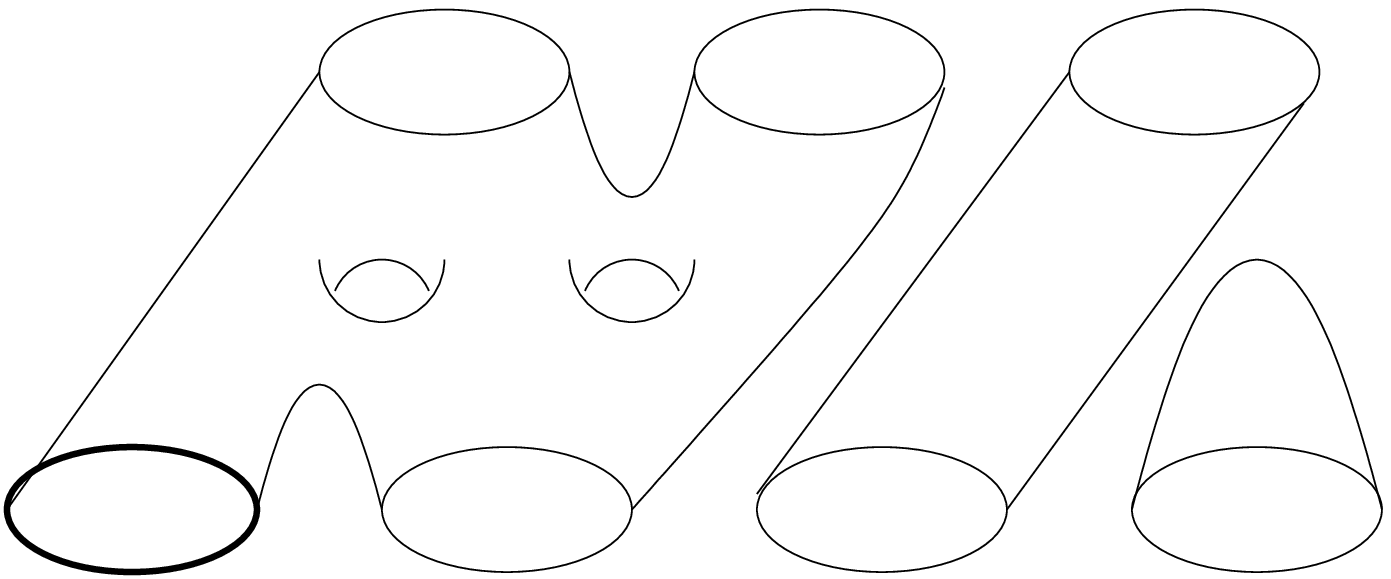}$, where the special circle is at the bottom left.
\\

\textbf{Cutting separating necks over $\mathbb{Z}$.} After the stage described above we use the $3S1$ relation to cut separating necks (the two upper sites in $\eps{10mm}{3S2}$ will be on two disconnected components after the cut). We will use the relation with the upper two sites on a non-special component and the lower site on the special component.
The result is separating the non-special component into two at the price of adding a handle to the special component as well as adding surfaces in which the non-special component become special (the other summands of the $3S1$ relation). Thus, whenever we have a non-special component that has more than one boundary circle, we can reduce it to a sum of components that are either special or have exactly one boundary circle. An
example of a summand in a reduced surface is $\eps{20mm}{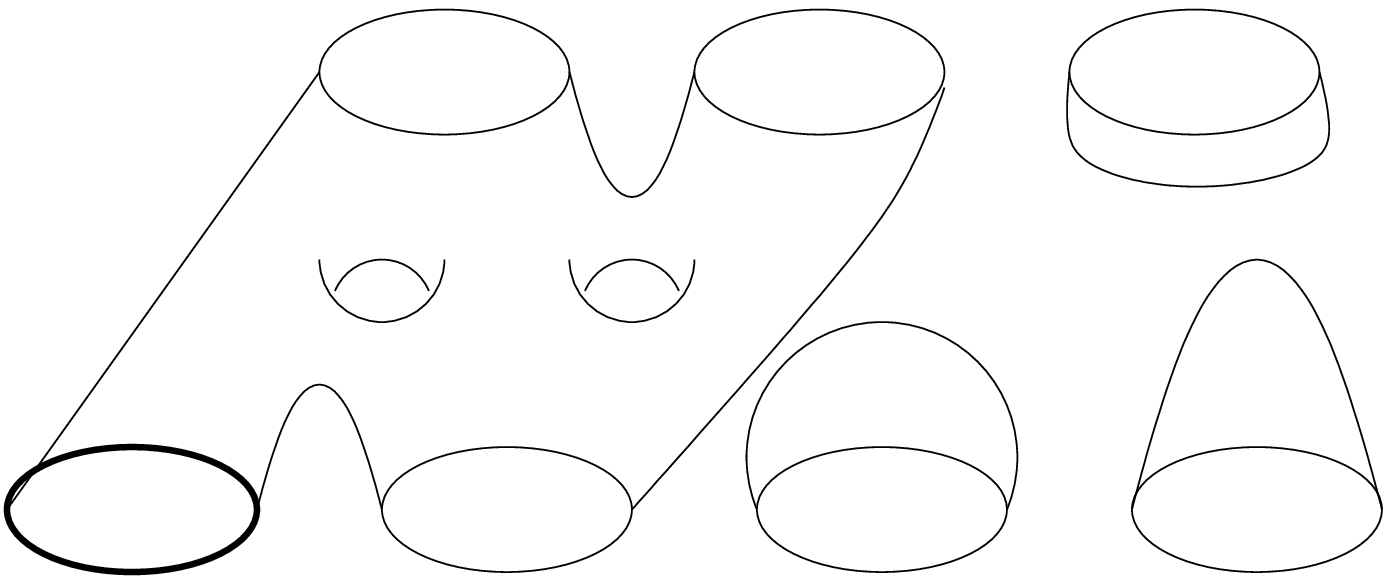}$, where the special circle is at the bottom left.
\\

\textbf{Ground ring extension and surface reduction formula over
$\mathbb{Z}$.} We now extend our ground ring to $\mathbb{Z}[H]$. We will
put all the above discussion into one formula for reducing any
surface into a combination of generators (considered as a map $\Psi$). Let $\alpha$
denote a family of boundary circles $\{\alpha_1,\ldots,\alpha_n\}$ and let $S$ denote the special boundary circle. We will use the notation $\beta\in 2^\alpha$ for any subset $\{\alpha_{i_1},\ldots,\alpha_{i_k}\}$ of $\alpha$, and denote by $\{\alpha_{i_{k+1}},\ldots,\alpha_{i_n}\}$ its complement in $\alpha$. $odd(g)$ is the function that equals $1$ if $g$ is odd and $0$ otherwise.
\\

\emph{The Surface Reduction Formula Over $\mathbb{Z}$}:\\

\emph{Special component reduction:}
\begin{equation}
\Sigma_g(S,\ldots)=H^g\cdot\Sigma_0(S,\ldots)
\end{equation}

\emph{Non-special component genus reduction ($g\geq1$):}
\begin{equation}\label{formula:first}
\Sigma_0(S,\ldots)\Sigma_g(\alpha) = 2\cdot odd(g)\cdot H^{g-1}
\Sigma_0(S,\ldots,\alpha) + (-1)^g H^g
\Sigma_0(S,\ldots)\Sigma_0(\alpha)
\end{equation}

\emph{Non-special component neck reduction ($n\geq2$):}
\begin{equation} \label{formula:last}
\Sigma_0(S,\ldots)\Sigma_0(\alpha) =
\sum_{\beta\in2^\alpha,\beta\neq\alpha}{(-1)^{n-k-1}H^{n-k-1}
\Sigma(S,\ldots,\beta) \Sigma_0(\alpha_{i_{k+1}}) \cdots
\Sigma_0(\alpha_{i_n})}
\end{equation}

\vspace{3mm} Formula \ref{formula:last} is plugged into the right hand side of formula \ref{formula:first} for full reduction. The above formulas extend naturally to unions of surfaces (just iterate the use of the formula). Thus we get a reduction of any surface into generators. The previous discussion shows that this formula uses the $3S1$ relation only (i.e. $\Psi(\Sigma)-\Sigma=\sum{3S1}$). This can also be proven directly by induction.
\\

The proof is finalized by taking the formula and observing that it is a map from surfaces to combinations of generators (by definition), it uses only the 4TU ($3S1$) relations (proved by the discussion above) and these generators are free (i.e. the map is well defined). The last part can be proven by a direct computation (applying the formula to all sides of a 4TU relation to get zero) or by applying a TQFT to the generators. By choosing a TQFT that respects the 4TU relations (like the standard Khovanov TQFT $\mathbb{Q}[X]/X^2$) one can show that it separates these surfaces (i.e. sends them to independent module maps). Let us mention that the reason for $\Psi(4TU)=0$ is coming from the fact that iterative application of the $3S1$ relation to a 4TU relation is zero as shown by the picture below (the special component is on top):

$\eps{10mm}{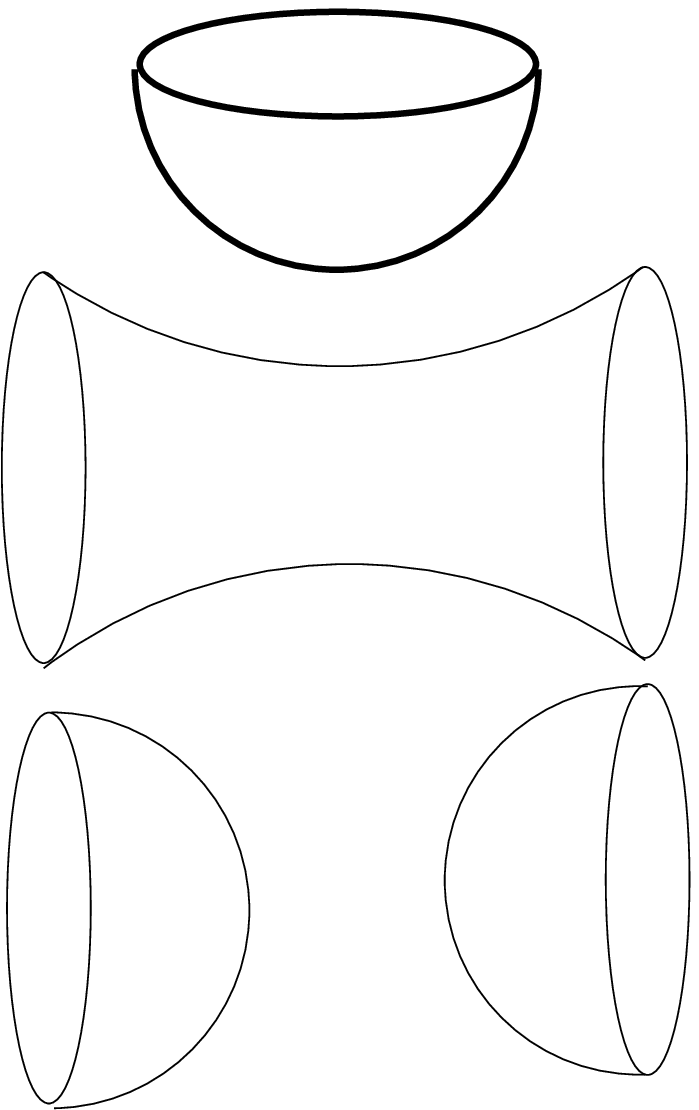} + \eps{10mm}{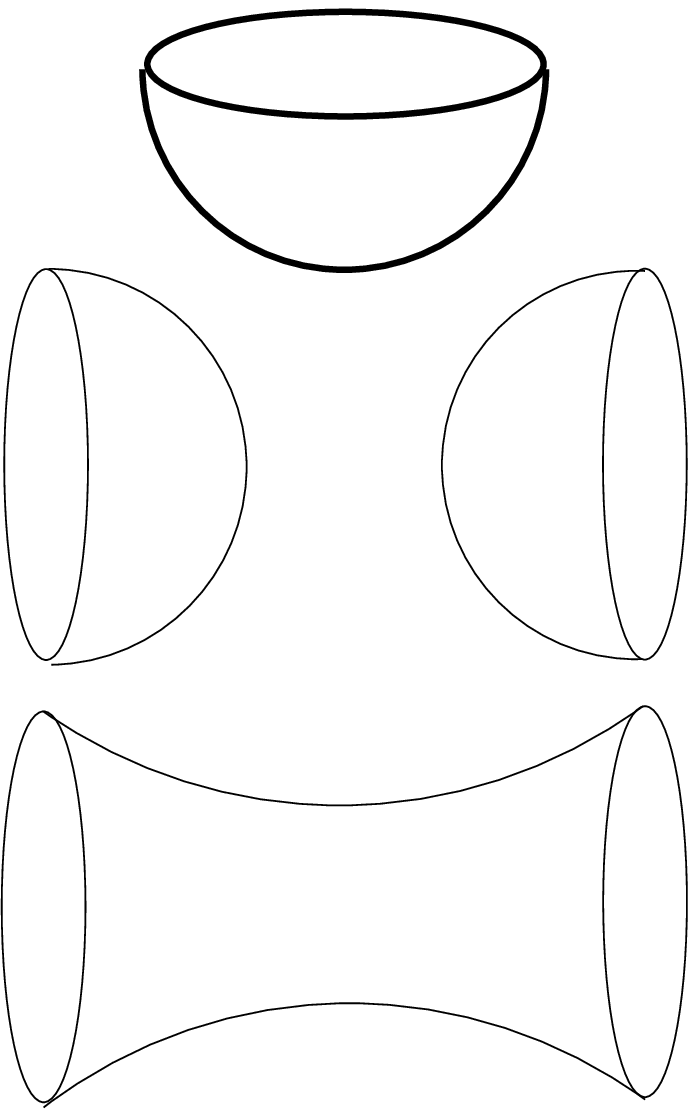} - \eps{14mm}{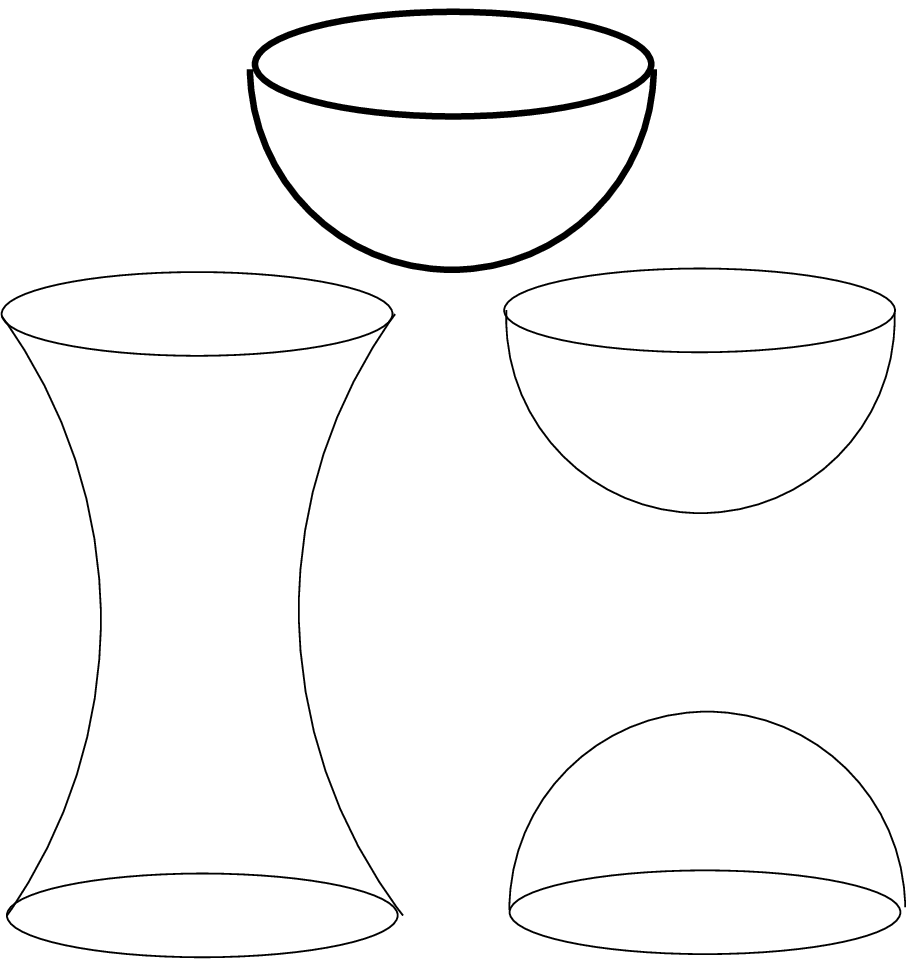}
- \eps{14mm}{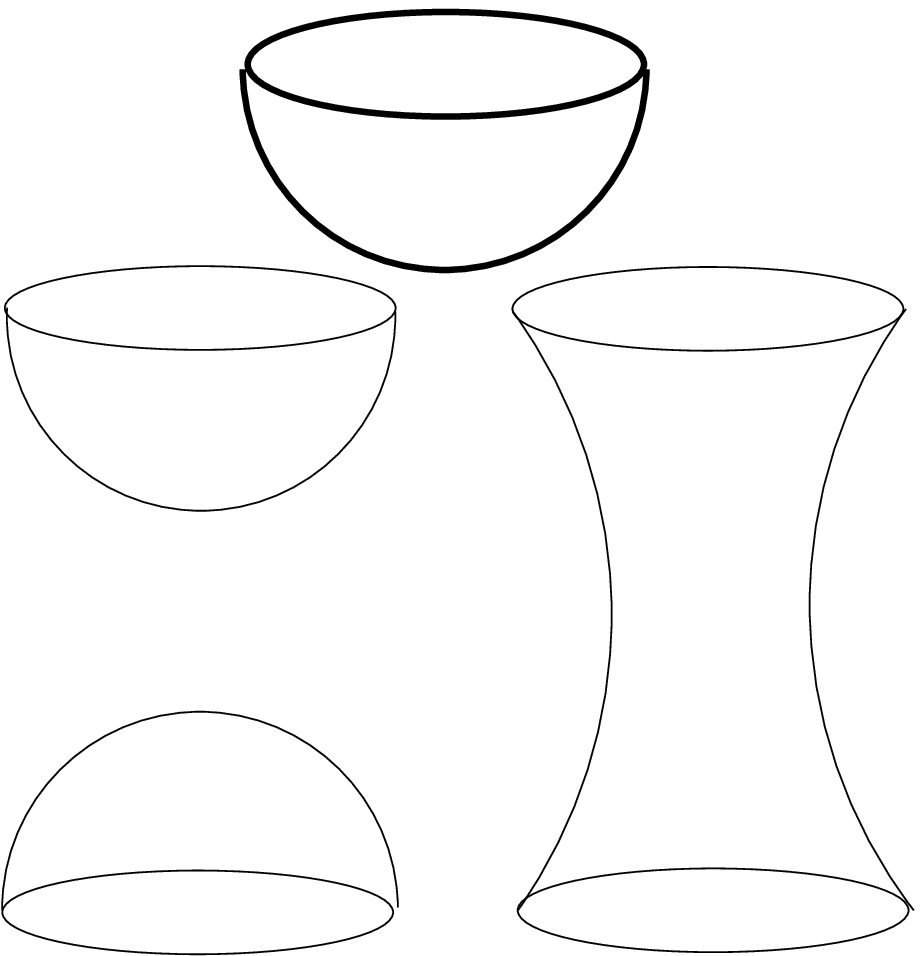} \mapsto \{ \eps{10mm}{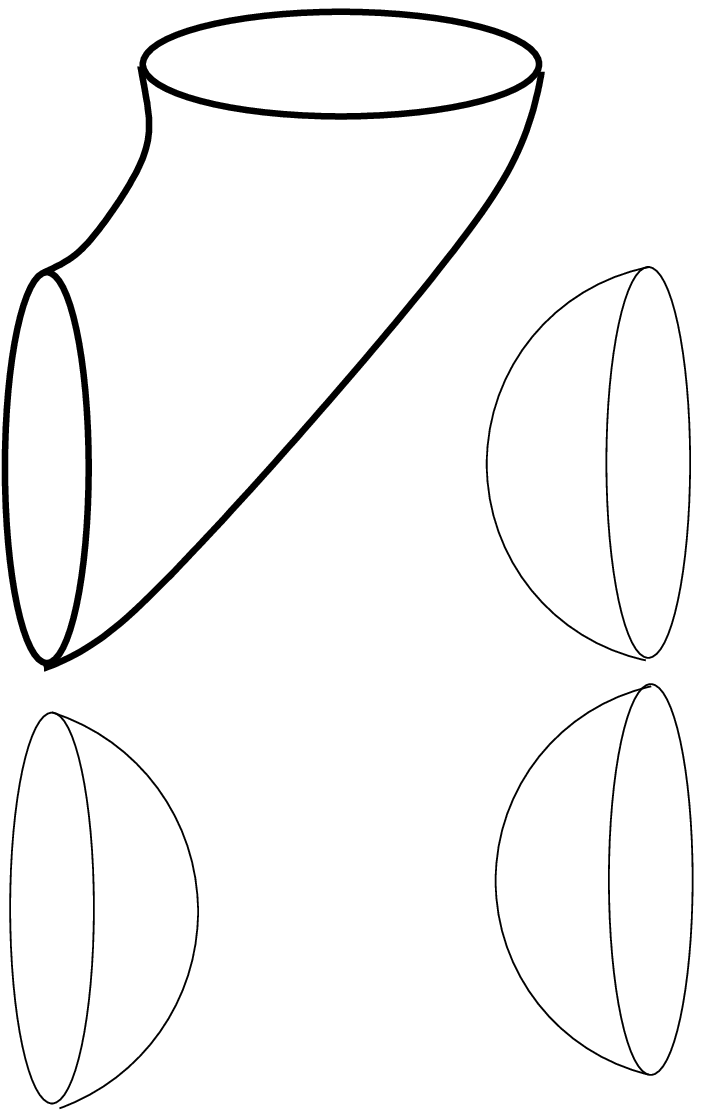} +
\eps{10mm}{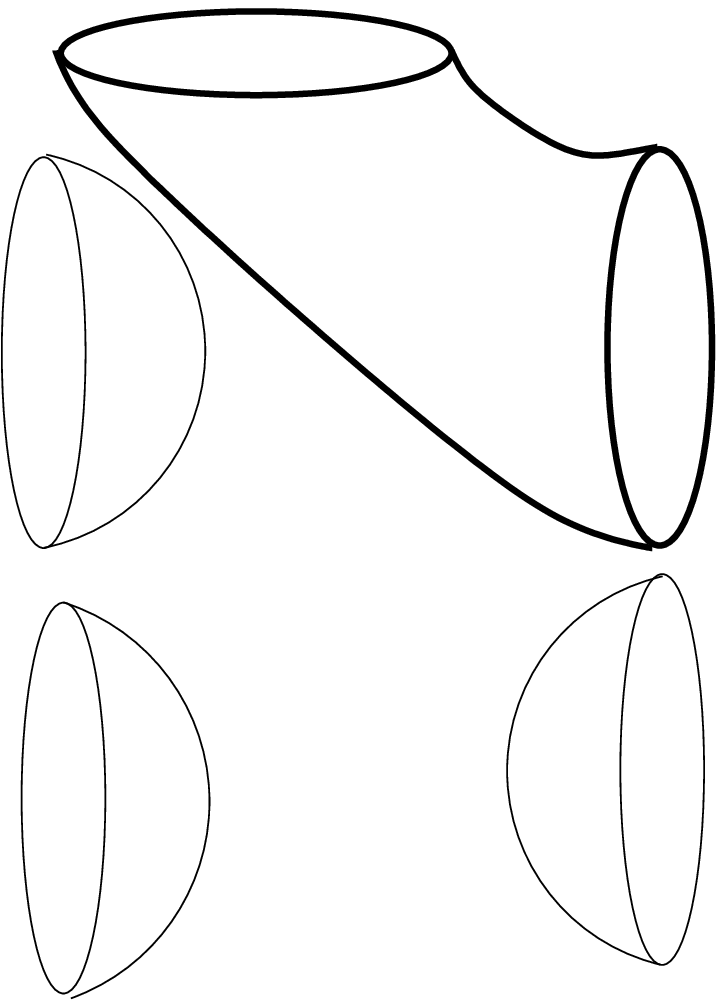} - \eps{10mm}{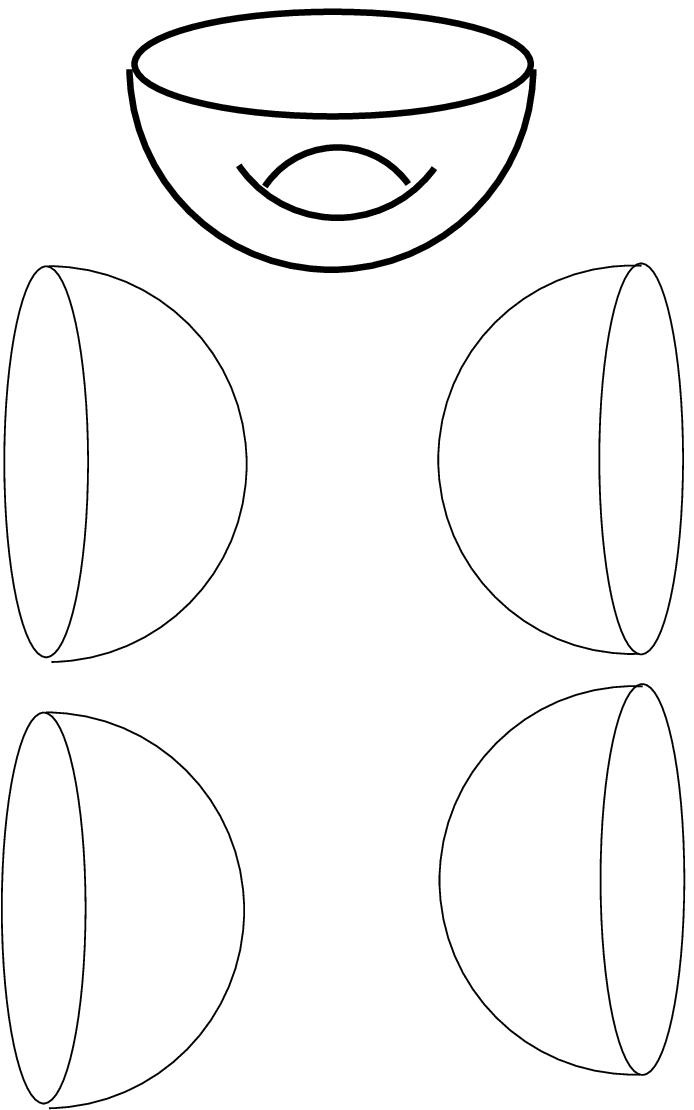} \} + \{
\eps{10mm}{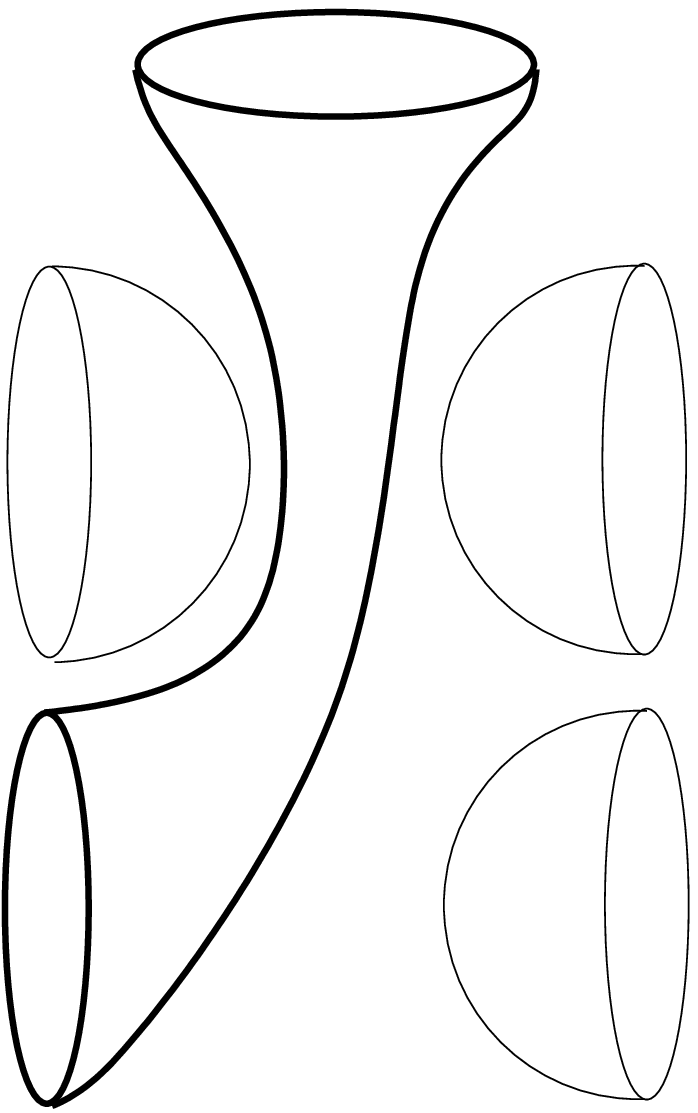} + \eps{10mm}{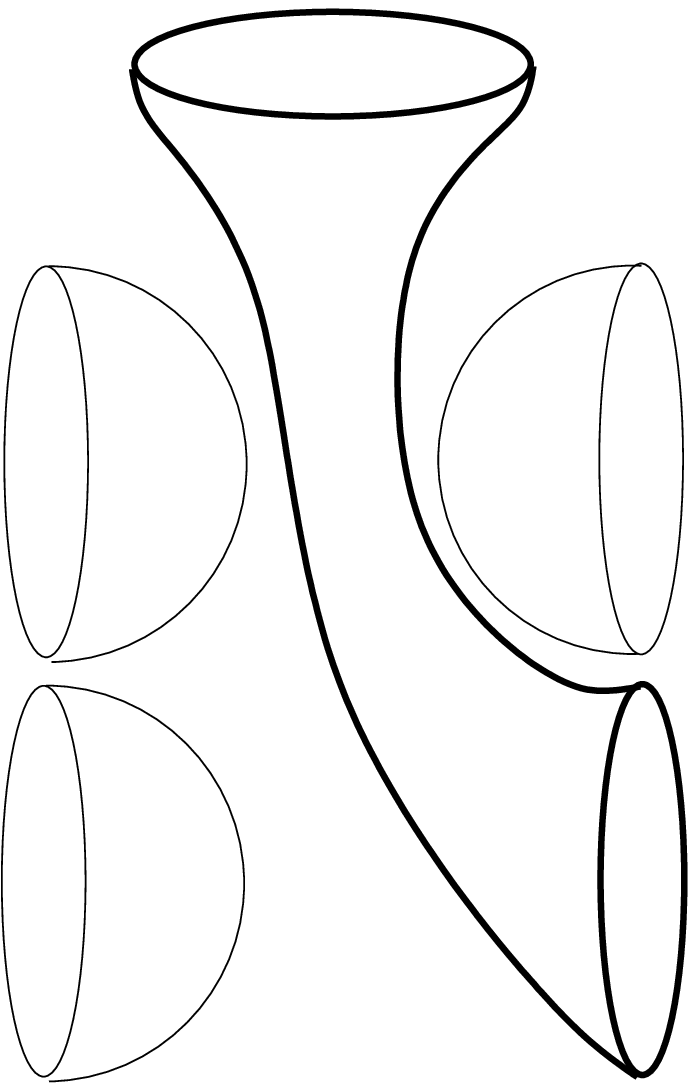} - \eps{10mm}{welldef7}
\} - \{ \eps{10mm}{welldef5} + \eps{10mm}{welldef8} -
\eps{10mm}{welldef7} \} - \{ \eps{10mm}{welldef9} +
\eps{10mm}{welldef6} - \eps{10mm}{welldef7} \} = 0 $

\qed
\end{proof} 

%% file: ChapComplexreductionZ.tex
\chapter{The structure of the geometric complex over $\mathbb{Z}$}\label{chap:ComplexReductionZ}

\section{Notations and a technical comment}

We would like to reduce the complex over $\mathbb{Z}$ as we did over $\mathbb{Q}$. Recall that whenever working with surfaces over $\mathbb{Z}$ we need to pick a special boundary circle first. We do that by marking the link at one point. After doing so, we have a special circle at every appearance of an object of $\Cobl$ in the complex (the one containing the mark). More over, every morphism in the complex connects special circle to special circle (just follow the construction carefully) and thus satisfies the technical requirement of remark \ref{remark:Hlinear}. For the sake of clarity we will take an equivalent approach and work in the category of 1-1 tangles (also called ``long knots''). One gets 1-1 tangles by cutting the link open at the point chosen. This is actually a canonical choice when working with knots and an irrelevant choice when working with links. We refer any further discussion on this technical issue to section \ref{section:commentmarkingpoint}. Notice that when working with 1-1 tangles, every appearance of the special circle will turn into an appearance of a \textbf{special line}. Morphisms between special lines will look like $\epsg{10mm}{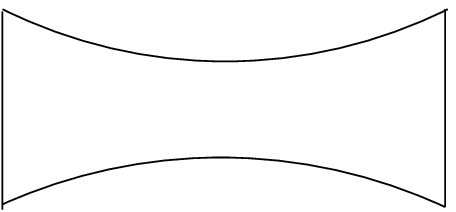}$ and will be called \textbf{curtains}. The special 1-handle operator $H$ still operates by adding a handle to the special component (this time a curtain with a handle added, $\epsg{10mm}{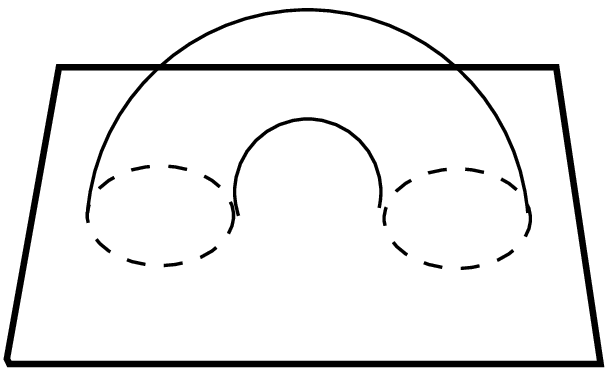}$). We will also use the following notation: a \textbf{dot} on a component means that a neck is connected between the dot and the special component (the curtain). That is, $\eps{10mm}{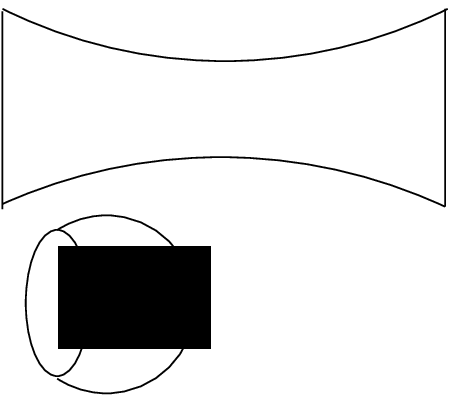} = \eps{10mm}{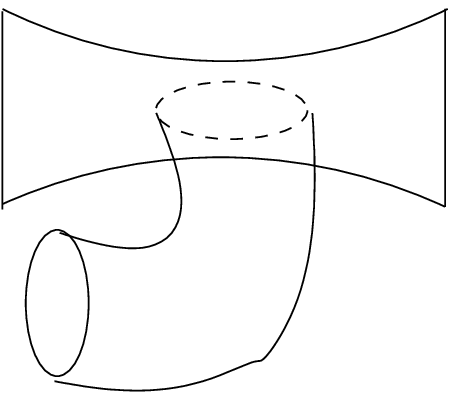}$.

\section{Reduction theorem over $\mathbb{Z}$}

\begin{theorem}
The complex associated to a link, over $\mathbb{Z}$, is equivalent
to a complex built from the category of free
$\mathbb{Z}[H]$-modules.
\end{theorem}

The theorem is a corollary of the following proposition, which we
will prove first :

\begin{proposition}
In $\Mat(\Cobl)$ a non-special circle is isomorphic to a column of
two empty set objects. This isomorphism extends (taking degrees into
account) to an isomorphism of complexes in $\Komh\Mat(\Cobl)$ where
the non-special circle objects are replaced by empty sets together
with the appropriate induced complex maps.
\end{proposition}

\begin{definition}\label{def:delooping}
The isomorphism from the proposition is called \emph{delooping} and
is given by the following diagram :

\begin{center}
$\eps{27mm}{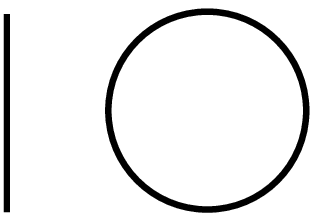}$ $\eps{17mm}{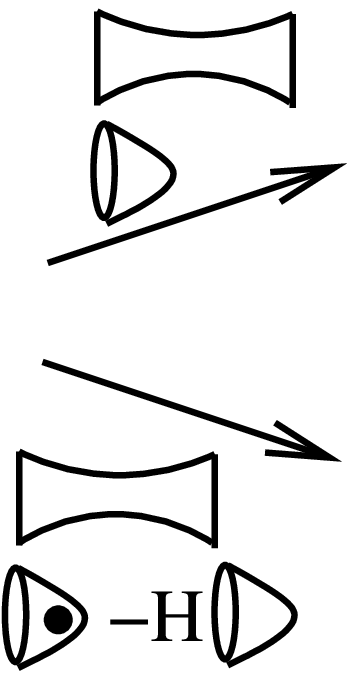}$ $ \left[
\begin{array}{c}
 |~\emptyset \{ -1 \}\\
\\
\\
|~\emptyset\{+1\}
\end{array} \right] $
$\eps{17mm}{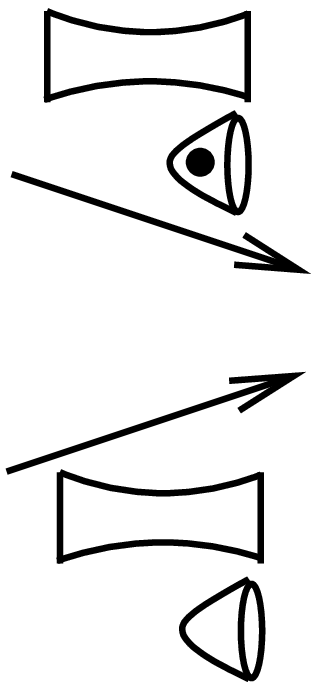}$ $\eps{27mm}{linecirc}$
\end{center}

This diagram shows how to make a non-special circle disappear into a
column of two empty sets in the presence of the special line. The
special line (remember we now work with 1-1 tangles for the sake of clarity)
functions as a ``probe'' which in his presence all the other circles
can be isomorphed into empty sets, leaving us with nothing but the
special line. The diagram shows the objects and morphisms (above and below the arrows). These contain \emph{curtains} between two special lines, drawn as $\epsg{10mm}{curtainexample}$. Also, remember that a \emph{dot} means a neck connecting to the curtain and that the special 1-handle operator $H$ will add a handle to the curtain (as described in the above notation section).
\\

\end{definition}

\begin{proof}(Prop.)\\
First we prove that the delooping process is indeed an isomorphism of objects.
We trace the arrows and use the relations of $\Cobl$. When composing the above diagram
from $|~\bigcirc$ to $|~\bigcirc$ we get: $\eps{10mm}{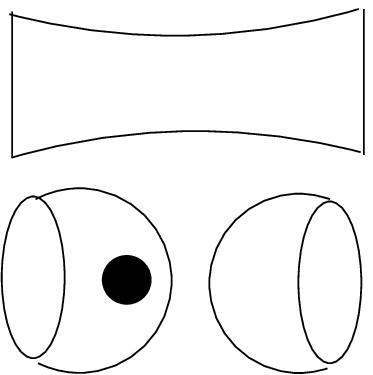} - \eps{10mm}{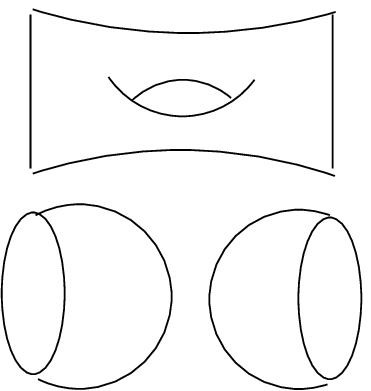} + \eps{10mm}{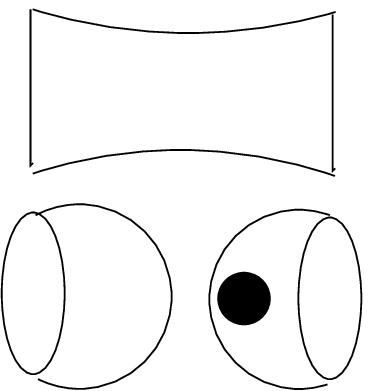}$ (remember that $H$ adds a handle to the special component). Modulo the $3S1$ relation this equals to $\eps{10mm}{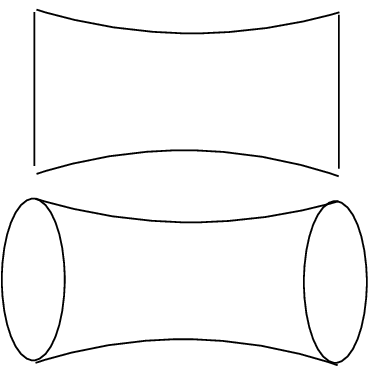}$ which is the identity cobordism. The other direction uses trivially the S relation. The extension to the full category is well defined and the second part of the proposition naturally follows, keeping track of degrees. One replaces the appearances of non-special circles by columns of empty sets and compose the above isomorphism with the original complex maps to get the induced maps.
\qed
\end{proof}

\vspace{3mm}

\begin{proof}(Thm.)\\
Given a link (1-1 tangle actually) we use the proposition to reduce its complex and replace all the non-special circles with columns of empty sets. The only thing left
in the complex are columns with the special line in its entries. The morphism set of that object consists of curtains with any genus which is isomorphic to $\mathbb{Z}[H]$, where $H$ is the special 1-handle operator. Thus we get a complex made of columns of free $\mathbb{Z}[H]$-modules and maps which are matrices with $\mathbb{Z}[H]$ entries.
\qed
\end{proof}

\begin{remark} Though surfaces with two boundary components in $\Cobl$ have two free generators over $\mathbb{Z}[H]$ only the connected generator (a curtain) appears in the complex associated to a link (1-1 tangle), therefore the morphism group above can be reduced from $\mathbb{Z}[H] \bigoplus \mathbb{Z}[H]$ to $\mathbb{Z}[H]$. As in the case over $\mathbb{Q}$ the appearance of $H$ comes only in homogeneous form, i.e. monomials entries, due to grading considerations.
\end{remark}

\section{The algebraic structure underlying the complex
over $\mathbb{Z}$.} We follow the same trail as the reduction over
$\mathbb{Q}$ to get some information on the underlying structure of
the geometric complex over $\mathbb{Z}$.
\\

The non-special circle, a basic object in the theory, decomposes into two copies of another fundamental object (the empty set) with relative degree shift 2. This restricts functors from the geometric category to any other category which might carry an algebraic structure of direct sums (TQFT for example, into the category of
$\mathbb{Z}$-modules). The special line (special circle) stays as is, but as we will see below it actually carries an intrinsic ``one dimensional'' object. In chapter \ref{chap:TQFT} we will also see that the special line can be \emph{promoted} to carry higher dimensional algebraic objects.
\\

Denote $|~\{-1\}$ by $v_-$ (comparing to ~\cite{ba1}) or X
(comparing to ~\cite{kho1}) and $|~\{+1\}$ by $v_+$ or 1 (comparing
accordingly). We use the tensor product symbol to denote unions of
such empty sets (always with one special line only), thus keeping
track of degree shifts.
\\

Now, we take the pair of pants map $\eps{7mm}{InvertedPOP}$ between
$|~\bigcirc\bigcirc$ and $|~\bigcirc$. By this we mean the complexes
$|~\bigcirc\bigcirc\xrightarrow{pants}|~\bigcirc$ and
$|~\bigcirc\xrightarrow{pants}|~\bigcirc\bigcirc$ where the special
lines are connected with a curtain. We reduce these complexes into
empty sets complexes (replacing the non-special circles with empty
sets columns and replacing the maps with the induced maps, as in the
reduction theorem). We denote by $\Delta_1$ the composition:

\begin{center}
$ \Delta_1: \left[ \begin{array}{c}
v_-\\
v_+
\end{array} \right]
\xrightarrow{isomorphism} |~\bigcirc
\xrightarrow{\eps{7mm}{InvertedPOP}} |~\bigcirc\bigcirc
\xrightarrow{isomorphism} \left[
\begin{array}{c}
v_-\otimes v_-\\
v_-\otimes v_+\\
v_+ \otimes v_-\\
v_+ \otimes v_+
\end{array} \right]
$
\end{center}

We denote by $m_1$ the composition in the reverse direction. By
composing surfaces and using the 4TU/S/T relations the maps we get
are:

\[
  \Delta_1: \begin{cases}
    v_+ \mapsto  \left[ \begin{array}{c} v_-\otimes v_+ \\ v_+\otimes v_- \\ - H v_+\otimes v_+  \end{array} \right] &
    \\
    v_- \mapsto v_-\otimes v_- &
  \end{cases}
  \qquad
  m_1: \begin{cases}
    v_+\otimes v_-\mapsto v_- &
    v_+\otimes v_+\mapsto v_+ \\
    v_-\otimes v_+\mapsto v_- &
    v_-\otimes v_-\mapsto Hv_-.
  \end{cases}
\]

Let us check what are the induced maps when the pair of pants involves the special line. These ``torn'' pair of pants looks like $\eps{7mm}{dotexample1}$ and will be denoted $\eps{7mm}{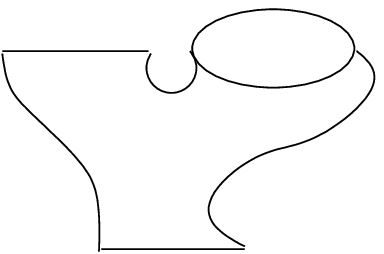}$. Denote by $\Phi$ the following composition :

\begin{center}
$ \Phi: \left[ \begin{array}{c}
v_-\\
v_+
\end{array} \right]
\xrightarrow{isomorphism} |~\bigcirc
\xrightarrow{\eps{7mm}{popline}} | $
\end{center}

Denote by $\Psi$ the composition in the reverse direction. Composing
and reducing surfaces we get the following maps :

\[
  \Phi: \begin{cases}
    v_- \mapsto H \\
    v_+ \mapsto 1
  \end{cases}
  \qquad
  \Psi: \begin{cases}
    \emptyset \mapsto v_-
  \end{cases}
\]

\vspace{3mm} ($\Delta_1,m_1$) is the type of algebraic structure one
encounters in a TQFT (co-product and product of the Frobenius
algebra). The appearance of a special line (special circle) is most
natural in the construction of the \emph{reduced knot homology},
introduced in ~\cite{kho3}. In the setting of Khovanov homology
theory, first, one views the entire chain complex as a complex of
$\mathbb{A}$-modules ($\mathbb{A}$ is the Frobenius algebra
underlying the TQFT used) through a natural action of $\mathbb{A}$
(here one has to mark the knot and encounter the special line).
Then, one can take the kernel complex of multiplication by X to be
the reduced complex. One can also take the \emph{co-reduced
complex}, which is the image complex of multiplication by X. The
above ($\Delta_1$,$m_1$) structure is exactly the TQFT denoted
$\calF_7$ in ~\cite{kho1}. The $mod~2$ specialization of this theory
appeared in section 9 in ~\cite{ba1}. The additional structure
coming from $\Psi$ and $\Phi$ is exactly the structure of the
\emph{co-reduced} theory of the ($\Delta_1,m_1$) TQFT, where the
special line carries the ``one dimensional'' object generated by X.
It is important to note again that this is an \emph{intrinsic}
information on the underlying structure, coming before any TQFT is
even applied. This gives the above ``pre-TQFT co-reduced structure''
a unique place in the universal theory.

%% file: ChapTQFT.tex
\chapter{TQFTs and link homology theories put on the complex}\label{chap:TQFT}

The complex invariant is geometric in nature and one cannot form
homology groups directly (kernels make no sense since the category
is additive but not Abelian). Thus we need to apply a functor into
an algebraic category where one can form homology groups. We will
classify such functors and present the universal link homology
theory. We will also explore the relative strength and the
information held within these functors. One of these types of
functors is a TQFT which maps the category $Cob$ into the category
of modules over some ring. In the two dimensional case, which is the
relevant case in link homology, such TQFT structures are equivalent
to Frobenius algebras and are classified by them ~\cite{ab} \cite{kho1}. A
priori, the link homology theory coming from a TQFT does not have to
satisfy the 4TU/S/T relations of the complex. It turns out that
any TQFT (up to twisting and base change) used to create link
homology can be put on the geometric complex, making the complex
universal for link homology coming from TQFTs. We show that the
universal link homology functor is actually a TQFT and every other
link homology functor factorizes through it (including non TQFTs). We also construct a unified mechanism of extracting all link homology theories from the universal complex, forming a convenient ``recipe'' of link homology theories.

\section{Two important TQFTs}
Given the structures $(\Delta_1, m_1)$ and
$(\Delta_2,m_2)$, introduced in the previous chapters, one can add
linearity and get the following algebraic structures (using the same
notations):

\[
  \Delta_1: \begin{cases}
    v_+ \mapsto v_+\otimes v_- + v_-\otimes v_+  - H v_+\otimes v_+ &\\
    v_- \mapsto v_-\otimes v_- &
  \end{cases}
  \]
  \[
  m_1: \begin{cases}
    v_+\otimes v_-\mapsto v_- &
    v_+\otimes v_+\mapsto v_+ \\
    v_-\otimes v_+\mapsto v_- &
    v_-\otimes v_-\mapsto Hv_-
  \end{cases}
\]

\[
  \Delta_2: \begin{cases}
    v_+ \mapsto v_+\otimes v_- + v_-\otimes v_+ &\\
    v_- \mapsto v_-\otimes v_- + T v_+\otimes v_+ &
  \end{cases}
  m_2: \begin{cases}
    v_+\otimes v_-\mapsto v_- &
    v_+\otimes v_+\mapsto v_+ \\
    v_-\otimes v_+\mapsto v_- &
    v_-\otimes v_-\mapsto Tv_+
  \end{cases}
\]

Then, one can construct TQFTs (Frobenius algebras) based on these algebraic structures. The first will be denoted $\calF_H$ and is given by the algebra $A_H=R_H[X]/(X^2-HX)$ over $R_H=R[H]$ (algebra $\calF_7$ in ~\cite{kho1}). The second algebra is $A_T=R_T[X]/(X^2-T)$ over the ring $R_T=R[T]$ which will be denoted $\calF_T$ (named $\calF_3$ in ~\cite{kho1}). Of course one can take the reduced or co-reduced homology structures, and we denote it by superscripts ($\calF^{co}_H$ for example). Due to the fact that these are the underlying algebraic structures of the geometric complex it is not surprising that they dominate the information in link homology theory.

\section{The universal link homology theories}

We give the most general link homology theory that can be applied to
the geometric complex.

\begin{theorem}
Every functor used to create link homology theory from the geometric
complex over $\mathbb{Q}$, or any ring with 2 invertible, factors
through $\calF_T$. I.e. the TQFT $\calF_T$ is the universal link
homology theory when 2 is invertible and holds the maximum amount of
information.
\end{theorem}

\begin{proof}
We have seen in chapter \ref{chap:ComplexReductionQ} that in the reduced geometric complex the objects are $\mathbb{Q}[T]$-modules. The target objects of the link homology functor must be $\mathbb{Q}[T]$-modules too (by functoriality). Since every object (circle) in the complex is isomorphic to a direct sum of two empty sets, the functor
will be determined by choosing a single $\mathbb{Q}[T]$-module corresponding to the empty set. The universal choice would be $\mathbb{Q}[T]$ itself, and the complex coming from any other module will be obtained by tensoring the complex with it. Once this is done, the ``pre-TQFT'' structure turns into a real TQFT structure $\calF_T$.

\qed
\end{proof}

\begin{theorem} Every functor used to create link homology
theory from the geometric complex over $\mathbb{Z}$ factors through
$\calF^{co}_H$ which holds the maximum amount of information. I.e.
the TQFT $\calF^{co}_H$ is the universal functor for link homology.
\end{theorem}

\begin{proof}
Similar to the case over $\mathbb{Q}$, the complex is equivalent to a $\mathbb{Z}[H]$-modules complex, and the universal choice is $\mathbb{Z}[H]$. This makes the ``pre-TQFT'' structure over $\mathbb{Z}$ turn into the dominant TQFT functor $\calF^{co}_H$. All other theories will be tensor products of the universal one.

\qed
\end{proof}

\section{A ``big'' Universal TQFT by Khovanov}
A priori, TQFTs can be used to create link homology theory without
the geometric complex (meaning without satisfying the 4TU/S/T
relations). In ~\cite{kho1} Khovanov presents a universal rank 2
TQFT (Frobenius algebra) given by the following formula (denoted
$\calF_5$ there):

\[
  \Delta_{ht}: \begin{cases}
    1 \mapsto 1\otimes X + X\otimes 1 -h1\otimes 1&\\
    X \mapsto X\otimes X + t1\otimes 1 &
  \end{cases}
\]
\[
  m_{ht}: \begin{cases}
    1\otimes X\mapsto X &
    1\otimes 1\mapsto 1 \\
    X\otimes 1\mapsto X &
    X\otimes X\mapsto t1 + hX.
  \end{cases}
\]

This structure gives the Frobenius algebra
$A_{ht}=R_{ht}[X]/(X^2-hX-t)$ over the ring
$R_{ht}=\mathbb{Z}[h,t]$. Comparing to our notations (and these of
~\cite{ba1}) is done by putting $v_-=X$ and $v_+=1$. We will denote
this TQFT $\calF_{ht}$.
\\

Given a (tensorial) functor $\calF$ from $Cob$ to the category of
$R$-modules (a TQFT), one can construct a link homology theory and
ask whether it is (homotopy) invariant under Reidemeister moves. If
it is invariant under the first Reidemeister move, then it is a
Frobenius algebra of rank 2. Khovanov showed that every rank 2
Frobenius algebras can be \emph{twisted} into a descended theory,
$\calF'$, which is a base change of $\calF_{ht}$ (base change is
just a unital ground ring homomorphism which induces a change in the
algebra. The notion of twisting is explained in ~\cite{kho1} and ref
therein). The complexes associated to a link using $\calF$ and
$\calF'$ are isomorphic, thus all the information is still there
after the twist. Since base change just tensors the chain complex
with the appropriate new ring (over the old ring), algebra
$\calF_{ht}$ is universal for link homology theories coming from
TQFTs. The fact that $\calF_{ht}$ satisfies the 4TU/S/T relations
allows one to apply it as a homology theory functor on the geometric
complex and use ~\cite{ba1} results (giving Proposition 6 in
~\cite{kho1}). Base change does not change the fact that a theory
satisfies the 4TU/S/T relations, and since $\calF_{ht}$ satisfies
these relations, we have that every tensorial functor $\calF$ that
is invariant under Reidemeister-1 move can be twisted into a functor
$\calF'$ that satisfies the relations of $\Cobl$ and thus can be put
on the geometric complex without losing any homological information
relative to $\calF$. This gives the universality of the geometric
complex for all link homology theories coming from TQFTs. Thus, our ``smaller'' universal link homology functors presented in the theorems above
capture all the information coming from \emph{any} TQFT constructions in link homology theory.
\\

Over $\mathbb{Q}$ one can see how our universal theory
captures all the information of khovanov's ``big'' TQFT. We start
with $\calF_{ht}$ over $\mathbb{Q}$ (i.e. $R_{ht}=\mathbb{Q}[h,t]$).
By doing a change of basis: $
\begin{cases}
\begin{array}{l}
v_- -\frac{h}{2}v_+ \mapsto v_-\\
v_+ \mapsto v_+
\end{array}
\end{cases}$
we get the theory given by:
\[
  \begin{cases}
    v_+ \mapsto v_+\otimes v_- + v_-\otimes v_+ &\\
    v_- \mapsto v_-\otimes v_- + (t+\frac{h^2}{4}) v_+\otimes v_+ &
  \end{cases}
  \begin{cases}
    v_+\otimes v_-\mapsto v_- &
    v_+\otimes v_+\mapsto v_+ \\
    v_-\otimes v_+\mapsto v_- &
    v_-\otimes v_-\mapsto (t+\frac{h^2}{4})v_+
  \end{cases}
\]
\\
Re-normalizing $t+\frac{h^2}{4}=\frac{T}{4}:=\tilde{T}$, we see that
the above theory is just $\calF_{\tilde{T}}$ with the ground ring
extended by another superficial variable $h$. Every Calculation done
using $\calF_{ht}$ to get link homology, is equal to the same
calculation done with $\calF_{\tilde{T}}$ tensored with
$\mathbb{Q}[h]$, and holds exactly the same information. Note that
our complex reduction is performing this change of basis (to ``kick
out'' one redundant variable) on the complex level, before even
applying any TQFT. Doing first the complex reduction and then
applying the TQFT $\calF_{ht}$ will factorize the result through
$\calF_{\tilde{T}}$.
\\

In the general case, over $\mathbb{Z}$, one needs some further
observations which we turn to now.

\section{About the 1-handle and 2-handle operators in
$\calF_{ht}$} Given a TQFT the 1-handle operator takes a specific algebraic form, depending on the TQFT. This is done by viewing the operator that adds a handle to a cobordism as a map $\eps{5mm}{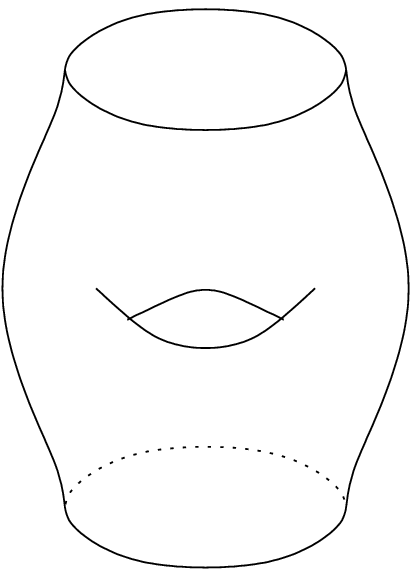}$: $A \mapsto A$ between the two Frobenius algebras associated to the boundary circles. One can find the form by looking at the
multiplication and co-multiplication formulas. In the case of $\calF_{ht}$: $H=2X-h$. The 2-handle operator $T$ follows immediately by computing $T=H^2$ and reducing modulo $X^2-hX-t$ to get: $T=4t+h^2$. The 2-handle operator is an element of the ground ring, as expected from our 2-handle lemma, and can be multiplied anywhere in a tensor product. The 1-handle operator is an element of $A$ but not of the ground ring, and thus when operating on a tensor product one needs to specify the component to operate on, as we did by picking the special circle. Since $A$ is two dimensional, picking the basis to be $\left(\begin{array}{c} 1 \\ X \end{array} \right)$ would give rise to a two dimensional representation of the operators above: $H=\left(\begin{array}{cc} -h&2t \\ 2&h \end{array} \right)$, $T=\left(\begin{array}{cc} 4t+h^2 & 0 \\ 0&4t+h^2  \end{array} \right)$.

\section{The special line and H promotion}
The main result of chapter \ref{chap:ComplexReductionZ} was a reduction of the geometric complex into a complex composed of columns of the special line whose maps are matrices with $H$ monomials entries. Given such a complex one wishes to create a homology theory
out of it, i.e. apply some functor that will put an $H$-module on the special line with the possibility of taking kernels (and forming homology groups). As we have seen, the intrinsic structure that the special line caries is the one dimensional module
$\mathbb{Z}[H]$ generated by $X$. We can get a different structure by replacing $H$ with any integer matrix of dimension $n$ and the special line with direct sum of $n$ copies of $\mathbb{Z}$. We call this type of process \emph{a promotion} (realizing it is a fancy name for tensoring). Another type of promotion is replacing $H$ with a matrix (of dimension $n$) with polynomial entries in two variables ($h$ and $t$, say) and then promoting the special line into the direct sum of $n$ copies of $\mathbb{Z}[h,t]$. Note that some promotions may lose information regarding the geometric complex.
\\

\section{TQFTs via promotion and new type of link homology functors}

First, we want to find a promotion of $H$ that is equivalent to the
theory $\calF_{ht}$ and which does not lose any information on the geometric
complex level. From the above discussions it is clear that such a promotion is $H=\left( \begin{array}{cc} -h&2t \\2&h \end{array} \right)$. The special line will be promoted into two copies of $\mathbb{Z}[h,t]$. Since each power of the promoted $H$ adds a power of $t$ and $h$ to the matrix entries, the power of the 1-handle operator $H$ can be uniquely determined after the promotion and we lose no information. The algebraic complex one gets by applying this promotion is equal to the complex one gets by applying the TQFT $\calF_{ht}$ before the complex reduction.
\\

We can now easily get the other familiar TQFTs using the same
complex reduction and $H$ promotion technique :
\[
\calF_{\tilde{T}} \leftarrow H=\left( \begin{array}{cc} 0&2\tilde{T} \\
2&0 \end{array} \right) \hspace{7mm} \calF_H \leftarrow H=\left( \begin{array}{cc} -H&0 \\
2&H \end{array} \right)
\]
One gets the standard Khovanov homology ($X^2=0$) by the promotion
$H=\left( \begin{array}{cc} 0&0 \\ 2&0 \end{array} \right)$. The
special line is promoted to double copies of
$\mathbb{Z}[\tilde{T}]$, $\mathbb{Z}[H]$ and $\mathbb{Z}$
respectively. One can get the \emph{reduced Khovanov homology} by
focusing on the bottom right entry of the $H$ promotion matrix. It is
interesting to note that in the standard Khovanov homology case
indeed we lose information and $H^2=T=0$. In the geometric language
this means that in order to get the standard Khovanov homology one
has to ignore surfaces with genus 2 and above. This was observed in
~\cite{ba1}. Reduction to Lee's theory ~\cite{lee1} is done by
substituting $\tilde{T}=1$.
\\

One can create other types of promotions that will enable us to
control the order of $H$ involved in the theory. We are able to
cascade down from the most general theory, the one that involves all
powers of $H$ (i.e. surfaces with any genera in the topological
language), into a theory that involves only certain powers of $H$
(i.e. genera up to a certain number). For example, promote the
special line to three copies of $\mathbb{Z}$, and $H$ to the matrix
$\left(\begin{array}{ccc} 0&0&0 \\ 1&0&0 \\ 0&1&0 \end{array} \right)$. This
theory involves only powers of $H$ smaller or equal to 2, that is
surfaces of genus up to 2. These promotions can be viewed as a
family of theories extrapolating between the ``standard'' Khovanov
TQFT and our universal theory (reminding of a genus perturbation expansion
in string theory).

\section{A comment on tautological functors}
In section 9 of ~\cite{ba1} tautological functors where defined on
the geometric complex associated to a link. One needs to fix an
object in $\Cobl$, say $\calO'$, and then the tautological functor
is defined by $\calF_{\calO'}(\calO)=\Mor(\calO',\calO)$, taking
morphisms to compositions of morphisms. Our classification allows us
to state the following regarding tautological functors:
\\

\begin{corollary}
The tautological functor $\calF_\bigcirc(-)$ over $\mathbb{Z}[H]$ is
the TQFT $\calF_H$.\\
The tautological functor $\calF_\emptyset(-)$ over $\mathbb{Q}[T]$
is the TQFT $\calF_T$.
\end{corollary}

\begin{proof}
Indeed this is a corollary to the surface classification. In the
first case declare the special circle to be the source $\bigcirc$.
\qed
\end{proof}

\vspace{3mm}
Due to the universality theorems of this thesis we see that \emph{tautological functors are able to extract maximal amount of information from the geometric complex}.

%% file: ChapCompExt.tex
\chapter{Computations, special cases, phenomena and extensions}\label{chap:CompExt}

This chapter presents a wide range of results in link homology theory, all emanating from the universal theory presented in this thesis. We describe how one computes efficiently (using a computer) the universal Khovanov link homology. We will describe the algorithm, and present examples of such calculations for many knots. Using these calculations it will be easy to understand the different specializations from the universal theory to specific homology theories. We present results regarding various homology theories that follow from the structure of the universal complex. These results are both theoretical and phenomenological. Finally, we will also see how the universal theory interacts with other tools already used in the field (like spectral sequences and the Rasmussen invariant) and with various extensions of Khovanov link homology (like the $sl(3)$ link homology theory). One of the objectives of this chapter is to show how the universal theory, presented in this thesis, easily projects all previous results regarding link homology and enables generalizations of these results. We hope to demonstrate that the universal complex is the better approach to phenomenology research as well as to theoretical research.

\section{Computations of the universal theory}
We wish to calculate the universal Khovanov homology theory. That is, we want to calculate the geometric complex \emph{over $\mathbb{Z}[H]$} (as described in chapter \ref{chap:ComplexReductionZ}) without any restrictions and \emph{before} applying any functor (like a TQFT) to produce a homology theory. Obviously if one follows the definitions one can theoretically compute the complex. This would be very slow and non efficient (as complexity grows exponentially with the number of crossings). Moreover, it would be very difficult to reduce the geometric complex, once computed, into a homotopic complex which is simpler and smaller. As was shown in \cite{ba4}, for the case of the original standard Khovanov homology (the TQFT $\mathbb{Q}[X]/X^2$), fast computations can make one happy! Thus we wish for an efficient algorithm. At first sight the
complex isomorphism  \ref{def:delooping} does not seem to reduce the
geometric complex at all (it doubles the amount of objects in the complex). The surprising fact is that one can use this isomorphisms (and some homology theory techniques) to create a very efficient crossing-by-crossing local algorithm to calculate the geometric complex (and thus any link homology theory). It is important to mention that the algorithm was first proposed by D.Bar-Natan for the standard Khovanov
homology (over $\mathbb{Q}$ with high genera set to zero), as reported in ~\cite{ba3} \cite{ba4}. More recently it was implemented by J.Green \cite{green} for calculating the universal case over $\mathbb{Z}[H]$ (with algorithm adjusted using the work presented in this thesis). We will be using this implementation to produce the calculations.

\subsection{The algorithm for the efficient computation of the universal Khovanov link homology theory}\label{section:algorithm}

 We follow the description of the basic algorithm in \cite{ba3} while inserting the adjustments needed when dealing with the universal theory, and giving a description of Green's implementation \cite{green}. The main ingredients of the calculation are the appropriate delooping process and a standard reduction lemma from homological algebra theory. Combined with the fact that the geometric complex behaves well under tangle compositions, one can create a \emph{local} calculation scheme in which each crossing is added to previous calculations, while reduction is done after every such step (instead of at the end only). For simplicity, we will focus only on knots in this chapter.
 \\

 We start with a knot. The universal theory requires calculations using a special component (or a special line as in chapter \ref{chap:ComplexReductionZ}) thus we cut open the knot into a 1-1 tangle. The program does that, along with ordering the crossings (the input for the program is actually an encoding of a planar diagram for our knot). Choosing a first crossing we create the geometrical complex that is associated to that crossing. Then, the program chooses another crossing and adds it to the geometric complex. This is done by gluing the new crossing to every geometric objects that are already in the complex, according to the way it was glued in the knot diagram.
 \\

Now, Whenever an object in our complex contains a closed loop, remove it using the delooping process in \ref{def:delooping} :
\begin{center}
$\eps{27mm}{linecirc}$ $\eps{17mm}{faczl2b}$ $ \left[
\begin{array}{c}
 |~\emptyset \{ -1 \}\\
\\
\\
|~\emptyset\{+1\}
\end{array} \right] $
$\eps{17mm}{faczra}$ $\eps{27mm}{linecirc}$
\end{center}

This done, our complex becomes a new complex. The new complex is bigger than the one before the delooping, but it is made up of fewer possible objects. Thus it is likely that many morphisms in the new complex are isomorphisms (and indeed this is the case).
The program keeps track of this process, it adds the new crossing to the old complex, and recognize closed circles. It deloops these circles while keeping track of degrees \emph{and the various handles} that are attached to the morphisms in the complex. Then it tries to recognize isomorphisms out of the complex morphisms. These are easy to recognize as they are plain curtains (a surface between two lines with no handles).
Once it recognizes the isomorphisms it iteratively uses the standard homological algebra reduction lemma to reduce the size of the complex :

\begin{lemma} \cite{ba3} If $\phi:b_1\to b_2$ is an isomorphism (in some additive
category $\calC$), then the four term complex segment in $\Mat(\calC)$
\begin{equation}
  \xymatrix@C=2cm{
    \cdots\
    \left[C\right]
    \ar[r]^{\begin{pmatrix}\alpha \\ \beta\end{pmatrix}} &
    {\begin{bmatrix}b_1 \\ D\end{bmatrix}}
    \ar[r]^{\begin{pmatrix}
      \phi & \delta \\ \gamma & \epsilon
    \end{pmatrix}} &
    {\begin{bmatrix}b_2 \\ E\end{bmatrix}}
    \ar[r]^{\begin{pmatrix} \mu & \nu \end{pmatrix}} &
    \left[F\right] \  \cdots
  }
\end{equation}
is isomorphic to the (direct sum) complex segment
\begin{equation}
  \xymatrix@C=3cm{
    \cdots\
    \left[C\right]
    \ar[r]^{\begin{pmatrix}0 \\ \beta\end{pmatrix}} &
    {\begin{bmatrix}b_1 \\ D\end{bmatrix}}
    \ar[r]^{\begin{pmatrix}
      \phi & 0 \\ 0 & \epsilon-\gamma\phi^{-1}\delta
    \end{pmatrix}} &
    {\begin{bmatrix}b_2 \\ E\end{bmatrix}}
    \ar[r]^{\begin{pmatrix} 0 & \nu \end{pmatrix}} &
    \left[F\right] \  \cdots
  }.
\end{equation}
Both these complexes are homotopy equivalent to the (simpler) complex
segment
\begin{equation}
  \xymatrix@C=3cm{
    \cdots\
    \left[C\right]
    \ar[r]^{\left(\beta\right)} &
    {\left[D\right]}
    \ar[r]^{\left(\epsilon-\gamma\phi^{-1}\delta\right)} &
    {\left[E\right]}
    \ar[r]^{\left(\nu\right)} &
    \left[F\right] \  \cdots
  }.
\end{equation}
Here $C$, $D$, $E$ and $F$ are arbitrary columns of objects in $\calC$
and all Greek letters (other than $\phi$) represent arbitrary matrices
of morphisms in $\calC$ (having the appropriate dimensions, domains
and ranges); all matrices appearing in these complexes are block-matrices
with blocks as specified. $b_1$ and $b_2$ are billed here as individual
objects of $\calC$, but they can equally well be taken to be columns of
objects provided (the morphism matrix) $\phi$ remains invertible.
\end{lemma}

Note that while reducing the complex the program needs to keep track of compositions of surfaces, this is done by counting boundary components and handles and using the Euler characteristic with computed genus of the surfaces in order to recognize the resulting surfaces (morphisms) after the lemma is applied. Not a trivial task!
\\

Finally, the complex we get after applying the reduction lemma is indeed smaller in size. Next, the program picks another crossing and repeats the process (glue the crossing + delooping + reduction lemma). It is obvious that this local algorithm (crossing by crossing) enables great reduction in computational costs and guarantees getting a very simple complex at the end of the process (due to all of the intermediate complex reductions).
\\

The final result, in the case of the universal complex, is a complex built out of special lines, and maps built out of morphisms which look like curtains with possible handles on them (powers of $H$). This is the universal Khovanov ($sl(2)$) link homology theory, as described in this thesis. Here is the universal complex for the knot $3_1$ (the Trefoil knot) :

\[
  \xymatrix@C=2cm{
    {\begin{bmatrix} |  \end{bmatrix}}_{-3,-8}
    \ar[r]^{\begin{pmatrix}    \eps{15mm}{curtain3}    \end{pmatrix}}
    &
    {\begin{bmatrix}     |    \end{bmatrix}}_{-2,-6}
    \ar[r]^{\begin{pmatrix}      0    \end{pmatrix}}
    &
    {\begin{bmatrix}    \emptyset    \end{bmatrix}}_{-1}
    \ar[r]^{\begin{pmatrix}    0    \end{pmatrix}}
    &
    \begin{bmatrix}    |    \end{bmatrix}_{0,-2}
  }
\]

\[\eps{30mm}{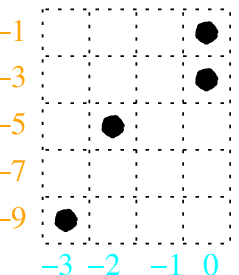}\]


On top, the universal complex is presented. The numbers on the right to each line are the homological degree $i$ (the first number) and the internal (vertical) grading $j$ (the second number). The maps are curtains with possible handles on them or zeros (if there are no curtains at all). We have no line at homological degree $-1$, thus the empty set. The table is a common encoding of the ``standard'' Khovanov homology over $\mathbb{Q}$ (that is the TQFT $\mathbb{Q}[X]/X^2$ applied to the geometric complex). The dots represent one copy of $\mathbb{Q}$, the vertical axis is the grading $j$ and the horizontal is the homological degree $i$.
\\
Instead of drawing lines and curtains we will use the following presentation of the universal complex :
\[
  \xymatrix@C=2cm{
    {\begin{bmatrix} -8  \end{bmatrix}}_{-3}
    \ar[r]^{\begin{pmatrix}    H    \end{pmatrix}}
    &
    {\begin{bmatrix}     -6    \end{bmatrix}}_{-2}
    \ar[r]^{\begin{pmatrix}      0    \end{pmatrix}}
    &
    {\begin{bmatrix}    ~    \end{bmatrix}}_{-1}
    \ar[r]^{\begin{pmatrix}    0    \end{pmatrix}}
    &
    \begin{bmatrix}    -2    \end{bmatrix}_{0}
  }
\]
Where the numbers in the square brackets represent the internal degrees of the special lines, the subscripts represent the homological degree and the maps are given by matrices of monomials in $H$, the 1-handle operator. We usually drop the empty brackets completely.

\subsection{Examples: The universal Khovanov link homology complex of all knots up to 7 crossings}
Following are simple examples of the universal complex. Want to calculate the universal complex for your favorite knot? Read section \ref{section:software}.
\\
\begin{center}

$3_1 \eps{15mm}{3_1}$
\[
  \xymatrix@C=1cm{
    {\begin{bmatrix} -8  \end{bmatrix}}_{-3}
    \ar[r]^{\begin{pmatrix}    H    \end{pmatrix}}
    &
    {\begin{bmatrix}     -6    \end{bmatrix}}_{-2}
    \ar[r]^{\begin{pmatrix}      0    \end{pmatrix}}
    &
    {\begin{bmatrix}    ~    \end{bmatrix}}_{-1}
    \ar[r]^{\begin{pmatrix}    0    \end{pmatrix}}
    &
    \begin{bmatrix}    -2    \end{bmatrix}_{0}
  }
\]
\vspace{6mm}

$4_1 \eps{15mm}{4_1}$
\[
  \xymatrix@C=1cm{
    {\begin{bmatrix} -4  \end{bmatrix}}_{-2}
    \ar[r]^{\begin{pmatrix}    H    \end{pmatrix}}
    &
    {\begin{bmatrix}     -2    \end{bmatrix}}_{-1}
    \ar[r]^{\begin{pmatrix}      0    \end{pmatrix}}
    &
    {\begin{bmatrix}    0    \end{bmatrix}}_{0}
    \ar[r]^{\begin{pmatrix}    0    \end{pmatrix}}
    &
    \begin{bmatrix}    2    \end{bmatrix}_{1}
    \ar[r]^{\begin{pmatrix}    H    \end{pmatrix}}
    &
    \begin{bmatrix}    4    \end{bmatrix}_{2}
  }
\]
\vspace{6mm}

$5_1 \eps{15mm}{5_1} $
\[
  \xymatrix@C=1cm{
    {\begin{bmatrix} -14  \end{bmatrix}}_{-5}
    \ar[r]^{\begin{pmatrix}    H    \end{pmatrix}}
    &
    {\begin{bmatrix}     -12    \end{bmatrix}}_{-4}
    \ar[r]^{\begin{pmatrix}      0    \end{pmatrix}}
    &
    {\begin{bmatrix}    -10    \end{bmatrix}}_{-3}
    \ar[r]^{\begin{pmatrix}    H    \end{pmatrix}}
    &
    \begin{bmatrix}    -8    \end{bmatrix}_{-2}
    \ar[r]^{\begin{pmatrix}    0    \end{pmatrix}}
    &
    \begin{bmatrix}    -4    \end{bmatrix}_{0}
  }
\]
\vspace{6mm}

$5_2 \eps{15mm}{5_2} $
\[
  \xymatrix@C=1cm{
    \begin{bmatrix} -12  \end{bmatrix}_{-5}
    \ar[r]^{\begin{pmatrix}    H    \end{pmatrix}}
    &
    \begin{bmatrix}     -10    \end{bmatrix}_{-4}
    \ar[r]^{\begin{pmatrix}      0    \end{pmatrix}}
    &
    \begin{bmatrix}    -8    \end{bmatrix}_{-3}
    \ar[r]^{\begin{pmatrix}    H\\0    \end{pmatrix}}
    &
    {\begin{bmatrix}-6 \\-6 \end{bmatrix}}_{-2}
    \ar[r]^{\begin{pmatrix}    0&H    \end{pmatrix}}
    &
    \begin{bmatrix}    -4    \end{bmatrix}_{-1}
        \ar[r]^{\begin{pmatrix}    0    \end{pmatrix}}
    &
    \begin{bmatrix}    -2    \end{bmatrix}_{0}
  }
\]
\vspace{6mm}

$6_1 \eps{15mm}{6_1} $
\[
  \xymatrix@C=1cm{
    \begin{bmatrix} -8  \end{bmatrix}_{-4}
    \ar[r]^{\begin{pmatrix}    H    \end{pmatrix}}
    &
    \begin{bmatrix}     -6    \end{bmatrix}_{-3}
    \ar[r]^{\begin{pmatrix}      0    \end{pmatrix}}
    &
    \begin{bmatrix}    -4    \end{bmatrix}_{-2}
    \ar[r]^{\begin{pmatrix}    H\\0    \end{pmatrix}}
    &
    {\begin{bmatrix}    -2\\-2    \end{bmatrix}}_{-1}
    \ar[r]^{\begin{pmatrix}   0&0 \\ 0&H    \end{pmatrix}}
    &
    {\begin{bmatrix}    0\\0    \end{bmatrix}}_{0}
     \ar[r]^{\begin{pmatrix}    0    \end{pmatrix}}
    &
    \begin{bmatrix}    2    \end{bmatrix}_{1}
     \ar[r]^{\begin{pmatrix}    H    \end{pmatrix}}
    &
    \begin{bmatrix}    4    \end{bmatrix}_{2}
  }
\]
\vspace{6mm}

$6_2 \eps{15mm}{6_2} $
\[
  \xymatrix@C=1cm{
    \begin{bmatrix} -0  \end{bmatrix}_{-4}
    \ar[r]^{\begin{pmatrix}    0\\H    \end{pmatrix}}
    &
    {\begin{bmatrix}     -8\\-8    \end{bmatrix}}_{-3}
    \ar[r]^{\begin{pmatrix}     H&0\\0&0    \end{pmatrix}}
    &
    {\begin{bmatrix}    -6\\-6    \end{bmatrix}}_{-2}
    \ar[r]^{\begin{pmatrix}    0&H \\ 0&0    \end{pmatrix}}
    &
    {\begin{bmatrix}    -4\\-4    \end{bmatrix}}_{-1}
    \ar[r]^{\begin{pmatrix}   0&0 \\ 0&H    \end{pmatrix}}
    &
    {\begin{bmatrix}    -2\\-2    \end{bmatrix}}_{0}
     \ar[r]^{\begin{pmatrix}    0    \end{pmatrix}}
    &
    \begin{bmatrix}    0    \end{bmatrix}_{1}
     \ar[r]^{\begin{pmatrix}    H    \end{pmatrix}}
    &
    \begin{bmatrix}    2    \end{bmatrix}_{2}
  }
\]
\vspace{6mm}

$6_3 \eps{15mm}{6_3} $
\[
  \xymatrix@C=1cm{
    \begin{bmatrix} -6  \end{bmatrix}_{-3}
    \ar[r]^{\begin{pmatrix}    H\\0    \end{pmatrix}}
    &
    {\begin{bmatrix}     -4\\-4    \end{bmatrix}}_{-2}
    \ar[r]^{\begin{pmatrix}     0&H\\0&0    \end{pmatrix}}
    &
    {\begin{bmatrix}    -2\\-2    \end{bmatrix}}_{-1}
    \ar[r]^{\begin{pmatrix}    0&0\\0&H\\0&0    \end{pmatrix}}
    &
    {\begin{bmatrix}    0\\0\\0    \end{bmatrix}}_{0}
    \ar[r]^{\begin{pmatrix}   \textit{0}&H \\ \textbf{0}&\hat{0}   \end{pmatrix}}
    &
    {\begin{bmatrix}    2\\2    \end{bmatrix}}_{1}
     \ar[r]^{\begin{pmatrix}    0&0\\0&H    \end{pmatrix}}
    &
    {\begin{bmatrix}    4\\4    \end{bmatrix}}_{2}
     \ar[r]^{\begin{pmatrix}    H&0    \end{pmatrix}}
    &
    \begin{bmatrix}    6    \end{bmatrix}_{3}
  }
\]

($\textbf{0},\textit{0},\hat{0}$ represent $2\times2,1\times2,2\times1$ zero matrices)
\vspace{6mm}

$7_1 \eps{15mm}{7_1} $
\[
  \xymatrix@C=0.5cm{
    \begin{bmatrix} -20  \end{bmatrix}_{-7}
    \ar[r]^{\begin{pmatrix}    H    \end{pmatrix}}
    &
    {\begin{bmatrix}     -18    \end{bmatrix}}_{-6}
    \ar[r]^{\begin{pmatrix}     0    \end{pmatrix}}
    &
    {\begin{bmatrix}    -16    \end{bmatrix}}_{-5}
    \ar[r]^{\begin{pmatrix}   H   \end{pmatrix}}
    &
    {\begin{bmatrix}    -14    \end{bmatrix}}_{-4}
    \ar[r]^{\begin{pmatrix}   0    \end{pmatrix}}
    &
    {\begin{bmatrix}    -12    \end{bmatrix}}_{-3}
     \ar[r]^{\begin{pmatrix}    H    \end{pmatrix}}
    &
    {\begin{bmatrix}    -10    \end{bmatrix}}_{-2}
     \ar[r]^{\begin{pmatrix}    0    \end{pmatrix}}
    &
    \begin{bmatrix}    -6    \end{bmatrix}_{0}
  }
\]

\vspace{6mm}

$7_2 \eps{15mm}{7_2} $
\[
  \xymatrix@C=1cm{
    \begin{bmatrix} -16  \end{bmatrix}_{-7}
    \ar[r]^{\begin{pmatrix}    H    \end{pmatrix}}
    &
    {\begin{bmatrix}     -14    \end{bmatrix}}_{-6}
    \ar[r]^{\begin{pmatrix}     0    \end{pmatrix}}
    &
    {\begin{bmatrix}    -12    \end{bmatrix}}_{-5}
    \ar[r]^{\begin{pmatrix}    H\\0    \end{pmatrix}}
    &
    {\begin{bmatrix}    -10\\-10    \end{bmatrix}}_{-4}
    \ar[r]^{\begin{pmatrix}   0&0 \\ 0&H    \end{pmatrix}}
    &
}\]
\[
  \xymatrix@C=1cm{
    {\begin{bmatrix}    -8\\-8    \end{bmatrix}}_{-3}
     \ar[r]^{\begin{pmatrix}    H&0\\0&0    \end{pmatrix}}
    &
    {\begin{bmatrix}    -6\\-6    \end{bmatrix}}_{-2}
     \ar[r]^{\begin{pmatrix}    0&H    \end{pmatrix}}
    &
    \begin{bmatrix}    -4    \end{bmatrix}_{-1}
    \ar[r]^{\begin{pmatrix}    0    \end{pmatrix}}
    &
    \begin{bmatrix}    -2    \end{bmatrix}_{0}
    }
\]
\vspace{6mm}

$7_3 \eps{15mm}{7_3} $
\[
  \xymatrix@C=1.5cm{
    \begin{bmatrix} 4  \end{bmatrix}_{0}
    \ar[r]^{\begin{pmatrix}    0    \end{pmatrix}}
    &
    {\begin{bmatrix}     6    \end{bmatrix}}_{1}
    \ar[r]^{\begin{pmatrix}     H\\0    \end{pmatrix}}
    &
    {\begin{bmatrix}    8\\8    \end{bmatrix}}_{2}
    \ar[r]^{\begin{pmatrix}    0&H\\0&0    \end{pmatrix}}
    &
    {\begin{bmatrix}    10\\10    \end{bmatrix}}_{3}
    \ar[r]^{\begin{pmatrix}   0&0 \\ 0&H \\0&0    \end{pmatrix}}
    &
}\]
\[
  \xymatrix@C=1.75cm{
    {\begin{bmatrix}    12\\12\\12    \end{bmatrix}}_{4}
     \ar[r]^{\begin{pmatrix}    H&0&0\\0&0&H    \end{pmatrix}}
    &
    {\begin{bmatrix}    14\\14    \end{bmatrix}}_{5}
     \ar[r]^{\begin{pmatrix}    0    \end{pmatrix}}
    &
    \begin{bmatrix}    16    \end{bmatrix}_{6}
    \ar[r]^{\begin{pmatrix}    H    \end{pmatrix}}
    &
    \begin{bmatrix}    18    \end{bmatrix}_{7}
    }
\]
\vspace{6mm}

$7_4 \eps{15mm}{7_4} $
\[
  \xymatrix@C=1.75cm{
    \begin{bmatrix} 2  \end{bmatrix}_{0}
    \ar[r]^{\begin{pmatrix}    0    \end{pmatrix}}
    &
    {\begin{bmatrix}     4\\4    \end{bmatrix}}_{1}
    \ar[r]^{\begin{pmatrix}     H&0\\0&0\\0&H    \end{pmatrix}}
    &
    {\begin{bmatrix}    6\\6\\6    \end{bmatrix}}_{2}
    \ar[r]^{\begin{pmatrix}    0&H&0\\0&0&0    \end{pmatrix}}
    &
    {\begin{bmatrix}    8\\8    \end{bmatrix}}_{3}
    \ar[r]^{\begin{pmatrix}   0&0 \\ 0&H \\0&0    \end{pmatrix}}
    &
}\]
\[
  \xymatrix@C=1.75cm{
    {\begin{bmatrix}    10\\10\\10    \end{bmatrix}}_{4}
     \ar[r]^{\begin{pmatrix}    H&0&0\\0&0&H    \end{pmatrix}}
    &
    {\begin{bmatrix}    12\\12    \end{bmatrix}}_{5}
     \ar[r]^{\begin{pmatrix}    0    \end{pmatrix}}
    &
    \begin{bmatrix}    14    \end{bmatrix}_{6}
    \ar[r]^{\begin{pmatrix}    H    \end{pmatrix}}
    &
    \begin{bmatrix}    16    \end{bmatrix}_{7}
    }
\]
\vspace{6mm}

$7_5 \eps{15mm}{7_5} $
\[
  \xymatrix@C=1.75cm{
    \begin{bmatrix} -18  \end{bmatrix}_{-7}
    \ar[r]^{\begin{pmatrix}    0\\H    \end{pmatrix}}
    &
    {\begin{bmatrix}     -16\\-16    \end{bmatrix}}_{-6}
    \ar[r]^{\begin{pmatrix}     H&0\\0&0\\0&0    \end{pmatrix}}
    &
    {\begin{bmatrix}    -14\\-14\\-14    \end{bmatrix}}_{-5}
    \ar[r]^{\begin{pmatrix}    0&H&0\\0&0&H\\0&0&0    \end{pmatrix}}
    &
    {\begin{bmatrix}    -12\\-12\\-12    \end{bmatrix}}_{-4}
    \ar[r]^{\begin{pmatrix}   0&0&0 \\ 0&0&H \\0&0&0    \end{pmatrix}}
    &
}\]
\[
  \xymatrix@C=1.75cm{
    {\begin{bmatrix}    -10\\-10\\-10    \end{bmatrix}}_{-3}
     \ar[r]^{\begin{pmatrix}    H&0&0\\0&0&0\\0&0&H    \end{pmatrix}}
    &
    {\begin{bmatrix}    -8\\-8\\-8    \end{bmatrix}}_{-2}
     \ar[r]^{\begin{pmatrix}    0&H&0    \end{pmatrix}}
    &
    \begin{bmatrix}   -6    \end{bmatrix}_{-1}
    \ar[r]^{\begin{pmatrix}    0    \end{pmatrix}}
    &
    \begin{bmatrix}    -4    \end{bmatrix}_{0}
    }
\]
\vspace{6mm}

$7_6 \eps{15mm}{7_6}$
\[
  \xymatrix@C=2.25cm{
    \begin{bmatrix} -12  \end{bmatrix}_{-5}
    \ar[r]^{\begin{pmatrix}    H\\0    \end{pmatrix}}
    &
    {\begin{bmatrix}     -10\\-10    \end{bmatrix}}_{-4}
    \ar[r]^{\begin{pmatrix}     0&H\\0&0\\0&0    \end{pmatrix}}
    &
    {\begin{bmatrix}    -8\\-8\\-8    \end{bmatrix}}_{-3}
    \ar[r]^{\begin{pmatrix}    0&0&0\\0&H&0\\0&0&H\\0&0&0    \end{pmatrix}}
    &
    {\begin{bmatrix}    -6\\-6\\-6\\-6    \end{bmatrix}}_{-2}
    \ar[r]^{\begin{pmatrix}   H&0&0&0 \\ 0&0&0&H \\0&0&0&0    \end{pmatrix}}
    &
}\]
\[
  \xymatrix@C=2cm{
    {\begin{bmatrix}    -4\\-4\\-4    \end{bmatrix}}_{-1}
     \ar[r]^{\begin{pmatrix}    0&0&0\\0&0&H\\0&0&0    \end{pmatrix}}
    &
    {\begin{bmatrix}    -2\\-2\\-2    \end{bmatrix}}_{0}
     \ar[r]^{\begin{pmatrix}    0&0&0\\0&H&0    \end{pmatrix}}
    &
    {\begin{bmatrix}   0\\0    \end{bmatrix}}_{1}
    \ar[r]^{\begin{pmatrix}    H&0    \end{pmatrix}}
    &
    \begin{bmatrix}    2    \end{bmatrix}_{2}
    }
\]
\vspace{6mm}

$7_7 \eps{15mm}{7_7}$
\[
  \xymatrix@C=2.25cm{
    \begin{bmatrix} -6  \end{bmatrix}_{-3}
    \ar[r]^{\begin{pmatrix}    H\\0\\0    \end{pmatrix}}
    &
    {\begin{bmatrix}     -4\\-4\\-4   \end{bmatrix}}_{-2}
    \ar[r]^{\begin{pmatrix}     0&H&0\\0&0&H\\0&0&0    \end{pmatrix}}
    &
    {\begin{bmatrix}    -2\\-2\\-2    \end{bmatrix}}_{-1}
    \ar[r]^{\begin{pmatrix}    0&0&0\\0&0&0\\0&0&H\\0&0&0    \end{pmatrix}}
    &
    {\begin{bmatrix}    0\\0\\0\\0    \end{bmatrix}}_{0}
    \ar[r]^{\begin{pmatrix}   H&0&0&0 \\ 0&0&0&0 \\0&0&0&H\\0&0&0&0    \end{pmatrix}}
    &
}\]
\[
  \xymatrix@C=2.25cm{
    {\begin{bmatrix}    2\\2\\2\\2    \end{bmatrix}}_{1}
     \ar[r]^{\begin{pmatrix}    0&H&0&0\\0&0&0&H\\0&0&0&0    \end{pmatrix}}
    &
    {\begin{bmatrix}    4\\4\\4    \end{bmatrix}}_{2}
     \ar[r]^{\begin{pmatrix}    0&0&0\\0&0&H    \end{pmatrix}}
    &
    {\begin{bmatrix}   6\\6    \end{bmatrix}}_{3}
    \ar[r]^{\begin{pmatrix}    H&0    \end{pmatrix}}
    &
    \begin{bmatrix}    8    \end{bmatrix}_{4}
    }
\]
\end{center}

\section{Specializations of the universal theory}

As we have seen, once the universal theory is computed there is a unified convenient way of specializing the theory to any specific choice of link homology theory. We called this specialization \emph{promotion} and described it in chapter \ref{chap:TQFT}. Basically, one takes the special line, promotes it to a module, and then takes the operator $H$ and promotes it to an appropriate operator (matrix). From chapter \ref{chap:TQFT} we have the following ``recipes'' of link homology theories:

\noindent $| \rightarrow \mathbb{Z}[h,t]\bigoplus\mathbb{Z}[h,t]$
\hspace{3mm} $H \rightarrow \left(
\begin{array}{cc} -h&2t \\ 2&h \end{array} \right)$ \hspace{3mm} $\simeq
\frac{\mathbb{Z}[X,h,t]}{X^2=hX+t}$ (Khovanov's ``big'' TQFT)
\\

\noindent $| \rightarrow \mathbb{Z}[t]\bigoplus\mathbb{Z}[t]$
\hspace{3mm} $H \rightarrow \left(
\begin{array}{cc} 0&2t \\ 2&0 \end{array} \right)$ \hspace{3mm}
$\simeq \frac{\mathbb{Z}[X,t]}{X^2=t}$ (Generalized Lee TQFT)
\\

\noindent $| \rightarrow \mathbb{Z}\bigoplus\mathbb{Z}$ \hspace{3mm}
$H \rightarrow \left(
\begin{array}{cc} 0&2 \\ 2&0 \end{array} \right)$ \hspace{3mm}
$\simeq \frac{\mathbb{Z}[X]}{X^2=1}$ (Lee's TQFT)
\\

\noindent $| \rightarrow \mathbb{Z}\bigoplus\mathbb{Z}$ \hspace{3mm}
$H \rightarrow \left( \begin{array}{cc} 0&0 \\ 2&0 \end{array}
\right)$ \hspace{3mm} $\simeq \frac{\mathbb{Z}[X]}{X^2=0}$ (Standard
TQFT)
\\

\noindent $| \rightarrow \mathbb{Z}$ \hspace{3mm} $H \rightarrow 0$
\hspace{3mm} $\simeq$ \emph{reduced} $\frac{\mathbb{Z}[X]}{X^2=0}$
\\

We want to explore the top case, and specialize the variables $h$ and $t$ to take some concrete numerical values. \textbf{Let us assume the ground ring is a field and that $h$ and $t$ take values in that field}. For example, $h=t=0$ is the ``standard'' Khovanov homology theory. Another example would be $h=0~~t=1$, this is Lee's theory \cite{lee1}.
\\

Here is an obvious proposition:
\begin{proposition}
If the ground ring is a field of characteristic 2 (say $\mathbb{Z}_2$) then any homology theory with $h=0$ is equivalent to the ``standard'' Khovanov homology theory (i.e. has the same chain complex).
\end{proposition}

\begin{proof}
No matter what $t$ is, all of these promotions are exactly the same promotions.

\qed
\end{proof}

Here is a less obvious generalization:
\begin{theorem}\label{promotionthm}
If the ground ring is a field and $h^2+4t=0$ then \emph{all} promotion of the above type are equivalent to the standard Khovanov homology theory (i.e. give rise to homotopic chain complexes). All homology theories with $h^2+4t\neq0$ are equivalent.
\end{theorem}

\begin{proof}
The expression $h^2+4t$ is (minus) the determinant of $H$ and thus determine its invertibility property. The characteristic polynomial of $H$ is $\lambda^2 - (h^2+4t)$ and thus this expression also determines the eigenvalues (and whether $H$ is diagonalizable). If $h^2+4t=0$ and the field is of characteristic $2$ then the theory is equivalent to the ``standard'' theory by the previous proposition. If $2$ is invertible then $H$ is neither invertible nor diagonalizable but still can be brought to its Jordan form. This form is $\frac{1}{2}$ times the ``standard'' promotion and thus give rise to homotopic complexes (via scaled change of basis). In the case $h^2+4t \neq 0$, when working over characteristic $2$, we immediately get $H$ invertible (and diagonal). When $2$ is invertible, we get that $H$ is invertible (and possibly diagonalizable if $h^2+4t=a^2$). All theories with $H$ invertible are equivalent. The structure of such theories will be explored in the next section and we refer the reader there for more details.
\qed
\end{proof}

\vspace{3mm}
\begin{remark}
In the case where $2$ is invertible, the universal theory can be reduced to a theory involving only the 2-handle operator $T$. When considering Khovanov's ``big'' universal TQFT, the operator $T$ is promoted to $h^2+4t$ times the identity matrix (just calculate $H^2$). If $h^2+4t=0$ then $T=0$ we reduce back to the ``standard'' TQFT. If $2$ is invertible and $h^2+4t=a \neq 0$ then we get a second type of theory (equivalent to Lee's theory). It does not matter what $a$ is equal to, the theories are invertible multiple of one another.
\end{remark}

\begin{remark}
These results re-produce and generalize results from \cite{turner1}. It is important to mention that if one wants to create a change of variables (and twists) between equivalent link homology theories with different $a$'s (say $a$ and $a'$) one has to demand that $\sqrt{\frac{a}{a'}}$ is an element of our ground field as well. This is a key component in the \emph{twist} of the Frobenius algebras \cite{turner1}. Thus it seems that one can find equivalent link homology theories such that the Frobenius algebras are not related by (at least an obvious) twist or a change of variables.
\end{remark}

\subsection{$H$ diagonalization and projections}
We now continue our discussion of the various homology theories ``projected'' from the our universal theory. As seen above, the expression $detH=-h^2-4t$ determines the type of promotion we get. Whenever $h^2+4t=a^2\neq 0$ the operator $H$ is diagonalizable (or already diagonal when the characteristic is $2$). This has implications on the structure of the complex and the homology groups.
\\

We remind the reader we work over a field. We say that the multiplication (or-co-multiplication) in a Frobenius algebra is \emph{diagonal} if there exists a basis to the algebra such that the multiplication table using that basis is diagonal and $a\times a$ is proportional to $a$ for all basis elements $a$. The following lemma applies in our Frobenius algebra:

\begin{lemma}
The 1-handle operator $H$ is diagonal (using a certain basis) iff the multiplication $m$ is diagonal (with the same basis) iff the co-multiplication $\Delta$ is diagonal.
\end{lemma}

Using this lemma we have:

\begin{proposition}
If the ground ring is a filed and $h^2+4t=a^2\neq 0$ then the complex is built from $2^n$ generators (where $n$ is the number of link components) and vanishing maps.
\end{proposition}

\begin{proof}
The condition for diagonalizing $H$ is $h^2+4t=a^2\neq 0$. The eigenvectors are $V_a=\left(\begin{array}{c} \frac{-(h-a)}{2} \\ 1 \end{array} \right) \hspace{3mm} V_{-a}=\left(\begin{array}{c} \frac{-(h+a)}{2} \\ 1 \end{array} \right)$. Note that if $2=0$, our operator $H$ is already diagonal using the standard basis ($1$ and $X$). After we diagonalize the operator $H$ the multiplication in the algebra is automatically diagonal as well. This allows for the arguments in \cite{lee1} and ~\cite{turner1} to be used, in order to prove that the complex is homotopic to one composed of $2^n$ generators and vanishing maps (where $n$ is the number of link components). We refer the reader to these papers for the proof. The process also allows to determine the \emph{homological} degrees of the surviving homology generators. \emph{In case of a knot it is always in degree zero}. Note that a local geometric prove exists and described in \cite{ba6}. Once we have the eigenvectors we can construct projection operators onto the eigenspaces. These projections are represented in the geometric category by certain cobordisms (maps). These projections are then used in order to extend the cobordism category into the Karoubi envelope (see \cite{ba6} for full details on this technique). The rest of the geometric proof can be found in \cite{ba6}.

\qed
\end{proof}

In fact:

\begin{proposition} \label{prop:2^n}
If the ground ring is a field and $h^2+4t=a\neq0$ then the complex is built from $2^n$ generators (where $n$ is the number of link components).
\end{proposition}

\begin{proof}
Although when $h^2+4t=a\neq0$ the operator $H$ is not necessarily diagonalizable, it is invertible. Thus when $2$ is invertible, such a theory is equivalent to Lee's theory. We know that for Lee's theory, when $2$ is invertible, the result hold (see \cite{lee1} for example). When $2=0$ the condition means that $H$ is diagonal and $\sqrt{a}$ is in the ground field, thus we are back to the case of the previous proposition..

\qed
\end{proof}

\vspace{3mm}

So far we have taken the universal theory and explored certain ``projections'' of (i.e. certain homology theories put on it). Using the results on various specializations we can get results regarding the structure of the universal theory (that is the geometric complex):

\begin{theorem}\label{ras}
The universal complex, in case of a knot, always has \emph{exactly} one complex component which looks like $0\rightarrow|\rightarrow0$. This component is in homological degree $0$.
\end{theorem}

\begin{proof}
This theorem follows from proposition \ref{prop:2^n} using a technique we call ''turning on and off''. Given a knot, we know that there is a certain promotion that results in two generators at degree zero. If we turn that homology theory off (that is look at the underlying universal complex before applying the promotion), these generators could only have come from a complex component $0\rightarrow|\rightarrow0$ at homological degree zero. If there were more such components we could have turned homology theory on and get a contradiction to the above proposition (there would be more homology generators)

\qed
\end{proof}

We can use a similar ``turning on and off'' to prove the following:

\begin{theorem}\label{thm:nocoeef}
In the universal complex, among all possible irreducible sub-complexes of the form $0 \rightarrow | \xrightarrow{mH^n} | \rightarrow 0$, only $m=1$ exists.
\end{theorem}

\begin{proof}
If such an irreducible sub-complex appears, with $m\neq1$, we can always find a field such that this sub-complex reduces to $0 \rightarrow | \xrightarrow{0} | \rightarrow 0$. Turning on a $h^2+4t=a\neq0$ homology theory will contradict proposition \ref{prop:2^n}.

\qed
\end{proof}

\vspace{3mm}
\begin{remark}
Lee's theory is a good example of the above, where $h=0$ and $t=1$ (thus $a^2=4$). Note that the only obstruction to $H$ diagonalization is division by 2, thus it is not a surprising consequence that Lee's theory over $\mathbb{Z}$ has only powers of $2$ torsion groups, on top of the 2 generators in homological degree zero (for knots). For the theory with $h=1$ and $t=0$ there is no obstruction to diagonalization over $\mathbb{Z}$ and thus no torsion exists (see also \cite{mac}).
\end{remark}

\subsection{The Rasmussen invariant and spectral sequences}\label{section:rasinv}

Let us now continue and see how the universal theory embodies further tools and supply a unified simple way of looking at the various link homology theories (this section is for the expert reader). Lee's theory (over $\mathbb{Q}$), as described in \cite{lee1}, is a part of a spectral sequence that starts with the ``standard'' Khovanov homology complex and converges to Lee's homology theory (I will not define here what a spectral sequence is, for a very brief introduction the reader can refer to \cite{cho}). Based on that theory, Rasmussen came up with a knot invariant which gives a lower bound for the slice genus of a knot and a purely combinatorial proof of the Milnor conjecture \cite{ras1}. I would like now to relate the universal theory to the Rasmussen invariant and the spectral sequence.
\\

Although the Rasmussen invariant construction involves Lee's filtered theory, and in some way the spectral sequence, the invariant itself is relatively simple. It was observed already in \cite{kho1} that the invariant can be extracted by looking at the structure of Lee's theory at homological degree zero. Given our theorem \ref{ras} we have:

\begin{proposition}
The Rasmussen invariant $s$ is equal to the internal (vertical) degree of the special line in the unique sub-complex component $0\rightarrow|\rightarrow0$ of the universal complex.
\end{proposition}

If the reader is not familiar with the construction of the Rasmussen invariant, please take the above proposition as the \emph{definition} of the Rasmussen invariant.
\\

With this proposition at hand, computations of the Rasmussen invariant become simple. One computes the universal complex in full, then searches for the appropriate component at homological degree zero and extracts the invariant. For example, the Rasmussen invariant of $16n864894$ (the non-alternating knot number $864894$ with 16 crossings) is $0$.
\\

The invariant for this knot was unknown as reported in \cite{shuma2}. There, the \emph{width} of the ``standard'' homology theory (the number of diagonals that the homology groups lie on -- discussed later in this thesis) is used to determine the possible differentials in Lee's spectral sequence, for the above knot this method was indecisive. Later, using some ideas from \cite{kho3} \cite{kho1} the author of \cite{shuma2} was able to \emph{deduce} that the invariant is indeed $0$ \cite{shuma1}. The universal theory allows for a \emph{computation} of the Rasmussen invariant without the need of heuristics and deduction.
\\

Here is the relevant piece of the universal complex computed for the knot $16n864894$ (only the degree $0$ sub-complex is given). One can easily extract the $0\rightarrow|\rightarrow0$ sub-complex at homological degree $0$ and internal degree $0$ to calculate the Rasmussen invariant:

\[
  \xymatrix@C=5.5cm{
    {\begin{bmatrix} -4\\-4\\-6  \end{bmatrix}}_{-1}
    \ar[r]^{\begin{pmatrix}0&0&0\\0&0&0\\0&0&0\\H&0&0\\0&H&0\\0&0&0\\0&0&H\end{pmatrix}}
    &
    {\begin{bmatrix}     0\\-2\\-2\\-2\\-2\\-4\\-4    \end{bmatrix}}_{0}
    \ar[r]^{\begin{pmatrix} 0&H&0&0&0&0&0\\0&0&H&0&0&0&0\\0&0&0&0&0&H&0\end{pmatrix}}
    &
    {\begin{bmatrix}    0\\0\\-2    \end{bmatrix}}_{1} }\]

\vspace{3mm}
Spectral sequences made their first appearance in link homology theory via the work of Lee \cite{lee1}. It can be shown that her homology theory is the final page of a spectral sequence defined using her differential on the filtered Khovanov complex. The second page is Khovanov's ``standard'' link homology theory. The idea of creating a spectral sequence that converges to certain link homology theories was later used intensively by Turner \cite{turner2} \cite{turner3} \cite{turner4} for the theory over $\mathbb{Z}_2$ and for the ``standard'' Khovanov link homology theory. These papers present some computations of these sequences which allow for certain computations of the link homology theories (along with various phenomenological observations).
\\

One of the advantages of the universal theory is that we can compute Lee's theory in full. By putting $h=0$ and $t=1$ in the ``big'' TQFT promotion one gets Lee's theory, and loses the grading ($t$ had degree $-4$). By keeping $t$ as a variable, we can compute the generalized Lee theory while keeping the grading of the theory. Thus by keeping track of the 1-handle operator $H$ in the universal theory (and thus keeping track of degrees of $t$ in the generalized Lee theory) we can keep track of the differentials in the spectral sequence (coming from powers of $t$ in the generalized Lee theory) and compute the different pages of the spectral sequence. The universal theory supplies a way of keeping track of the lost grading even when we work over the filtered theory (keeping track of degrees of $t$ \emph{before} setting it to $1$). The spectral sequence thus can be computed in full from the universal complex.
\\

One of the issues that this type of computations addresses is the \emph{speed of convergence} of the spectral sequence. To be concrete, let us work with Lee's theory (using the generalized Lee promotion) taken over various ground fields ($\mathbb{Q}$ or $\mathbb{Z}_p$) or $\mathbb{Z}$. We compute the universal theory, adjust the ground ring and use the ``standard'' promotion to calculate the first page (this would be just the ``standard'' Khovanov homology we get when we set $t=0$). Each page consists of the homology of the previous page where the $n$-th differential consists of only the appearances of the $n$-th power of $t$ in the universal complex (which is then assigned $t=1$). Thus the first page is the ``standard'' Khovanov homology which comes about if we ignore all powers of $t$ (equivalent to setting $t=0$ in this promotion). The next page uses differentials coming from first power of $t$ (these appear in $H$ and $H^2$ in the universal complex). The page after that will use the differential of degree 8 originating from $t^2$ (appearing first in $H^3$ in the universal complex). And so on. This spectral sequence will converge eventually to Lee's theory (which except possible torsion groups, is degenerated to $2^n$ generators, as mentioned earlier). It might take more than 2 pages for this convergence to happen. The number of pages it takes is called the \emph{speed of convergence} of Lee's spectral sequence.
\\

We say that the sequence is \emph{converging fast} if Lee's spectral sequence is converging on the second page (right after the ``standard'' Khovanov homology page, note there are various page counting conventions). Otherwise we say it \emph{converges slow}. The literature on link homology spectral sequences is of small size \cite{turner1} \cite{turner2} \cite{turner3} \cite{turner4} \cite{shuma1} \cite{shuma2} and barely mention the issue of speed of convergence (see also \cite{ba3}). There has not been a systematic way to determine the speed of convergence, and only certain heuristics has been used in the literature. These are:

\begin{proposition}\label{speed}\cite{shuma1} \cite{shuma2} \cite{kho1}\\
1. If the width (the number of lines the non zero homology groups occupy in $(i,j)$ space) of the ``standard'' Khovanov homology theory for a knot is at most 3 then Lee's theory converges fast.
\\
2. If the dimension of the ``standard'' Khovanov homology groups differ from the ones of the \emph{reduced} ``standard'' Khovanov homology groups (see \cite{kho3}) by 1, Lee's theory converges fast.
\\
3. If convergence is fast over $\mathbb{Z}_p$ then the second page over $\mathbb{Z}$ will have no $p$-torsion. If this happens for all $p$ then convergence over $\mathbb{Z}$ is fast as well. Rephrasing, if convergence over $\mathbb{Z}$ is slow due to $p$-torsion then convergence is slow over $\mathbb{Z}_p$
\end{proposition}

\vspace{3mm}
The above are obviously not enough in order to determine speed of convergence for every knot. ``Fat'' torus knots (TorusKnot$[m,n]$ with $m,n\geq5$, say) will be thick enough for possible slow convergence over various fields or rings. The advantage of the universal theory is that it finally gives a way of \emph{calculating} directly the speed of convergence thus there is no need of non-organized heuristics. Let us finish this discussion with an interesting example, the torus knot TorusKnot$[6,5]$. When one computes the universal theory of TorusKnot$[6,5]$ the following irreducible sub-complex appears (it is called a \emph{diamond} and will be discussed soon):

 \[
  \xymatrix@C=2cm{
    {\begin{bmatrix} 36  \end{bmatrix}}_{12}
    \ar[r]^{\begin{pmatrix}    H^3\\3H^2    \end{pmatrix}}
    &
    {\begin{bmatrix}     42\\40    \end{bmatrix}}_{13}
    \ar[r]^{\begin{pmatrix}      -3&H    \end{pmatrix}}
    &
    {\begin{bmatrix}    42    \end{bmatrix}}_{14}
    }
\]


This sub-complex can be promoted to the various homology theories. For example, $H=\left(\begin{array}{cc} 0&0 \\ 2&0 \end{array} \right)$ gives us the ``standard'' Khovanov homology over the integers, and after some simplifications we get the following homology groups (rows are internal degrees and columns are homological degrees):
  \[
  \begin{array}{|c|c|c|c|}
  \hline&12&13&14\\
  \hline43&&&\mathbb{Z}_3\\
  \hline41&&\mathbb{Z}&\\
  \hline39&&\mathbb{Z}&\\
  \hline37&\mathbb{Z}&&\\
  \hline35&\mathbb{Z}&&\\
  \hline
  \end{array}\]

Now, let us proceed and make the following promotion to Lee's theory $H=\left(\begin{array}{cc} 0&2 \\ 2&0 \end{array} \right)$. After simplifications the homology theory  of the sub-complex calculates to:

  \[
  \begin{array}{|c|c|c|c|}
  \hline&12&13&14\\
  \hline43&&&\\
  \hline41&&\mathbb{Z}_4&\\
  \hline39&&\mathbb{Z}_4&\\
  \hline37&&&\\
  \hline35&&&\\
  \hline
  \end{array}\]

Speed of convergence (from the ``standard'' table to Lee's table) can be easily calculated using the universal theory, while keeping track of the degrees of the maps. The result is that the theory converge \emph{fast} over $\mathbb{Z}$. \\

The exact same calculation can be done over $\mathbb{Z}_3$. We get that Lee's theory has no homology groups coming from this sub-complex (as expected when $2$ is invertible) and the ``standard'' Khovanov theory has homology groups:

  \[
  \begin{array}{|c|c|c|c|}
  \hline&12&13&14\\
  \hline43&&\mathbb{Z}_3&\mathbb{Z}_3\\
  \hline41&&\mathbb{Z}_3&\\
  \hline39&&\mathbb{Z}_3&\\
  \hline37&\mathbb{Z}_3&&\\
  \hline35&\mathbb{Z}_3&&\\
  \hline
  \end{array}\]

This time, though, convergence is \emph{slow}. This gives us an example showing that the converse of part 3 of proposition \ref{speed} is false (slow convergence over $\mathbb{Z}_3$ but fast over $\mathbb{Z}$). The mirror image of our example, TorusKnot[6,-5], has slow convergence over $\mathbb{Z}$ \cite{shuma1}, and therefore slow convergence over $\mathbb{Z}_3$ (proposition \ref{speed}). Thus we also get an example for a knot with fast convergence over $\mathbb{Z}$ who's mirror has slow convergence over $\mathbb{Z}$ (this cannot happen over $\mathbb{Z}_p$ \cite{shuma1}).
\\

\section{Phenomenology of the universal complex and link homology}\label{section:pheno}
Some time after Khovanov homology was introduced (what is called the ``standard'' Khovanov homology over $\mathbb{Q}$ in this thesis) computations of the homology groups of various knots appeared. With it came the phenomenology of Khovanov homology. Statements and conjectures regarding the shape, pattern and inner structure of the homology groups were made. Most of these statements dealt with the notion of \emph{thickness} of the homology and some (much less) dealt with the torsion of the ``standard'' theory over $\mathbb{Z}$. Most of the phenomenology surrounding link homology can be referred to Bar-Natan \cite{ba1} \cite{ba5}, Shumakovitch \cite{shuma3} and Khovanov \cite{kho3}. It is important to realize, though, that some theoretical advancement was made, proving theorems regarding the structure of the homology groups. See \cite{stosic1} \cite{stosic2} for theoretical results regarding torus knots, \cite{lee1} for results regarding alternating knots and \cite{asaeda1} \cite{shuma3} for results regarding torsion groups. This section describes how the universal theory is used in order to explain and unify most of the phenomenological statements. We will see how the universal theory projects down to the various phenomena and how through the eyes of the universal complex many patterns are ``natural''. We hope to convince the reader that the universal complex is the right approach for stating and proving results regarding the various homological phenomena.

\subsection{Thickness, torsion and homology patterns}
We start by defining some notions that are used widely in knot homology phenomenology. It is accustomed to present the ``standard'' Khovanov homology theory over $\mathbb{Q}$ ($\mathbb{Q}[X]/X^2$ TQFT) via the dimensions of the homology groups $\calH^{i,j}$, were $i$ represents the homological degree (columns) and $j$ is the internal degree (rows). Here is an example of the knot $4_1$ (only non-trivial groups are marked) :

 \def\merge{\multicolumn{2}{c|}{}}
  \def\putknot{
    \multicolumn{2}{c|}{\smash{\raisebox{1.5mm}{$\eps{2cm}{4_1}$}}}
  }
\[
  \begin{tabular}{|c|c|c|c|c|c|} \hline
          & $-2$& $-1$&  0  &  1  &  2  \\ \hline
      5   &     &     &     &     &1\\ \hline
      3   &     &     &     &&\\ \hline
      1   &     &     &1&1&     \\ \hline
     $-1$ &     &1&1&  \merge   \\ \cline{1-4}
     $-3$ &&&&  \merge   \\ \cline{1-4}
     $-5$ &1&     && \putknot  \\ \cline{1-4}
     $<-5$&     &     &&  \merge   \\ \hline
  \end{tabular}
\]

Another example of such a table was given in \ref{section:algorithm} for the $3_1$ knot, and partial tables were shown in the previous section. Using this table we define a \emph{diagonal} to be a line $2i-j=b$ for some $b$. Homological \emph{thickness} is the number of diagonals the homology groups occupy (sometimes also denoted as \emph{width}). A knot is called \emph{thin} if it occupies only $2$ diagonals, otherwise it is called \emph{thick}. For example, every alternating knot is thin \cite{lee1}. Such tables are also used for theories over $\mathbb{Z}_p$, or even over $\mathbb{Z}$, where one distinguishes between copies of $\mathbb{Z}$ and possible torsion groups.
\\

There are some obvious patters, when one observes such tables for many knots, and these were given names (see \cite{shuma3} for a good summary):
\\
\[
\begin{tabular}{ccccc}
  $\begin{array}{|c|c|}\hline&\mathbb{Q}\\\hline&\\\hline\mathbb{Q}&\\\hline\end{array}$&
  $\begin{array}{|c|c|}\hline&\mathbb{Z}_2\\\hline\mathbb{Z}_2&\mathbb{Z}_2\\\hline\mathbb{Z}_2&\\\hline\end{array}$&
  $\begin{array}{|c|c|}\hline&\mathbb{Q}\\\hline&\mathbb{Q}\\\hline&\\\hline\end{array}$& $\begin{array}{|c|c|}\hline&\mathbb{Z}\\\hline&\mathbb{Z}_2\\\hline\mathbb{Z}&\\\hline\end{array}$&
  $\begin{array}{|c|c|}\hline&\mathbb{Z}\\\hline&>\mathbb{Z}_2\\\hline\mathbb{Z}&\\\hline\end{array}$
  \\
  $\mathbb{Q}$ knight & $\mathbb{Z}_2$ tetris & $\mathbb{Q}$  & $\mathbb{Z}_2$ torsion & excess torsion\\
  move   & piece  & pawn & knight & knight\\
\end{tabular}
\]

Due to theorem \ref{ras}, for example, one can see that the homological thickness of any knot is at least $2$, and that there is always one pawn at degree $0$. The other homology groups (over $\mathbb{Q}$) seem to be arranged in knight moves and it was conjectured that this is always the case (over $\mathbb{Q}$). When one takes torsion into account, the notion of thickness can be extended further. A knot is called \emph{torsion thin} if it contains only $\mathbb{Z}_2$-torsion knights (except the usual pawn at degree $0$). It is called \emph{torsion rich} if it contains excess torsion knights (possibly for various torsion groups) or other extra pieces of torsion (usually torsion pawns). There are some conjectures regarding these torsion patterns. For example it is was conjectured that every thin knot is torsion thin. It was also conjectured that a knot is torsion rich iff its reduced homology has torsion. For these and other conjectures (and some theorems!) see \cite{shuma3}.

\subsection{Lines and Diamonds}
The universal theory is calculated over $\mathbb{Z}[H]$. As opposed to complexes over the integers or over the polynomial ring $\mathbb{Q}[X]$, such a complex will not have a simple decomposition (i.e the types of irreducible sub-complexes are more complicated). By looking at the examples given above for knots of up to $7$ crossings one immediately observes that a common irreducible sub-complex is $0 \rightarrow | \xrightarrow{H} | \rightarrow 0$ (on top of the usual $0\rightarrow|\rightarrow0$ at degree $0$). It can be shown that over a field (say $\mathbb{Q}$) the complex always reduces to direct sum of $0 \rightarrow | \xrightarrow{H^n} | \rightarrow 0$ for various powers of $n$, see \cite{kho1} \cite{kho5}. We will call these sub-complexes \emph{lines of order $n$}. Over the integers, the case is much more complicated and there are no general decomposition results. The next level of complication that I have encountered in my computations, are the \emph{diamonds}.
\\

\begin{definition}\label{def:diamonds}
An irreducible sub-complex of the universal knot homology theory of the form:
$\eps{25mm}{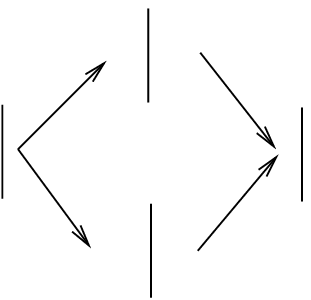}$ is called a \emph{diamond}. If the highest power of $H$ in the complex maps is $n$ we call it a diamond of order $n$.
\end{definition}

 An example for a diamond of order $3$ was given at the end of section \ref{section:rasinv} for the TorusKnot[6,5]:
 \[
  \xymatrix@C=2cm{
    {\begin{bmatrix} 36  \end{bmatrix}}_{12}
    \ar[r]^{\begin{pmatrix}    H^3\\3H^2    \end{pmatrix}}
    &
    {\begin{bmatrix}     42\\40    \end{bmatrix}}_{13}
    \ar[r]^{\begin{pmatrix}      -3&H    \end{pmatrix}}
    &
    {\begin{bmatrix}    42    \end{bmatrix}}_{14}
    }
\]

It is important to mention that so far I have not encountered in the universal complex lines of order higher than $2$ and diamonds of order higher than $3$. Diamonds appear in knots of high crossing numbers. It is also interesting to notice that once promoted, the universal theory's diamonds decompose into lines. We will call this process \emph{cutting}. So far I have not encountered powers of $H$ higher than $3$ in any homological structure, and I am not aware of any irreducible sub-complex which is more complicated than the diamonds. There is no reason to assume that these will not arise in higher crossing numbers.

\subsection{Unifying various phenomena under the universal theory}
We will now unify the various phenomena encountered so far in knot homology under the universal theory, and show how to take the higher structures of the universal theory (i.e. the diamonds) and broaden the phenomenological perspective on link homology.
\\

Let us start with a line of order $1$: $0 \rightarrow | \xrightarrow{H} | \rightarrow 0$. Promotion of this line to the standard homology over $\mathbb{Q}$ gives the knight move. On the other hand, promoting to the standard homology theory over $\mathbb{Z}$ gives us the $\mathbb{Z}_2$-torsion knight. Working over $\mathbb{Z}_2$ and promoting to the standard homology gives us the tetris piece. Thus we have:
 \\

 \noindent \textbf{Observation 1:}
 \emph{
  The knight move, the $\mathbb{Z}_2$-torsion knight move and the tetris piece are all the manifestation of a line of order 1 in the universal theory.}
\\

One homology theory that has been studied is the $h=1$ $\mathbb{Z}_2$ filtered theory. These theories are presented in filtered tables (the $j$ index is replaced by a filtered $j$ index -- $j \leq b$ for various $b$) \cite{ba1}. It has been noticed that $\mathbb{Z}_2$ pawns appear often in this theory where tetris pieces used to be in the standard $\mathbb{Z}_2$ theory, but no explanation was given (see also \cite{turner2} for spectral sequence techniques for that theory). The universal theory easily explains such a phenomena. Again, we look at a line of order 1 and promote it over $\mathbb{Z}_2$. The standard Khovanov homology theory now gives $H=0$ promotion. This results in the tetris piece as discussed earlier. The $h=1$ promotion results in $H$ equals the identity matrix. Thus the tetris piece disappears completely. If we keep track of degrees while \emph{restricting} to the filtered table we can see that the tetris piece will leave a finger print in the shape of the $\mathbb{Z}_2$ pawn.
\\

 \noindent \textbf{Observation 2:}
 \emph{
  The $\mathbb{Z}_2$ pawns of the $h=1$ $\mathbb{Z}_2$ theory are also the manifestation of a line of order 1 in the universal theory. These are actually leftovers of the tetris piece due to ignoring the degrees and using the filtered theory table.}
\\

A line of order 2 appears first in the knot $8_{19}$ (the first thick knot). The standard table over $\mathbb{Z}$ is given by:
 \def\merge{\multicolumn{2}{c|}{}}
  \def\putknot{
    \multicolumn{2}{c|}{\smash{\raisebox{1.5mm}{$\eps{2cm}{8_19}$}}}
  }
\[
  \begin{tabular}{|c|c|c|c|c|c|c|} \hline
          & 0& 1&  2  &  3  &  4&5  \\ \hline
      17   &     &     &     &     &&$\mathbb{Z}$\\ \hline
      15   &&     &     &     &&$\mathbb{Z}$\\ \hline
      13   &     &     &&$\mathbb{Z}$&$\mathbb{Z}$&     \\ \hline
     11 &     &&&$\mathbb{Z}_2$&$\mathbb{Z}$ &   \\ \hline
     9 &&&$\mathbb{Z}$&&  \merge   \\ \cline{1-5}
     7 &$\mathbb{Z}$&&&& \putknot  \\ \cline{1-5}
     5 &$\mathbb{Z}$&&&&  \merge   \\ \hline
  \end{tabular}
\]

It is easy to recognize the effect of the line of order 2: $\xymatrix@C=2cm{
    {\begin{bmatrix} 12  \end{bmatrix}}_{4}
    \ar[r]^{\begin{pmatrix}    H^2    \end{pmatrix}}
    &
    {\begin{bmatrix}     16    \end{bmatrix}}_{3}
    }
$
on the table. It looks as if there are two knight moves one on top of the other, and thus it might be surprising that these knight moves does not have $\mathbb{Z}_2$ torsion in them (which makes the knot torsion thick). But in fact these are \emph{not} 2 separate knight moves coming from two lines of order 1, but one line of order 2.
\\

 \noindent \textbf{Observation 3:}
 \emph{
             Torsion thickness arises from lines of order 2.}
\\

It seems that almost all the standard phenomenological observations done until today can be interpreted using the universal theory's lines of order 1 and 2. We wish to discuss excess torsion. Lines with coefficients does not appear on their own (theorem \ref{thm:nocoeef}) but they do appear as a part of a diamond. Various promotions cut these diamonds and produce excess torsion knights.
\\

 \noindent \textbf{Observation 4:}
 \emph{
             Excess torsion should be studied via the structure of diamonds and should be viewed as a ``wide'' phenomena (possibly occurring over 3 columns of the table).}
\\

We will demonstrate this principle by looking at two examples. First, let us look at how diamonds can be cut into excess torsion that spreads over $3$ columns. The torus knot TorusKnot$[4,5]$ has the following diamond in its universal theory:
  \[
  \xymatrix@C=2cm{
    {\begin{bmatrix} 24  \end{bmatrix}}_{8}
    \ar[r]^{\begin{pmatrix}    H^2\\2H    \end{pmatrix}}
    &
    {\begin{bmatrix}     28\\26    \end{bmatrix}}_{9}
    \ar[r]^{\begin{pmatrix}      -2&H    \end{pmatrix}}
    &
    {\begin{bmatrix}    28    \end{bmatrix}}_{10}
    }
\]

After promoting this diamond to the ``standard'' homology over $\mathbb{Z}$ the diamond is cut into two lines:

 $ \xymatrix@C=2cm{
    {\begin{bmatrix} 24  \end{bmatrix}}_{8}
    \ar[r]^{\begin{pmatrix}    0&0\\4&0    \end{pmatrix}}
    &
    {\begin{bmatrix}     26    \end{bmatrix}}_{9}
    }$ and
    $\xymatrix@C=2cm{
    {\begin{bmatrix} 28  \end{bmatrix}}_{9}
    \ar[r]^{\begin{pmatrix}    2&0\\0&2    \end{pmatrix}}
    &
    {\begin{bmatrix}     28    \end{bmatrix}}_{10}
    }$
These two pieces give rise to an excess torsion knight and a torsion pawn respectively. Put together the diamond is reflected through the following piece of standard table:
  \[
  \begin{array}{|c|c|c|c|}
  \hline&8&9&10\\
  \hline29&&&\mathbb{Z}_2\\
  \hline27&&\mathbb{Z}&\mathbb{Z}_2\\
  \hline25&&\mathbb{Z}_4&\\
  \hline23&\mathbb{Z}&&\\
  \hline
  \end{array}\]

As a second example we will see how diamonds might be cut down to torsion tetris pieces. Let us look at the knot $13n3663$ (the first torsion rich knot according to \cite{shuma3}). It has the following piece of table (where we dropped an irrelevant piece at $(-2,-5)$):
  \[
  \begin{array}{|c|c|c|c|}
  \hline&-4&-3&-2\\
  \hline-1&&&\mathbb{Z}\\
  \hline-3&&&\mathbb{Z}_2+\mathbb{Z}_2\\
  \hline-5&&\mathbb{Z}+\mathbb{Z}_2&\mathbb{Z}_2\\
  \hline-7&&\mathbb{Z}_2&\\
  \hline
  \end{array}\]

\vspace{3mm}
The universal complex has the following diamond in it:
  \[
  \xymatrix@C=2cm{
    {\begin{bmatrix} -6  \end{bmatrix}}_{-4}
    \ar[r]^{\begin{pmatrix}    -H\\-2    \end{pmatrix}}
    &
    {\begin{bmatrix}     -4\\-6    \end{bmatrix}}_{-3}
    \ar[r]^{\begin{pmatrix}      -2&H    \end{pmatrix}}
    &
    {\begin{bmatrix}    -4    \end{bmatrix}}_{-2}
    }
\]
together with an extra line  $\xymatrix@C=2cm{
    {\begin{bmatrix} -4  \end{bmatrix}}_{-3}
    \ar[r]^{\begin{pmatrix}    H    \end{pmatrix}}
    &
    {\begin{bmatrix}     -2    \end{bmatrix}}_{-2}
    }$. After promoting, the line of order one is responsible for a $\mathbb{Z}_2$-torsion knight, while the diamond is nicely cut into a $\mathbb{Z}_2$-torsion tetris piece. These two complexes overlap to create the above piece of table.
\\

 \noindent \textbf{Observation 5:}
 \emph{
                              Diamonds are cut in many ways. Some produce torsion tetris pieces and some produce a ``wide'' phenomena which is composed of torsion knights and torsion pawns combined.
                              }
\\

We have seen how the phenomenological observations done so far in the literature fit well into the universal theory, and how the universal theory unifies and explains the various phenomena which seemed to be non related. The universal theory can lead to proofs of some of the conjectures stated regarding the patterns of various link homology theories, and definitely improve the language in which link homology phenomenology is phrased and understood. higher objects (like diamonds) coming from the universal theory project down to various homology tables, creating rich phenomena.

\subsection{Phenomenology summary table}
Here is a quick summary table for the various phenomena discussed in section \ref{section:pheno}. On the left column are the simplest irreducible sub-complexes of the universal complex. The various rows give the projection of these sub-complexes on the homology table for the different link homology theories. Detailed description of all these phenomena can be found in the previous sections.

\[
  \begin{tabular}{|c|c|c|c|c|c|c|} \hline
  &standard  & standard & Lee  & standard & $\mathbb{Z}_2~h=1$& Lee  \\
  &$\mathbb{Q}$ &$\mathbb{Z}$&$\mathbb{Z}$&$\mathbb{Z}_2$&\emph{filtered} table& $\mathbb{Q}$\\ \hline
      $0 \rightarrow | \xrightarrow{H} | \rightarrow 0$
      &$\begin{array}{|c|c|}\hline&\mathbb{Q}\\\hline&\\\hline\mathbb{Q}&\\\hline\end{array}$
      &$\begin{array}{|c|c|}\hline&\mathbb{Z}\\\hline&\mathbb{Z}_2\\\hline\mathbb{Z}&\\\hline\end{array}$
      &$\begin{array}{|c|c|}\hline&\mathbb{Z}_2\\\hline&\mathbb{Z}_2\\\hline&\\\hline\end{array}$ &$\begin{array}{|c|c|}\hline&\mathbb{Z}_2\\\hline\mathbb{Z}_2&\mathbb{Z}_2\\\hline\mathbb{Z}_2&\\\hline\end{array}$
      &$\begin{array}{|c|c|}\hline&\mathbb{Z}_2\\\hline&\mathbb{Z}_2\\\hline&\\\hline\end{array}$
      &0
      \\ \hline
      $0 \rightarrow | \xrightarrow{H^2} | \rightarrow 0$
      &$\begin{array}{|c|c|}\hline&\mathbb{Q}\\\hline&\mathbb{Q}\\\hline\mathbb{Q}&\\\hline\mathbb{Q}&\\\hline\end{array}$
      &$\begin{array}{|c|c|}\hline&\mathbb{Z}\\\hline&\mathbb{Z}\\\hline\mathbb{Z}&\\\hline\mathbb{Z}&\\\hline\end{array}$
      &$\begin{array}{|c|c|}\hline&\mathbb{Z}_4\\\hline&\mathbb{Z}_4\\\hline&\\\hline&\\\hline\end{array}$
      &$\begin{array}{|c|c|}\hline&\mathbb{Z}_2\\\hline&\mathbb{Z}_2\\\hline\mathbb{Z}_2&\\\hline\mathbb{Z}_2&\\\hline\end{array}$
      &$\begin{array}{|c|c|}\hline&\mathbb{Z}_2\\\hline&\mathbb{Z}_2+\mathbb{Z}_2\\\hline&\mathbb{Z}_2\\\hline&\\\hline\end{array}$
      &0\\ \hline
      $0 \rightarrow | \rightarrow 0$
      &\multicolumn{4}{c|}{degree $0$ pawn}
      &$\begin{array}{|c|c|}\hline&\mathbb{Z}_2\\\hline&\mathbb{Z}_2+\mathbb{Z}_2\\\hline&\mathbb{Z}_2+\mathbb{Z}_2\\\hline&\vdots\\\hline\end{array}$
      &pawn\\ \hline
      $\eps{25mm}{diamond}$
      &$\begin{array}{c}standard\\lines\end{array}$
      &$\begin{array}{c}torsion:\\``wide''\\excess\\knights\\pawns\end{array}$
      &$\begin{array}{c}2^n\\torsion\end{array}$
      &\multicolumn{2}{c|}{$\begin{array}{c}standard\\lines\end{array}$}
      &0\\ \hline

  \end{tabular}
\]

\vspace{5mm}

\section{Extensions of the universal theory}
This section discusses \emph{shortly} how the techniques and methods I described in this thesis can be used in various extensions of the ``standard'' Khovanov homology theory.
\\

$\bullet$ \textbf{$sl_3$ link homology} --
As soon as the ``standard'' ($sl_2$) link homology theory was published \cite{kho2} it was clear that the theory should be extended to other knot polynomials coming from other quantum groups. The first extension was published by Khovanov in \cite{kho4}, and extended the theory to $sl_3$. It took a while, but eventually this extension was also cast into a certain geometrical formalism by Mackaay and Vaz \cite{mac}. Their construction, although relays on the geometrical complex, does not use the full strength and universality of the complex as demonstrated in this thesis. More recently, Morrison and Nieh \cite{sco1} produced a similar cobordism theory for the $sl_3$ link homology theory which try to use the full strength of the geometric formalism. For that purpose they applied the surface reduction technique and the delooping process described in this thesis (which was published in\cite{naot1}). They did not work over $\mathbb{Z}$ thus restricting some of the generality, and they did not query regarding universal TQFTs and other possible homology theories that can be put on the geometric complex, thus leaving room for further applications of this thesis to $sl_3$ link homology.
\\

$\bullet$ \textbf{Unoriented TQFTs} --
It is clear that the geometric formalism can be applied directly to knots on (thick) surfaces. The construction remains the same, the proofs do not need adjustments, the relations needed are the same and the only difference is that the basic objects are now circles on surfaces. This difference forces us to consider unoriented surfaces in our theory, thus extending the category. In \cite{tur1} Turaev and Turner extend the geometric formalism to links and knots on a surface (see \cite{man1} for similar extension to virtual knots). They construct an appropriate geometric category and apply various TQFTs (unoriented versions) to the geometric complex in order to get homology theory. The techniques and methods in this thesis can be applied to answer universality questions.
\\

$\bullet$ \textbf{Open-Closed TQFTs} --
Another extension of the geometric formalism that may benefit from the methods described in this thesis is the open-closed TQFT extension for tangles. In \cite{lau1} it is shown how to use open-closed TQFTs (knowledgable Frobenius algebras) in order to go from the geometric complex associated to a tangle into an algebraic one (to compute homology) while still preserving the composition properties of the geometric complex (coming from tangle composition). This setting can benefit much from the methods of this thesis when trying to understand the relations between the algebraic part of the open-closed TQFTs and the topological description. Again, questions of universality and information extraction should be addressed in this case as well.

%% file: ChapComments.tex
\chapter{Further comments}\label{chap:comments}

\section{Homotopy classes vs. Homology groups}
The geometric complex associated to a link is an invariant up to homotopy of complexes. Thus its fullest strength lies in the homotopy class of the complex itself, and have the potential of being a stronger invariant than any homology theory applied to it. The reduction theorem given in this thesis might be the beginning of an approach to the following question: \emph{classify all complexes associated to links up to homotopy}. The category $\Komh\Mat(\Cobl)$ seemed at first too big and complicated for an answer, but complexes built on modules over polynomial rings look more hopeful. The complex is built from a category equivalent to a category of $\mathbb{Z}[H]$-modules.
Thus, it has a faithful algebraic representation, i.e. we have found a homology theory that represents faithfully the complexes in $\Komh\Mat(\Cobl)$. The geometric complex holds the same amount of information as the chain complex of the universal TQFT. Thus the question is: Classify all homotopy types of chain complexes with $\mathbb{Z}[H]$-modules as chain groups and $\mathbb{Z}[H]$ matrices as maps (monomial entries w.l.o.g.) arising from links. Since homology groups are relatively easy to calculate we are also interested to know whether the possible homology groups of the complexes associated to links determine the homotopy class of the complexes. The answers will determine the strength of the complex invariant (relative to its homology groups). It seems that the answers to these questions are not as simple as in the case of working over a field \cite{kho5}. One evidence for the complexity of these questions are the diamonds discussed in chapter \ref{chap:CompExt}.
\\

\section{Comments on marking the link, 1-1 tangles and functoriality} \label{section:commentmarkingpoint}

A small technical issue that should be addressed is the issue of marking the link at one point to get the 1-1 tangle presentation, as mentioned in the beginning of chapter \ref{chap:ComplexReductionQ}. When dealing with knots this choice of a point to mark and cut the knot open does not affect the result -- ``long knot'' theory and knot theory are the same theories. When we deal with links, one might choose different components of the link to place the mark but these choices still give isomorphic complexes (they
might be presented differently though). This is true because of the fact that the complex reduction is local, and thus in every appearance of an object of $\Cobl$ in the complex one can choose the special circle \emph{independently} and apply the complex
isomorphisms. Different choices are linked through a series of complex isomorphisms.
So all together we get that the choices made when marking the link by a point (in order to cut it open to a 1-1 tangle) do not influence the final result, and the complexes we get for various choices are isomorphic. With this said, one has to be a little careful when discussing the functoriality of link homology theory with respect to link cobordisms (much like the remark \ref{remark:Hlinear}). It seems that if one wishes to maintain this functoriality property (one of the flag properties of link homology) one needs to restrict discussion to 1-1 tangles only and 1-1 tangles cobordism , i.e cobordisms which take the special line to the special line in a connected way, thus action of $H$ is well behaved.

\section{Comments on embedded vs. abstract surfaces} As
noticed in section 11 of ~\cite{ba1}, when looking at surfaces in
$\Cobl$ over a ground ring with the number 2 invertible, there is no
difference between embedded surfaces (inside a cylinder
say) and abstract surfaces. This is true due to the fact that
any knotting of the surface can be undone by cutting necks and
pulling tubes to unknot the surface. In other words, by cutting and
gluing back, using the NC relation (divided by 2) both ways, one can
go from any knotted surface to the unknotted version of it embedded
in 3 dimensional space. Our claim is that the same is true even when
2 is not invertible. The proof is a similar argument applied to any
knotted surface using the $3S1$ relation:
\\

\begin{center}
$\begin{array}{ccccccc} \eps{20mm}{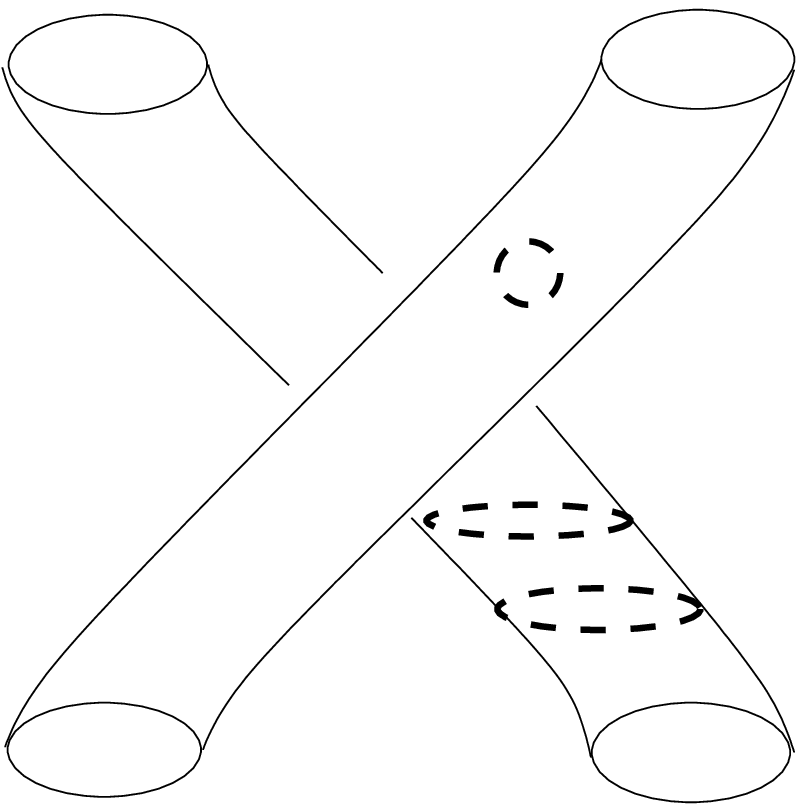} & = & \eps{20mm}{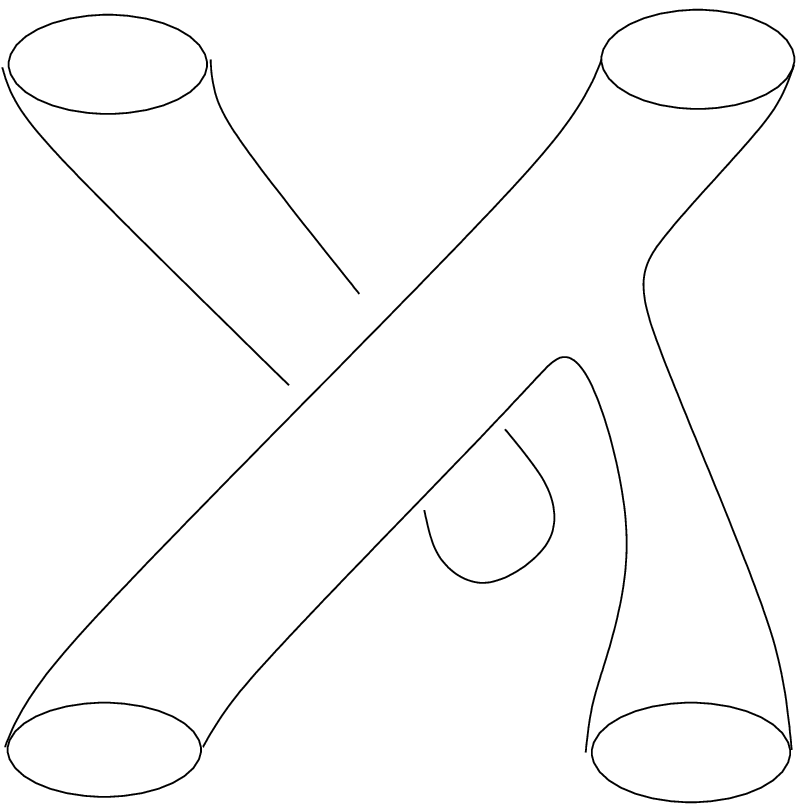} & +
& \eps{20mm}{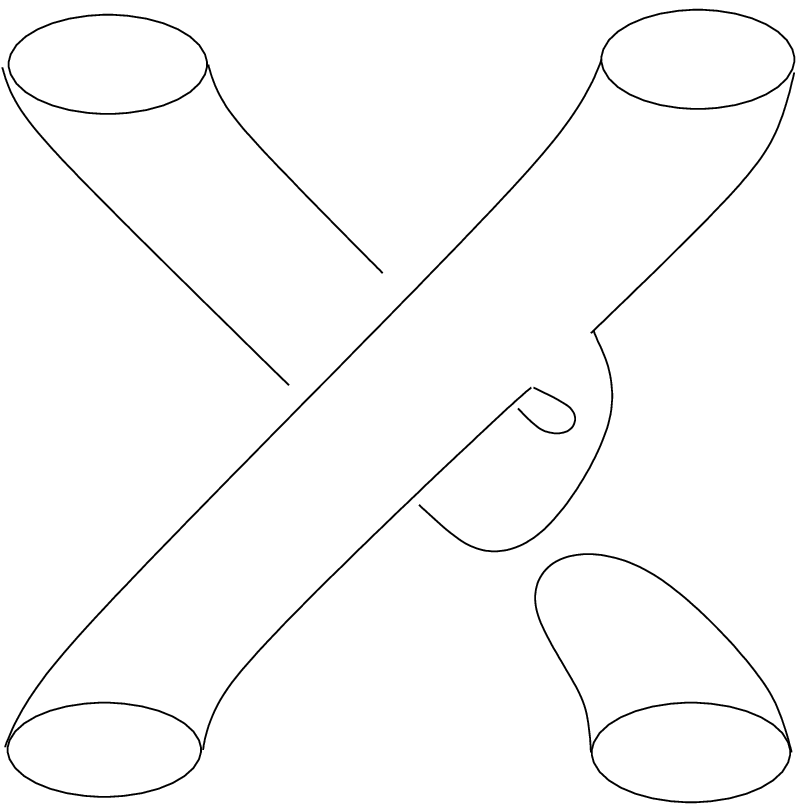} & - & \eps{20mm}{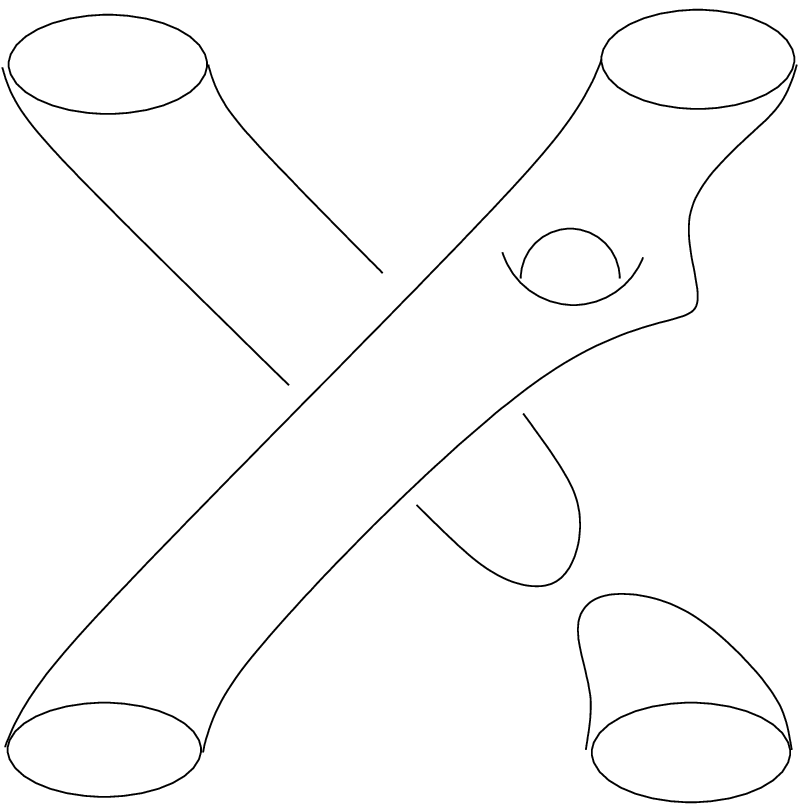} \\
 & & \updownarrow & & \updownarrow & & \updownarrow\\
 \eps{20mm}{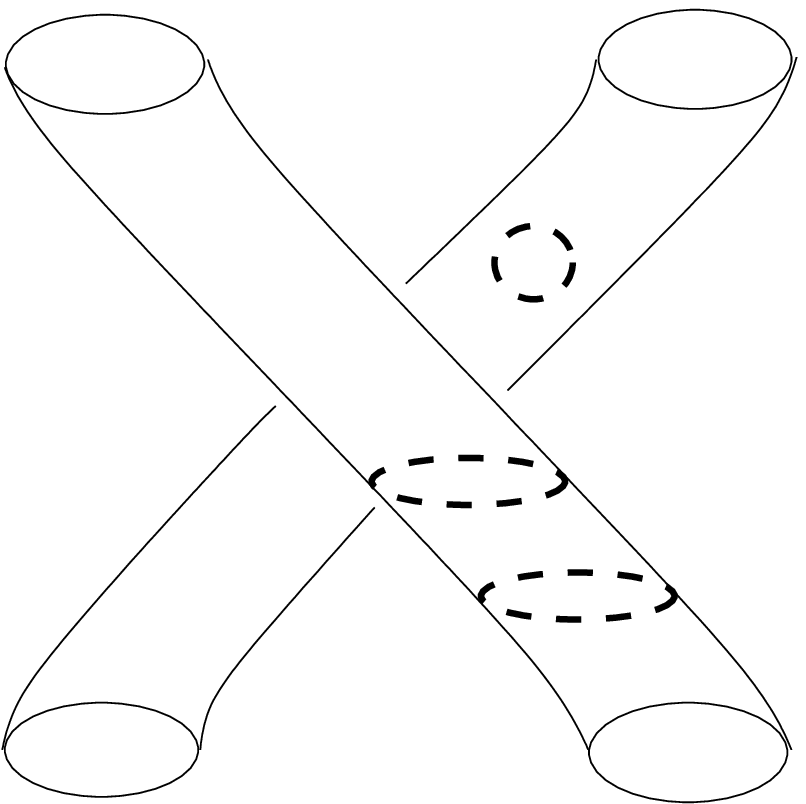} & = & \eps{20mm}{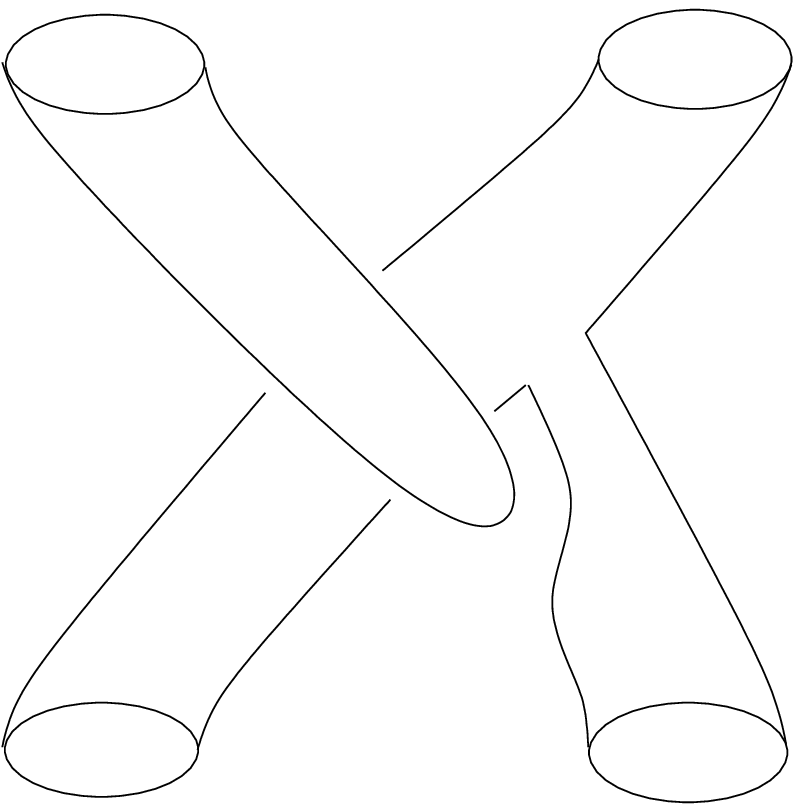} & + & \eps{20mm}{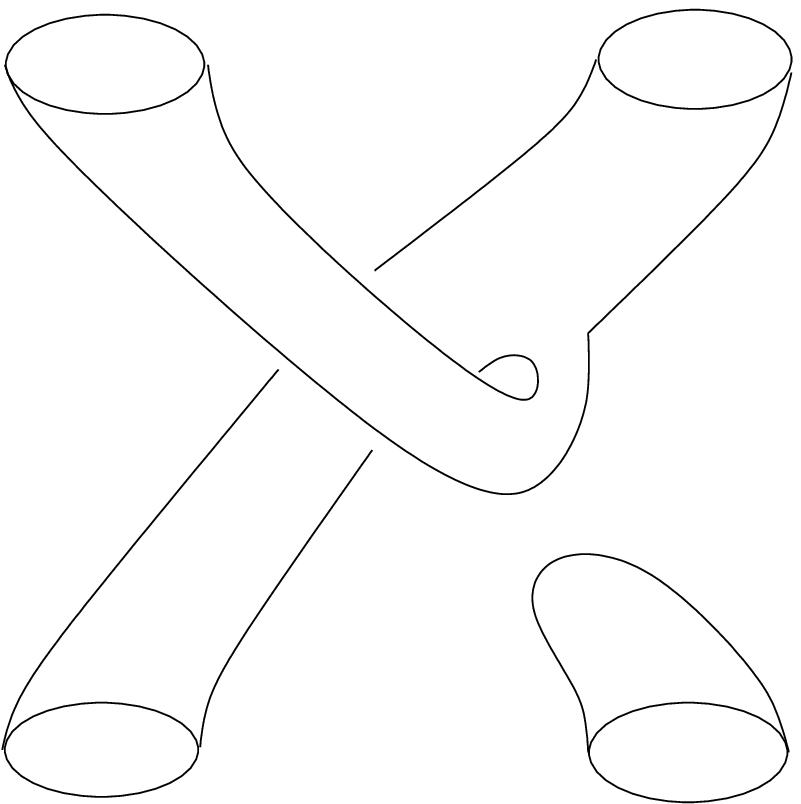} &
 - & \eps{20mm}{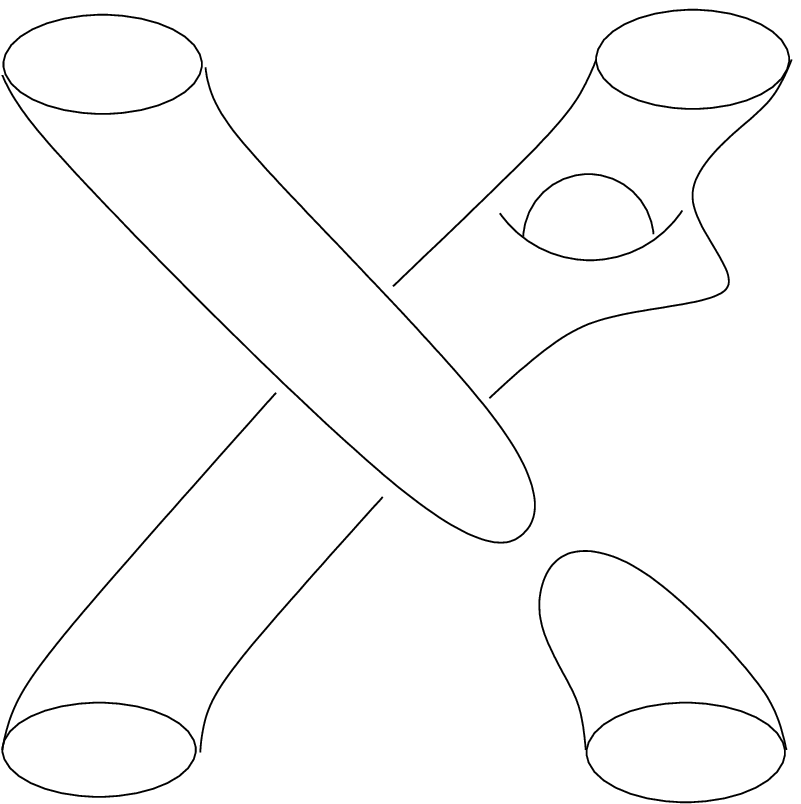}
 \end{array}$
\end{center}

The above picture shows that every crossing (a part of a knotted
surface embedded in 3 dimensions) can be flipped using the $3S1$
relation twice. Apply the $3S1$ relation once on the dashed sites
(going from top left), then smoothly change the surface (going down
the arrows) and finally use the dashed sites for another application
of the $3S1$ relation (reaching the bottom left). Every embedded
surface can be unknotted this way.

\section{FAQ: ``I want to compute the universal complex too! How do I reproduce your computations?''}\label{section:software}
The universal complex is free, easy and fun to compute! All you need is a Java compiler (to run Green's program \cite{green}), Mathematica program (if you want to use the convenient KnotTheory' software package) and temporary access to the internet. Start by going to The Knot Atlas (www.katlas.math.toronto.edu). This is the knot atlas where anyone can edit! Click on the KnotThoery' manual and read the instructions. You will be able to download the entire package (which operates on Mathematica). The package includes Green's JavaKh java software (which can be run outside of Mathematica). Browsing the KnotTheory' manual to the Khovanov homology section you will be able to learn how to use it, via some examples, in order to calculate the standard Khovanov homology over various rings and fields (note that a common reason for failing at this stage is improper path definition).
\\

In order to compute the universal theory, presented in this thesis, you need to take an extra step. The program JavaKh computes the universal complex but the output is not convenient for direct reading, thus the package KnotTheory' does not handle this computation yet. In order to translate the output into something more comprehensive you will need a certain Mathematica code that was written by Bar-Natan for this purpose. It can be easily downloaded from \cite{ba2} (you need to copy the code to your Mathematica notebook).
\\

Finally, your Mathematica notebook should look like this:
\\
First, set your path and load the KnotTheory' package using the following commands:
\begin{verbatim}
AppendTo[$Path, "c:\path"];
AppendTo[$Path, "c:\path\KnotTheory`"];
<< KnotTheory`
\end{verbatim}

Then, execute the piece of code (function definition) that you downloaded separately (in order to read the output of JavaKh) :

\begin{verbatim}
KhN[L_] := KhN[PD[L]];
KhN[pd_PD] := Module[
    {n, dir, f, cl, out},
    n = Max @@ (Max @@@ pd);
    pd1 = pd /. {
          X[n, i_, 1, j_] :> X[n, i, n + 1, j],
          X[i_, 1, j_, n] :> X[i, n + 1, j, n],
          X[1, j_, n, i_] :> X[n + 1, j, n, i],
          X[j_, n, i_, 1] :> X[j, n, i, n + 1]
          };
    f = OpenWrite["pd", PageWidth -> Infinity];
    WriteString[f, ToString[pd1]];
    Close[f];
    cl = StringJoin[
        "!java -classpath \"", ToFileName[KnotTheoryDirectory[], "JavaKh"],
        "\" ", " JavaKh -U", " < pd"];
    f = OpenRead[cl];
    out = Read[f, Expression];
    Close[f];
    out = StringReplace[out, {"q" -> "#1", "t" -> "#2"}];
    kh = ToExpression[out <> "&"][q, t];
    minr = Exponent[kh, t, Min];
    maxr = Exponent[kh, t, Max];
    obs = Expand[kh /. h -> 0 /. M[_, n_, ___]  :> Plus @@ Array[Arc, n]];
    obs = obs /. (q^j_.)*Arc[i_] :> Arc[j, i] /. Arc[i_] :> Arc[0, i];
    mos = Expand[
        h*kh /. {M[0, _] -> 0, M[_, 0] -> 0, h -> H}
          /. M[m_, n_, cs___] :> Plus @@ Flatten[MapIndexed[
                  (#1*Curtain @@ Reverse[#2]) &,
                  Partition[{cs}, n],
                  {2}
                  ]]
        ];
    mos =
      mos /. (q^j_.)*Curtain[k_, l_] :> Curtain[j, k, l] /.
        Curtain[k_, l_] :> Curtain[0, k, l];
    mos =
      mos /. (H^g_.)*Curtain[j_, k_, l_] :>
          H^(g - 1)Curtain[j, k, j + 2(g - 1), l];
    Table[{r, Coefficient[obs, t, r],  Coefficient[mos, t, r]}, {r, minr,
        maxr}]
    ]
\end{verbatim}

You are now ready to compute the universal complex! Let us compute an example. Here is the result you will see on your screen for the knot $6_3$ :
\begin{verbatim}
KhN[Knot[6, 3]]
\end{verbatim}
\noindent\{\{-3, Arc[-6, 1], H Curtain[-6, 1, -4, 1] + H Curtain[-6, 1, -4, 2]\},\\
\{-2, Arc[-4, 1] + Arc[-4, 2], -H Curtain[-4, 1, -2, 1] + H Curtain[-4, 2, -2, 1]\},\\ \{-1, Arc[-2, 1] + Arc[-2, 2], H Curtain[-2, 2, 0, 2]\},\\
\{0, Arc[0, 1] + Arc[0, 2] + Arc[0, 3], H Curtain[0, 3, 2, 1]\},\\
\{1, Arc[2, 1] + Arc[2, 2], -H Curtain[2, 2, 4, 2]\},\\
\{2, Arc[4, 1] + Arc[4, 2], H Curtain[4, 1, 6, 1]\},
\{3, Arc[6, 1], 0\}\}

We need to construct the complex from this output. This is pretty straight forward. Each homological degree is nested in $\{~\}$. The first number is the homological degree. Then the special lines are collected (denoted ``Arc[ ]''). Each special line comes with an internal degree (the first number) and a serial number (the second number). Then the maps are collected. These are denoted by ``Curtain[ ]''. They might be accompanied by a monomial in $H$, a sign, or a coefficient. Each curtain specifies the source arc and the target arc numbers. Thus the complex can be constructed precisely.

Note that the JavaKh software does not reduce the complex completely, and thus there are still some manual complex reductions that can be done. For example, the sub-complex from the above example:
\\
\noindent\{-3, Arc[-6, 1], H Curtain[-6, 1, -4, 1] + H Curtain[-6, 1, -4, 2]\},\\
\{-2, Arc[-4, 1] + Arc[-4, 2], -H Curtain[-4, 1, -2, 1] + H Curtain[-4, 2, -2, 1]\},\\ \{-1, Arc[-2, 1]\}
\\
is actually reducible to:\\
\noindent\{-3, Arc[-6, 1], H Curtain[-6, 1, -4, 1]\},\\
\{-2, Arc[-4, 1] + Arc[-4, 2], H Curtain[-4, 2, -2, 1]\},\\
 \{-1, Arc[-2, 1]\}
\\

After these manual reductions are done, the complex can be finally presented in the form we used in the thesis :
\[
  \xymatrix@C=1cm{
    \begin{bmatrix} -6  \end{bmatrix}_{-3}
    \ar[r]^{\begin{pmatrix}    H\\0    \end{pmatrix}}
    &
    {\begin{bmatrix}     -4\\-4    \end{bmatrix}}_{-2}
    \ar[r]^{\begin{pmatrix}     0&H\\0&0    \end{pmatrix}}
    &
    {\begin{bmatrix}    -2\\-2    \end{bmatrix}}_{-1}
    \ar[r]^{\begin{pmatrix}    0&0\\0&H\\0&0    \end{pmatrix}}
    &
    {\begin{bmatrix}    0\\0\\0    \end{bmatrix}}_{0}
    \ar[r]^{\begin{pmatrix}   \textit{0}&H \\ \textbf{0}&\hat{0}    \end{pmatrix}}
    &
    {\begin{bmatrix}    2\\2    \end{bmatrix}}_{1}
     \ar[r]^{\begin{pmatrix}    0&0\\0&H    \end{pmatrix}}
    &
    {\begin{bmatrix}    4\\4    \end{bmatrix}}_{2}
     \ar[r]^{\begin{pmatrix}    H&0    \end{pmatrix}}
    &
    \begin{bmatrix}    6    \end{bmatrix}_{3}
  }
\]
($\textbf{0},\textit{0},\hat{0}$ represent $2\times2,1\times2,2\times1$ zero matrices)